\documentclass[11pt,oneside,reqno]{article}
\usepackage[top=2.5cm, bottom=2.5cm, right=2.5cm, left=2.5cm]{geometry}
\usepackage[utf8]{inputenc}
\usepackage[T1]{fontenc}
\usepackage{lmodern}
\usepackage[UKenglish]{babel}
\usepackage{amsmath, mathtools, amsthm, amssymb}
\usepackage{changepage}
\usepackage{hyperref}
\usepackage{cite}
\usepackage{bbm}
\usepackage{xcolor}

\DeclareMathOperator*{\esssup}{ess\,sup}

\theoremstyle{plain}
\newtheorem{thm}{Theorem}[section] % reset theorem numbering for each section
\newtheorem{corollary}{Corollary}[thm]
\newtheorem{lemma}[thm]{Lemma}
\newtheorem{prop}[thm]{Proposition}

\theoremstyle{definition}
\newtheorem{defn}[thm]{Definition} % definition numbers are dependent on theorem numbers
\newtheorem{assumption}{Assumption}

\theoremstyle{remark}
\newtheorem{rmk}{Remark}
\newtheorem*{notation}{Notation}

\newcommand{\mathsc}[1]{\text{\textsc{#1}}}

\newcommand{\R}{\mathbb{R}}
\newcommand{\Prob}{\mathbb{P}}
\newcommand{\Q}{\mathbb{Q}}
\newcommand{\E}{\mathbb{E}}
\newcommand{\dx}{\, \mathrm{d}x}
\newcommand{\dy}{\, \mathrm{d}y}
\newcommand{\dz}{\, \mathrm{d}z}
\newcommand{\ds}{\, \mathrm{d}s}
\newcommand{\dr}{\, \mathrm{d}r}
\newcommand{\dt}{\, \mathrm{d}t}
\newcommand{\di}{\, \mathrm{d}}
\newcommand{\expect}[1]{\E \left[ #1 \right]}
\newcommand{\Vrho}[1]{\| \varrho \|_{\mathsc{tv}(0,#1)}}

\begin{document}
	
\title{A free boundary problem for the mean-field limit of diffusing particles with nonlinear boundary reactivity}
	
\author{Eliana Fausti\thanks{Imperial College London and Scuola Normale Superiore di Pisa, {\tt eliana.fausti@sns.it}.} \and Andreas S{\o}jmark\thanks{London School of Economics, {\tt a.sojmark@lse.ac.uk}.}}
	
\date{\today}
	
\maketitle

\begin{abstract}
	Consider a finite system of diffusing particles coupled through a reactive boundary. Each particle is reflected, but may react with the boundary according to a killing mechanism which depends on the current reactivity of the boundary and the particle's local time along it. With every such reaction, the boundary moves and its reactivity adjusts. We show that this system admits a unique mean-field limit, described by a free boundary problem with nonlinear and nonlocal reactivity. The latter generalises the classical Robin condition for the case of a fixed boundary with constant reactivity. Via Skorokhod's M1 topology and a characterisation of the particles' behaviour near the boundary, we first identify the weak limit points of the empirical measure flows with killing. Then, we combine a probabilistic decoupling technique and energy estimates to prove uniqueness and deduce convergence. Our analysis gives a rigorous mean-field description of a model of epidemic spreading. Moreover, it contributes to the literature on inert drift systems and yields a novel mean-field perspective on the recent encounter-based framework for diffusion-mediated surface reaction from [Phys.~Rev.~Lett.~\textbf{125}~(2020)~078102].
\end{abstract}
	
\begin{adjustwidth}{0.6cm}{0.6cm}
\begin{flushleft}
			
\end{flushleft}
\end{adjustwidth}

\section{Introduction}

We study a system of $n$ real-valued particles that are diffusing in a constrained space determined by a moving boundary. Each particle is reflected back into the domain, but it may eventually react with the boundary and is then moved to a cemetery state. Whenever a particle collides with the boundary, the probability of a reaction depends both on the corresponding local time and the current reactivity of the boundary, as specified by a suitable mean-field statistic. The latter, in turn, depends on the realised reactions, and it is also this that drives the moving boundary. The system was introduced in \cite{fausti_soj_22} as a model for epidemic spreading. In that case, the reactions correspond to infection and the moving boundary represents the total advance of the epidemic.

As the central contribution of this work, we provide a unique characterisation of the mean-field limit for the non-Markovian particle system described above. Furthermore, our analysis yields novel representation formulas and associated estimates for the densities of interacting diffusion processes with a reactive boundary that interpolates nonlinearly between reflection and absorption.

In the mean-field limit, we obtain a novel free boundary problem whose boundary advances at a rate proportional to the density of mean-field particles immediately adjacent to it with a factor of proportionality that depends on the position of the free boundary and how strongly it has recently advanced. The mean-field density obeys a Fokker--Planck equation subject to a reactive (or mixed) boundary condition which governs the aforementioned advance of the free boundary.

\subsection{A mean-field model for the spreading of an epidemic}\label{sect:heuristics_particle}

In the epidemic model of \cite{fausti_soj_22}, each particle is represented by its level of \emph{shielding} from a disease. That is, $X_t^i$ describes how shielded the $i$'th particle is from being exposed to that disease. Away from the reach of the disease, each $X_t^i$ evolves according to given diffusive dynamics
\begin{equation*}
\di X_t^i = b(t,X_t^i)\dt + \sigma(t,X_t^i)\di B_t^i,
\end{equation*}
where $B^1,\ldots,B^n$ are independent Brownian motions. The spreading of the disease is modelled by a moving boundary $A^n$, which we refer to as the \emph{advancing front} of the epidemic. Initially, each particle $X^i$ reflects off the front $A^n$; however, the particles are at risk of infection whenever they collide with it. When a particle is at-risk in this way, it may become infected. The conditional probability of this occurring decomposes into (i) the extent to which the particle is at-risk, i.e., how extensively it is colliding with the front, and (ii) how likely it is, currently, that being at-risk leads to infection. We shall now specify the front and then the precise infection mechanism.

The advancing front $A^n$ grows continuously in proportion to the cumulative number of infected particles. This advance is modulated by an \emph{infection-to-recovery} kernel $\varrho$ that captures the way in which a newly infected particle gradually becomes more infectious before then starting to recover, hence ceasing to expand the total reach of the disease. Specifically,
\begin{equation}\label{eq:advancing_front_intro}
A^n_t = a_0 + \alpha \int_{t-\bar{d}}^t \varrho(t-s) I_s^n\ds\quad \text{with}\quad I^n_t=\frac{1}{n}\sum_{i=1}^n \mathbf{1}_{[0,t]}(\tau^i),
\end{equation}
where $\tau^i$ is the infection time of the $i$'th particle. We refer to $I^n$ as the \textit{infected proportion} and $\bar{d}$ as the \textit{duration of infectiousness}. Heuristically, $\tau^i$ is specified by the conditional probabilities
\begin{equation}\label{eq:elastic_infection_intro}
\Prob\bigl(\tau^i \in(t,t+h] \mid \ell^{i},C^n, \{t<\tau^i\} \bigr) \approx \int_t^{t+h} \gamma(s,C^n_s) \di \ell^i_s,
\end{equation}
for small $h>0$, at any time $t$, where $\ell^i$ is the local time of the collisions of $X^i$ with the front $A^n$. This expression decomposes the probability of infection into (i) the extent to which the $i$'th particle is at-risk, as measured by the local time increment $\di \ell^{i}_s$, and (ii) the effective `rate' of infection $\gamma(s,C_s^n)$ given that the particle is at-risk. Importantly, $\gamma$ depends not only on inherent properties of the disease, but also the \emph{current contagiousness} within the population $C^n$, defined by
\begin{equation}\label{eq:current_contagiousness_intro}
C^n_t = \int_{t-\bar{d}}^t \varrho(t-s) (I^n_s- I^n_{s-\bar{d}}) \ds = \frac{1}{\alpha}( A^n_t - A^n_{t-\bar{d}} ).
\end{equation}
This accounts for how many particles have recently been infected and how infectious they are; equivalently, how much the front has advanced by in the last $\bar{d}$ units of time. Of course, \eqref{eq:elastic_infection_intro} must be made precise, and we note that there is a two-way feedback between this and $A^n$, so care is needed in the construction of the particle system. The precise formulation is given in Section \ref{sect:particle_sys}. 

Looking ahead, a central finding will be that, in the mean-field limit, we have
\begin{equation}\label{eq:I_mean-field}
 I_t = \int_0^{t} v(s,A_s) \gamma(s,C_s)\sigma(s,A_s)^2\ds,
\end{equation}
where $v$ and $A$ solve the free boundary problem alluded to above. It follows that, in the mean-field limit, infections occur at a well-defined rate (per unit of time) which factorises as the product of (a) the density of susceptible individuals near the front of the epidemic, (b) the rate of infection given that these individuals are at-risk, and (c) the extent to which they are at-risk per unit of time, as given by the variability of their shielding levels near the front.

\subsection{Inert drift systems and encounter-based reactivity}\label{subsect:inert_and_encounter}

Beyond the epidemic model discussed above, our analysis contributes to two different strands of literature. Firstly, the moving boundary driven by the local time of collisions connects to the study of inert drift systems initiated by F.~Knight in \cite{Knight}. Secondly, the reactive killing mechanism that ultimately converts collisions into absorption events provides a new perspective on the notion of encounter-dependent boundary reactivity introduced by D.~Grebenkov in \cite{Grebenkov2020}.

Consider, first, the motion of an `inert', or `massive', Newtonian particle $S$ whose velocity is driven by collisions with a Brownian particle $Y$, through a transfer of momentum proportional to the local time of collisions. Since \cite{Knight}, various versions of this problem have attracted significant interest; see e.g.~\cite{White, banerjee_burdzy_duarte, banerjee_etal2, banerjee_brown, banerjee_etal}. Of particular relevance to our setting, \cite{barnes} studies the mean-field limit of a corresponding particle system with $n$ Brownian particles $Y^i$ impinging on the Newtonian particle $S^n$. This system is given by
\begin{equation}\label{eq:Barnes_system}
\begin{cases}
Y^{i}_t=Y^i_0 + B^{i}_t+\frac{1}{2}\ell^i_t,\quad Y_0^i > s_0,\\
S^n_t = s_0 - \alpha \int_{0}^t L_s^n\hspace{1pt}\ds,\quad L^n_t=\frac{1}{n}\sum_{i=1}^n \ell^{i}_t,
\end{cases}
\end{equation}
where $\ell^i$ is the local time of the collisions of $Y^i$ with $S^n$. Note that the specification of $S^n$ is similar to $A^n$ in \eqref{eq:advancing_front_intro} if we simplify to $A^n=a_0+\alpha \int_0^tI_s^n \ds$. However, distinct from \eqref{eq:Barnes_system}, our system involves two separate, but interacting, time scales: a faster time scale for the accumulation of local time and a slower time scale for the reaction events that ultimately advance the boundary. Whilst our killing mechanism \eqref{eq:elastic_infection_intro} distinguishes the two systems, the separation of timescales itself should become less visible for a large number of particles. Indeed, \cite[Thm.~2.5]{fausti_soj_22} gives
\begin{equation}\label{eq:approx_In_asymptotic}
I^n_{t} \approx \frac{1}{n}\sum_{i=1}^{n}\int_0^{t}\mathbbm{1}_{s<\tau^i}\gamma(s,C_s^n) \di \ell^{i,n}_s,
\end{equation}
in a suitable sense, for large $n$. Thus, aside from the critical reactivity aspect, the discontinuous evolution of $I^n$ is in fact comparable to that of $L^n$ from \eqref{eq:Barnes_system} in a way that should become precise in the mean-field limit. By \cite[Thm.~2.1]{barnes}, the mean-field limit of \eqref{eq:Barnes_system} is characterised by the free boundary problem
\begin{equation}\label{eq:Barnes_PDE}
\left\{ \begin{array}{@{}l@{}l}
\partial_t p(t,x) = \frac{1}{2}\partial_{xx} p(t,x), \quad (t,x)\in (0,\infty)\times (S_t,\infty), \\[3pt]
\frac{1}{2}\partial_x p(t,S_t) = -S^\prime_t p(t,S_t),\quad S^{\prime\prime }_t=-\alpha p(t,S_t).
\end{array} \right.
\end{equation}
This amounts to the limit $L$ of $L^n$ in \eqref{eq:Barnes_system} increasing at the rate $ L^\prime_t=p(t,S_t)$, much like the limit $I$ of $I^n$ will  increase at a rate proportional to the corresponding mean-field density at the interface, as postulated in \eqref{eq:I_mean-field}. However, the rate $I^\prime_t$ will involve the nonlinear proportionality factor $\gamma(s,C_s)\sigma(s,A_s)^2$, while this is unity for $L^\prime_t$. Moreover, the boundary condition in \eqref{eq:Barnes_PDE} simply serves to enforce zero flux, while a key aspect of the limiting free boundary problem in our setting will be the particular reactive boundary condition that governs the flux.

We now turn to the connection between our killing mechanism \eqref{eq:elastic_infection_intro} and the recent `encounter-based' framework for diffusion-mediated boundary reactions of \cite{Grebenkov2020}. Several problems have been studied within this framework; see e.g.~\cite{Grebenkov2022, Bressloff2022, Bressloff2022spectral, Bressloff2024}. Working with reflected Brownian motion, \cite{Grebenkov2020} introduced the notion of an `encounter-dependent' reactivity $\kappa(l)$ and its associated survival function $\Psi(l)=\exp(-\int_0^l \kappa(\lambda)\di \lambda)$, giving the probability of no surface reaction up to a level $l$ of the local time at the boundary (modulo constants). Letting $Y$ denote the reflected Brownian motion, $\ell$ its local time, and $\tau$ the reaction (or killing) time, we get
\begin{equation}\label{eq:Grebenkov}
\Prob(\tau \leq t \mid Y)=1-\Psi(\ell_t)=1-\exp\Bigl(-\int_0^t\kappa(\ell_s)\di \ell_s\Bigr).
\end{equation}
This can be seen as a direct extension of elastic Brownian motion \cite{mckean_local} where the coefficient $\kappa$ is now a function of the local time, and it reveals the connection to \eqref{eq:elastic_infection_intro}. Since $\kappa$ only depends on $\ell_s$ in \eqref{eq:Grebenkov}, the probability of surface reaction by time $t$ is a function of $\ell_t$ alone, unlike \eqref{eq:elastic_infection_intro}.

Recently, \cite{Bressloff2024} studied interacting Brownian gas on the half-line with the encounter-dependent notion of boundary reactivity, deriving the hydrodynamic limit under a mean-field ansatz. The memory trace of encounters with the (fixed) boundary is taken to be internal to each particle so that the exact specification \eqref{eq:Grebenkov} applies independently across the particles. It is noted that an alternative would be to let the memory trace be inherent to the boundary, but that this \emph{`significantly complicates the analysis of the multi-particle Brownian gas, since the probability that any one particle is absorbed will depend on previous interactions between the boundary and all other particles'} \cite[p.~18]{Bressloff2024}. Exact details are not discussed, but a natural setup would amount to letting the reaction times $\tau^i$ be determined via a common reactivity $\kappa(\tilde{L}^n_t)$, where $\tilde{L}_t^n:=\frac{1}{n}\sum_{i=1}^n \ell^i_{t\land \tau^i}$, in place of individual values $\kappa(\ell_t^i)$. Akin to \eqref{eq:elastic_infection_intro}, the common variable can also be $\tilde{I}^n_t:=\frac{1}{n}\sum_{i=1}^n \mathbf{1}_{\tau^i\leq t}$. Concerning the examples discussed in \cite{Grebenkov2020}, dependence on $\tilde{I}^n$ could, e.g., describe membrane pore blocking, where each particle that reacts at a pore obstructs it and thus reduces overall reactivity, while dependence on $\tilde{L}^n$ would be appropriate for describing, e.g., catalytic surface activation (or passivation), where repeated contact with the boundary increases (or decreases) overall reactivity. Both cases can be analysed with our methods, thus allowing for a rigorous characterisation of the associated mean-field limits and addressing the difficulties flagged in \cite{Bressloff2024}. With regard to the latter, we note that the particles in \cite{Bressloff2024} interact via their drifts. As this depends on the empirical measure in a regular way, it could readily be incorporated in our analysis, but we choose to focus only on the more challenging interactions via the free boundary and its reactivity. 

Finally, we note the following, which was stressed
in \cite{Bressloff2024}: unless $\Psi$
in~\eqref{eq:Grebenkov} is exponential (in which case the classical Robin boundary condition is recovered, in agreement with that special case of our problem), neither the single-particle setting of \cite{Grebenkov2020} nor the mean-field limit of \cite{Bressloff2024} admits a `closed equation' for the density of surviving particles. Indeed, the reactive boundary condition requires the joint law of a representative particle's position and local time. By contrast, our mean-field limit resolves this, as the law of large numbers relates the collective collision or absorption statistics to the mean-field density at the boundary.

\section{Free boundary problem in the mean-field limit}\label{sect:main_results}

In this section, we first give the precise formulation of the particle system and the assumptions we work under. Then, we introduce the free boundary problem which will characterise the mean-field limit, and we present our main results on the convergence to this limit and its well-posedness. Finally, we discuss some of the central technical aspects in relation to the existing literature.

\subsection{The interacting particle system}\label{sect:particle_sys}

We first specify the filtered probability spaces on which we shall be working. We take as given a probability space $(\Omega, \Prob , \mathcal{F})$ that supports the random inputs we introduce below, for all $n\geq 1$.

As in \cite[(3.1)]{burdzy_nualart}, we let the local time $\ell$ of a continuous semimartingale $X$ along a (random) continuous curve $t \mapsto h(t)$ be defined by
\begin{equation}\label{eq:local_time_defn}
\ell_t = \lim_{\varepsilon \rightarrow 0} \frac{1}{\varepsilon} \int_0^t \mathbbm{1}_{[h(s), h(s) + \varepsilon)} (X_s) \di \langle X \rangle_s.
\end{equation}
Following \cite{fausti_soj_22}, we introduce the filtrations $(\mathcal{F}_t^n)_{t\geq 0}$ defined by
\begin{equation}\label{eq:full_filtration}
\mathcal{F}^n_t:= \sigma\bigl({X_0^j,B_s^j, \{\chi^{j,(k)}\}_{k=1}^n } : s\in[0,t], \,j=1,\ldots,n  \bigr),\quad t\geq0,
\end{equation}
where $\{ X_0^j \}_{j=1}^n$ are the starting points, which may be constants or a family of independent random variables, $\{B^j\}_{j=1}^n$ is a family of independent standard Brownian motions, and $\{\{\chi^{j,(k)}\}_{j=1}^n\}_{k=1}^n$ is a family of independent standard Exponential random variables (utilized in the specification of the infection times). All three families are taken to be independent of each other. 

Given these inputs, the particle system $\mathbf{X}^n=(X^{1,n},\dots,X^{n,n})$ takes the form
	\begin{equation}\label{eq:sys_X}
	\left\{ \begin{array}{@{}l@{}l}
	\di X_t^{i,n} = b(t, X_t^{i,n}) \dt + \sigma (t, X_t^{i,n}) \di B^i_t + \tfrac{1}{2} \di \ell_t^{i,n}, & \quad t \in [0, \tau^i), \vspace{3pt}\\
	A_t^{n} = a_0 + \alpha\int_{t-\bar{d}}^t \varrho(t-s) I^{n}_{s} \ds, & \quad t \ge 0, \vspace{4pt}\\ 
	I_t^{n} = \frac{1}{n} \sum_{j=1}^n \mathbbm{1}_{[0,t]}(\tau^j), & \quad t \ge 0, \vspace{4pt}\\ 
	X_{t}^{i,n} = \dagger, &\quad t \ge \tau^{i}, 
	\end{array} \right.
\end{equation}
with initial conditions $X_0^{i,n} = X^i_0$,	where $\mathbf{X}^n_t$ lives in $\left([A_t^n, \infty) \cup \{ \dagger \} \right)^n$, $\ell^{i,n}$ is the local time of $X^{i,n}$ along the front $A^n$ (supported on the set $\{t\in [0,\tau^i) : X_t^{i,n} =A^n_t \}$), and the infection times $\tau^i = \inf \{t \ge 0 \, : \, X^{i,n}_t = \dagger \}$ satisfy
\begin{equation}\label{eq:tau_i_prob}
	\Prob (\tau^i \le t \mid \hat{\mathcal{F}}^{i,n}_t) = 1 - \exp \left\{ - \int_0^t \gamma(s, \hat{C}^{n, (-i)}_s) \di \hat{\ell}_s^{i,n} \right\},\quad t\geq 0,
\end{equation}
with
\begin{equation}\label{eq:reduced_filtration}
	\hat{\mathcal{F}}^{i,n}_t:= \sigma \bigl( (X_0^i,B_s^i), (X_0^j,B_s^j, \{\chi^{j,(k)}\}_{k=1}^n)  : s\in[0,t] , j\in \{ 1,\ldots,n  \}\! \setminus \! \{i\}    \bigr).
\end{equation}
Regarding the right-hand side of \eqref{eq:tau_i_prob}, for every $i$, $\mathbf{\hat{X}}^{n,(-i)}=(\hat{X}^{1,n,(-i)},\ldots,\hat{X}^{n,n,(-i)})$ denotes the auxiliary particle system where everything is specified as for $\mathbf{X}^{n}$ except that the $i$'th particle $\hat{X}^{i,n,(-i)}$ is fully reflected and immune (hence having no impact on the other particles). Within that system, $\hat{\ell}^{i,n}$ denotes the local time of $\hat{X}^{i,n,(-i)}$ along the corresponding front $\hat{A}^{n, (-i)}$, defined in terms the infected proportion $\hat{I}^{n,(-i)}$, and likewise $\hat{C}^{n, (-i)}$
is the current contagiousness defined as in \eqref{eq:current_contagiousness_intro} but in terms of $\hat{I}^{n,(-i)}$. For details on this, we refer to \cite{fausti_soj_22}.

\begin{assumption}[Structural conditions]\label{assump:coefficient_assumptions} Beyond measurability, we work under the following conditions on the coefficients of the particle system for $(t,x)\in[0,T]\times \R$, for any given $T>0$:
\begin{itemize}
\item $b(t,x)$ is Lipschitz continuous in $x$ with at most linear growth  in $x$, uniformly in $t\in[0,T]$,
\item $\sigma(t,x)$ is non-degenerate, bounded, and Lipschitz continuous jointly in $(t,x)$,
\item $\gamma(t,x)$ is Lipschitz continuous in $x$, bounded away from zero, and $C^1$ jointly in $(t,x)$,
\item $\varrho(t)$ is supported on $[0, \bar{d}]$, non-negative and right-continuous with $\varrho(0)=0$ and $\Vert \varrho \Vert_{L^1(0,\bar{d})}=1$. Moreover, it is of bounded variation with $\int_0^{\bar d} r^{-1}\Vrho{r}\,\mathrm{d}r < \infty$.
\end{itemize}
\end{assumption}

With these assumptions, we can work with the solution $\mathbf{X}^n$ to \eqref{eq:sys_X}--\eqref{eq:reduced_filtration} constructed in \cite[Theorem~2.3]{fausti_soj_22}. This is adapted to the filtration $\mathcal{F}^n$ from \eqref{eq:full_filtration} and the construction is uniquely determined by the random inputs. Finally, we enforce the following for the mean-field analysis.

\begin{assumption}[Initial conditions] \label{assump:ic_assumptions} The initial values $\{X_0^i\}_{i\geq 1}$ satisfy $\E[\exp(c(X^{i}_0)^2)]<\infty$ for all $i\geq 1$, for some $c>0$. Moreover, as $n\rightarrow \infty$, the corresponding empirical measures
\[
P_0^n:= \frac{1}{n}\sum_{i=1}^n \delta_{X_0^i}
\]
tend to $P_0\in \mathcal{P}([a_0,\infty))$, in the sense of weak convergence of measures, with $\int_{a_0}^\infty e^{cx^2}\di P_0(x)<\infty$, for some $c>0$. Here, $\mathcal{P}(S)$ denotes the space of probability measures on a given space $S$.
\end{assumption}

\subsection{Free boundary problem}\label{sect:free_boundary}

The mean-field limit of the particle system will be given by a moving boundary $A$ and a sub-probability density $v(t,x)$ on $[A_t,\infty)$, for $t>0$, where the pair $(v,A)$ is the unique solution to a free-boundary problem. A classical solution to this problem would amount to
\begin{equation}\label{eq:classic_PDE}
\partial_{t}v(t,x)=\displaystyle\frac{1}{2}\partial_{xx}\bigl(\sigma^{2}v\bigr) (t,x) -\partial_{x}(bv) (t,x),\quad (t,x)\in[0,\infty)\times [A_t,\infty),
\end{equation}
with the reactive (or mixed) boundary condition
\begin{equation}\label{eq:classic_BC}
{\frac{1}{2} \, \partial_{x} \big( \sigma^{2}} v \big) (t,A_{t}) = \big( \sigma(t,A_{t})^{2} \gamma(t,C_{t})+ b(t,A_{t})-A_{t}^{\prime}\big) \, v(t,A_{t}),\quad t\in[0,\infty),
\end{equation}
where $C_t = \int_{t-\bar{d}}^{t}\varrho(t-s)(I_{s}-I_{s-\bar{d}}) \ds$, subject to
\begin{equation}\label{eq:classic_A_I_rate}
A^\prime_{t}=\alpha \displaystyle \int_{t-\bar{d}}^{t}\varrho(t-s)I^\prime_{s} \ds \quad \text{with}\quad  I^\prime_{t}=\displaystyle v(t,A_{t})\gamma(t,C_{t})\sigma(t,A_{t})^{2},\quad t\in[0,\infty),
\end{equation}
for given initial conditions $v(0,\cdot)=v_0$, $A_0=a_0$, and $I_0=0$. Here, it is understood that $I_s=0$ for $s< 0$, and we note that the spatial derivative of $v$ at the boundary $x=A_t$ and the time derivatives at $t=0$ should of course be understood as one-sided derivatives from the right.

Notice that, if $C_t$ is causing $\gamma(t,C_t)$ to be small, then \eqref{eq:classic_BC} will be close to a fully insulated boundary condition, and so the front $A$ will only move slowly, as can also be seen from \eqref{eq:classic_A_I_rate}. On the other hand, large values of $\gamma(t,C_t)$ bring \eqref{eq:classic_BC} close to a fully absorbing boundary condition. Finally, we note that \eqref{eq:classic_BC}--\eqref{eq:classic_A_I_rate} enforce conservation of mass in the sense that
\[
I_t+\int_{A_t}^\infty v(t,x) \dx =1.
\]

For our analysis, we shall work with the following weak formulation. Given sufficient smoothness of the solution and its coefficients, along with some decay at infinity, it would follow from integration by parts that this weak formulation agrees with the classical formulation above. As initial conditions, we fix any $(P_0,a_0)\in \mathcal{P}([a_0,\infty))\times \R^+$ satisfying Assumption \ref{assump:ic_assumptions}.

\begin{defn}[Weak formulation]\label{def:weak_formulation} Let $(v,A)$ be such that (i) $t\mapsto A_t$ is continuous in $t\geq 0$, (ii) $v\geq 0$, (iii)  $(t,x) \mapsto v(t,x)$ is continuous jointly in $t>0$ and $x\geq A_t$, and (iv) $v(t,\cdot)\in L^1([A_t,\infty))$ and $v(\cdot,A_{\cdot})\in L^1([0,t])$ for all $t>0$. Let also $A_0=a_0$ and $v(t,x)\dx \rightarrow P_0$ as $t\downarrow 0$ in the topology of weak convergence of measures. Then, we say that $(v,A)$ is a weak solution of the free boundary problem \eqref{eq:classic_PDE}--\eqref{eq:classic_A_I_rate} with initial condition $(P_0,a_0)$ if 
\begin{align}\label{eq:the_weak_formulation}
	&\int_{A_t}^\infty \!\! \phi(x)v(t,x)\dx - 	\int_{a_0}^\infty \!\! \phi(x) \di P_0(x) = \int_0^t\! \int_{A_s}^\infty \!\! \partial_{xx}\phi(x)\frac{1}{2}\sigma(s,x)^2v(s,x)\dx \ds \nonumber \\
	&+ \int_0^t \!\int_{A_s}^\infty \!\!\partial_{x}\phi(x)b(s,x)v(s,x) \dx \ds + \int_0^t  \bigl( \tfrac{1}{2} \partial_x\phi(A_s)-\gamma(s,C_s)\phi(A_s)\bigr)\sigma(s,A_s)^2v(s,A_s) \ds ,
\end{align}
for all $t\geq 0$ and all $\phi \in C^2_b(\R)$, with
\begin{equation}\label{eq:weak_A_C_I}
	A_t=a_0 + \alpha \! \displaystyle \int_{t-\bar{d}}^{t}\varrho(t-s)I_{s} \ds, \;\;
	C_t = \!\int_{t-\bar{d}}^{t}\varrho(t-s)(I_{s}-I_{s-\bar{d}}) \ds, \;\;  I_t=1-\!\int_{A_t}^\infty \!\!\! v(t,x) \dx,
\end{equation}
for all $t>0$.
\end{defn}

Note that we can set  $I_0:=\lim_{t\downarrow 0}I_t=0$. The last term on the right-hand side of \eqref{eq:the_weak_formulation} encodes the boundary condition. It is well-defined by virtue of the continuity of $v$ at the boundary and the integrability of $t\mapsto v(t , A_t)$. Moreover, taking the test function $\phi\equiv1$, it also follows from the continuity of $t\mapsto v(t,A_t)$ for $t>0$ that $A$, $C$, and $I$ are continuously differentiable at positive times. In turn, we still have the condition \eqref{eq:classic_A_I_rate} for $t>0$.

\begin{rmk}[Shape of the moving boundary] We have $A_t=F(\int_{t-\bar{d}}^t \varrho(t-s) I_{s} \ds) $ with $F(x)=a_0 + \alpha x$. For a general $F:\R\rightarrow \R$, the arguments in Sections \ref{sect:density_perturb_and_bounds}--\ref{sect:decoupled_mf} continue to apply as long as $F\in W^{1,\infty}(\R)$, and the arguments in Section \ref{sect:free_boundary_uniqueness} go through with minimal changes if $F\in W^{2,\infty}(\R)$.
\end{rmk}

\subsection{Well-posedness and mean-field convergence}

Below, we summarise our main results on the well-posedness of the free boundary problem and the weak convergence of the empirical measures associated to the particle system.

\begin{thm}[Well-posedness]\label{thm:free_boundary_problem} Let Assumptions \ref{assump:coefficient_assumptions} and \ref{assump:ic_assumptions} be in place. There exists a unique weak solution $(v,A)$ to the free boundary problem given by Definition \ref{def:weak_formulation}. Moreover, for any $T>0$, we have
\[
 v(t,x) \leq \frac{c_1 }{\sqrt{t}} \exp\bigl( - c_2(x-A_t)^2 \bigr),\quad \text{for all}\quad x\in[A_t,\infty), \,t\in(0,T] ,
\]
for some constants $c_1,c_2>0$.
\end{thm}

 As in the classical works \cite{Ito_S_prime, Mitoma_S_prime}, we view the empirical measures $(P^n_t)_{t\geq 0}$ associated to \eqref{eq:sys_X} as $\mathcal{S}^\prime$-valued c\`adl\`ag stochastic processes given by
\begin{equation}\label{eq:emprical_measures}
\phi \mapsto  \langle P^n_t,\phi \rangle := \frac{1}{n}\sum_{i=1}^n \mathbf{1}_{t<\tau^i}\phi(X^{i,n}_t),\quad \phi\in \mathcal{S},
\end{equation}
where $\mathcal{S}$ denotes the Schwartz space on $\R$ and  $\mathcal{S}^\prime$ its dual, the space of tempered distributions. We then write $D_{\mathcal{S}^\prime}[0,T]$ for the space of $\mathcal{S}^\prime$-valued c\`adl\`ag stochastic processes on $[0,T]$ and we will show that there is weak convergence on this space for all $T>0$.

\begin{thm}[Mean-field limit]\label{thm:summary_contributions} Let Assumptions \ref{assump:coefficient_assumptions} and \ref{assump:ic_assumptions} be in place. Let $(v,A)$ be the unique weak solution to the free boundary problem given in Definition \ref{def:weak_formulation}, and let the empirical measures $P^n$ be defined as in \eqref{eq:emprical_measures}. For any $T>0$, we have 
\[(P^n_t,A^n_t)_{t\in[0,T]} \Rightarrow (\nu_t,A_t)_{t\in[0,T]},\quad\text{as}\quad n\rightarrow \infty,
\] on $D_{\mathcal{S}^\prime}[0,T]\times D_{\R}[0,T]$ in the uniform topology, where $\nu_0=P_0$ and $\nu_t=v(t,x)\hspace{1pt}\dx$ for $t>0$.
\end{thm}

Due to the discontinuous effect of reaction events in \eqref{eq:emprical_measures}, the convergence of the $P^n$ takes place on the Skorokhod space $D_{\mathcal{S}^\prime}[0,T]$ rather than a corresponding Wiener space. Nevertheless, the mean-field limit will be shown to be continuous. We shall start out by establishing the weak convergence with respect to Skorokhod's M1 topology on $D_{\mathcal{S}^\prime}[0,T]$, see \cite{ledger_M1}. Since the limit is shown to be continuous, this will yield weak convergence with respect to the uniform topology.

\subsection{A natural extension}

As discussed in \cite{fausti_soj_22}, it can be natural to allow the drift and diffusion coefficients in \eqref{eq:sys_X} to depend on the initial position $X_0^{i}$. In the epidemic model, the latter describes the initial shielding from the disease and it is convenient to partition the population according to this. For example, different subsets of the population starting with certain levels of initial shielding can then be taken to mean-revert around different base-levels, as illustrated in \cite[Fig.~3]{fausti_soj_22}.

This setting is akin to the study of interacting particle systems in random media, as in \cite{Dai_Pra}, and simply requires replacing $P_t^n$ of \eqref{eq:emprical_measures} with the double-layer empirical flow
\begin{equation*}
 \bar{P}^n_t  = \frac{1}{n}\sum_{i=1}^n \mathbf{1}_{t<\tau^i}\delta_{(X^{i,n}_t,X^{i}_0)},
\end{equation*}
where it happens that the $i$'th component of the random medium is the $i$'th initial position. Analogously to \cite{Dai_Pra}, the mean-field limit will then be a collection of coupled PDEs indexed by a medium parameter, which here is the initial point $x_0\sim P_0$. That is, \eqref{eq:the_weak_formulation}--\eqref{eq:weak_A_C_I} becomes
\begin{align*}
	&\int_{\bar{A}_t}^\infty \!\! \phi(x)v(t,x;x_0)\dx - 	\phi(x_0)= \int_0^t\! \int_{\bar{A}_s}^\infty \!\! \partial_{xx}\phi(x)\frac{1}{2}\sigma(s,x;x_0)^2v(s,x;x_0)\dx \ds \nonumber \\
	&+ \int_0^t \!\int_{\bar{A}_s}^\infty \!\!\partial_{x}\phi(x)b(s,x;x_0)v(s,x;x_0) \dx \ds + \int_0^t  \mathcal{D}_{\bar{A}_s,\bar{C}_s}(\phi)\sigma(s,\bar{A}_s;x_0)^2v(s,\bar{A}_s;x_0) \ds ,
\end{align*}
with $\mathcal{D}_{\bar{A}_s,\bar{C}_s}(\phi)= \tfrac{1}{2} \partial_x\phi(\bar{A}_s)-\gamma(s,\bar{C}_s)\phi(\bar{A}_s)$, where $\bar{A}$ and $\bar{C}$ are as in \eqref{eq:weak_A_C_I} but in terms of \[
\bar{I}_t=1-\!\int_{\bar{A}_t}^\infty \! \int_{a_0}^\infty \!\! v(t,x;x_0) \di P_0(x_0) \dx.
\]

Our analysis allows us to recover this formulation in the limit, but we leave out the details for simplicity of the presentation. As long as $b$ and $\sigma$ are continuous functions of $x_0$ and Assumption \ref{assump:coefficient_assumptions} applies to the coefficients uniformly in $x_0$, the arguments go through as they are. For the epidemic model discussed in Section \ref{sect:heuristics_particle}, a very practically relevant, yet simple, example would amount to having $P_0$ (and $P_0^n$) supported on $k$ distinct subsets of $(a_0,\infty)$ with certain coefficients $b^l(t,x)$ and $\sigma^l(t,x)$ specified for each of the subsets $l=1,\ldots k$ (in this case, the coefficients are continuous in $x_0$ for the disjoint union topology). Then, the above simplifies to $k$ free boundary problems in the form of Definition \ref{def:weak_formulation} but coupled through a common free boundary.

\subsection{Overview of our analysis and related works}

For our weak convergence and uniqueness arguments, we rely on a perturbation formula and Aronson estimate for the density of each particle conditional on a sub-filtration of
$\hat{\mathcal{F}}^{i,n}_t$ from \eqref{eq:reduced_filtration}. This shares some similarities with the methods of \cite{Zhongmin2004, QianZheng2004} which are extended to reflected SDEs in \cite{Zhongmin2023}. Among other things, \cite{Zhongmin2023} derives a perturbation formulation for $p_h$ in terms of $p$ and $h$, where $p_h$ is the transition density of a reflected diffusion $\di X_t=h(t,X_t)\dt+\di B_t+\di \ell_t$ with bounded drift and $p$ is the transition density of a reflected Brownian motion (in a smooth domain $D$). This goes via a Girsanov argument, which relies on the boundedness of $h$, and the arguments furthermore exploit existing Aronson estimates for $p_h$. We take a different approach to deriving our perturbation formula and obtain the Aronson estimates directly. Furthermore, we address the reactive boundary aspect and allow for our more general assumptions on the coefficients.

The rest of the paper concerns the mean-field limit. In Section \ref{sect:weak_convergence}, we establish tightness of the empirical measure flows and identify their weak limits. A key subtlety lies in rigorously characterising the boundary condition, and we first obtain an a priori weaker notion of random measure-valued solutions, which we later connect back to Definition \ref{def:weak_formulation}. In Section \ref{sect:decoupled_mf}, we introduce a `decoupled' mean-field particle which takes the moving boundary and its reactivity as inputs specified by a given limit point. The machinery of Section \ref{sect:density_perturb_and_bounds} guarantees that it enjoys good properties pointwise and allows us to show that it must agree with the limiting randomised weak solution from which it was defined, thus yielding a probabilistic representation and showing that each random limit point can be realised as a family of weak solutions to the desired free boundary problem. In Section \ref{sect:free_boundary_uniqueness}, we then prove uniqueness of weak solutions in the sense of Definition \ref{def:weak_formulation}, thereby collapsing any limit point to the single unique weak solution of the free boundary problem. This is achieved through energy estimates at the level of the sub-probability distribution function for a given solution, as this pairs naturally with the nature of the particle-boundary interactions.

As discussed in Section \ref{subsect:inert_and_encounter}, there are interesting connections to \cite{barnes}, but the analysis is fundamentally different. In particular, the system \eqref{eq:Barnes_system} is encoded by a (coupled) Skorokhod map for $L^n=\frac{1}{n}\sum_{i=1}^n\ell_t^i$ with $\ell^i_t=2\sup_{r\leq t}(Y_0^i+B_r^i-S_r^n)^{-}$, where $S^n=s_0+\alpha\int_0^t L^n_s\ds$. This is exploited to get $C[0,T]$-tightness of $L^n$ directly from that of the Brownian motions. Next, the Lipschitzness of the Skorokhod map and a delicate discretisation argument is employed to show that $L^n$ tends uniformly a.s.~to a deterministic limit $L$. By estimates for the Skorokhod map, there is tightness of the empirical measures, so the Lipschitzness of the local time functionals and a coupling argument give Wasserstein convergence to a reflected Brownian motion whose local time $\ell_t=2\sup_{r\leq t}(Y_0+B_r-S_r)^{-}$ and density $p(t,x)$ satisfies $L_t = \E[\ell_t]= \int_0^t p(s,S_s) \ds$ with $ S_t=s_0+\alpha\int_0^tL_s \ds$. The second equality for $L$ is obtained from \eqref{eq:local_time_defn} by passing the limit in $\varepsilon\downarrow 0$ outside the expectation and inside the integral in time, justified by the density estimates and regularity from \cite[Thms.~2.4 and 2.9]{Burdzy_et_al}. By \cite{Burdzy_et_al}, $p(t,x)$ is seen to be a classical solution to the free boundary problem \eqref{eq:Barnes_PDE} (note that \cite{barnes} refers to the boundary as $C^2([0,T])$ throughout when it is in fact only $C^2((0,T])$, but the results still hold), and  the Skorokhod map for $\ell$ together with $L_t=\E[\ell_t]$ yields a contraction argument that gives uniqueness.

Our approach is more general and we note that it readily accommodates extensions such as dependence on the empirical measures in the coefficients $b$ and $\sigma$, although we do not implement it here. We also note that \cite{barnes} relies on \cite{Burdzy_et_al}, but the relevant results of \cite{Burdzy_et_al} on reflected Brownian and heat equations in time-dependent domains do not cover the dynamics we consider. This is handled by the machinery in Section \ref{sect:density_perturb_and_bounds}. Finally, the way we realise the boundary term of the weak formulation \eqref{eq:the_weak_formulation} in the limit also goes via \eqref{eq:local_time_defn}, but we rely on a more delicate operation at level of the particle system: we first exploit \eqref{eq:approx_In_asymptotic} via a martingale argument, and then we show that the limit in $\varepsilon \downarrow 0$ commutes with an expectation of the empirical measures $P_t^n$, before finally confirming that the limits in $n\rightarrow \infty $ and $\varepsilon \downarrow 0$ can be interchanged. This relies critically on appropriate conditioning together with the density estimates and regularity from Section \ref{sect:density_perturb_and_bounds}.

We conclude by highlighting a few other related strands of work. Firstly, \cite{baker_shkolnikov_undercooling, hambly_meier} have studied kinetic undercooling in the one-dimensional Stefan problem via a McKean--Vlasov formulation for Brownian motion with the classical elastic killing mechanism, i.e., \eqref{eq:Grebenkov} with constant $\kappa>0$. This also involves a free boundary that increases with killing events, but the analysis is distinct from ours. Secondly, one-dimensional free boundary problems for epidemic spreading with Stefan-type conditions have been widely studied in the PDE literature; see e.g.~\cite{Hadeler, Wang_etal_2019, Wang_Du_2021, Wang_Du_2022}. These models have not been considered in a mean-field context, and we stress that our infection mechanism, which achieves the natural factorisation \eqref{eq:I_mean-field}, leads to a different reactive and nonlinear specification of the free boundary. Thirdly, we note that \eqref{eq:I_mean-field} can be seen as a generalisation of the classical SIR model, and we stress that \cite{berestycki_desjardins_weitz_oury} directly extends the SIR model by attaching a risk variable to the susceptible population that diffuses on the positive half-line. Unlike our mean-field limit, there is no free boundary and the risk variable is only reflected: instead, infections occur across the entire susceptible population at a rate that increases with the risk level. Finally, we note that \cite{pang_pardoux_2022} studies epidemic mean-field limits for non-Markovian SIR-type models in which infections arise from a Poisson process with a rate depending on the state of the system.

\section{Perturbation formula and density bounds}\label{sect:density_perturb_and_bounds}

Recall that the particle system involves an auxiliary family of `reduced information' filtrations $\{\hat{\mathcal{{F}}}^{i,n}\}_{i=1}^n$, which are defined in \eqref{eq:reduced_filtration}. We begin by decomposing these into
\begin{gather}
\hat{\mathcal{F}}^{i,n}_t = \mathcal{{G}}_{t}^{i,n} \lor \sigma(B^i_r,X_0^i : r\leq t), \nonumber \\
\mathcal{{G}}_{t}^{i,n}:=\sigma((X_{0}^{j},B_{r}^{j},\chi^{j,(k)}):j\neq i,\;k\leq n,\,r\leq t), \quad \text{for} \quad i=1,\ldots,n. \label{eq:G_filtration}
\end{gather}
We will show that, conditional on \eqref{eq:G_filtration}, the law of each particle enjoys an Aronson estimate and may be expressed as a perturbation around the transition density of a reflected Brownian motion. The conditioning on \eqref{eq:G_filtration} plays a key role in Section \ref{sect:weak_convergence}. We note that our approach is self-contained and only relies on probabilistic techniques. We find it both intuitive and succinct.

\subsection{A first Girsanov argument}
We are looking to perturb the $\mathcal{G}^{i,n}$-conditional density of each particle $X^{i,n}$ around that of a standard reflected Brownian motion. To this end, we first rely on a Girsanov argument to relate each particle to a reflected SDE on the half-line with unit volatility.

To ease notation, we set 
\begin{equation*}
\hat{A}^{i,n}_t:=\hat{A}_t^{n,(-i)}, \quad   
\hat{C}_t^{i,n}:=\hat{C}_t^{n,(-i)}, \quad \mathrm{and}
\quad
\hat{I}_t^{i,n}:=\hat{I}_t^{n,(-i)}, \quad
\mathrm{for} \: i=1,\ldots, n.
\end{equation*}
By Tonelli and integration by parts, we have that
\begin{equation}\label{eq:A_is_AC}
\hat{A}^{i,n}_t = a_0 + \alpha \int_0^t\Bigl(\int_{s-\bar{d}}^s \varrho(s-r) \di \hat{I}^{i,n}_r\Bigr)\di s,\quad \Bigl| \int_{s-\bar{d}}^s \varrho(s-r) \di \hat{I}^{i,n}_r \Bigr|  \leq \Vert \varrho \Vert_{L^\infty([0,\bar{d}])}<\infty, 
\end{equation}
where the bound used that $\hat{I}^{i,n}$ is non-decreasing and bounded by 1. Now define
\begin{equation}\label{eq:Lamperti_transform_Gamma}
 \Gamma^{i,n}(t, y)(\omega):=\int_{\hat{A}^{i,n}_t(\omega)}^{y} \frac{1}{\sigma(t,x)}\dx.
\end{equation}
In view of \eqref{eq:A_is_AC} and Assumption \ref{assump:coefficient_assumptions}, it then follows from \cite[Theorem~2.4]{fausti_soj_22} and its proof that
\begin{equation}\label{eq:conditional_sub_prob}
    \Prob (X_{t}^{i,n}\in(a,b)\mid\mathcal{{G}}_{T}^{i,n})=\Prob \bigl(Z_{t}^{i,n}\in(\Gamma^{i,n}(t,a),\Gamma^{i,n}(t,b)),\;t<\tau_{Z}^{i,n}\mid\mathcal{{G}}_{T}^{i,n}\bigr)
\end{equation}
with
\begin{equation}\label{eq:Z_i_n}
	\di Z_{t}^{i,n} = \tilde{b}^{i,n}(t,\omega,Z_{t}^{i,n})\dt + \di B_{t}^{i,n} + \tfrac{1}{2} \di \ell_{t}^{0}(Z^{i,n}),
\end{equation}
where $B^{i,n}$ is a Brownian motion independent of $\mathcal{{G}}^{i,n}$, $\ell_{t}^{0}(Z^{i,n})$ is the local time of ${Z^{i,n}}$ at $0$,
\begin{equation}\label{eq:tau_Upsilon}
    \tau_{Z}^{i,n}=\inf\Bigl\{ t\geq0:\int_{0}^{t} \sigma(r,\hat{A}_{r}^{i,n}) \gamma(r,\hat{C}_{r}^{i,n}) \di  \ell_{r}^{0} (Z^{i,n}) > \chi^{i,n} \Bigr\} 
\end{equation}
for a standard exponential random variable $\chi^{i,n}$ independent of $\hat{\mathcal{{F}}}^{i,n}$, and $\tilde{b}^{i,n}$ is defined as in the proof of \cite[Theorem~2.4]{fausti_soj_22}. All we need to know is that $\tilde{b}^{i,n}$ is $\mathcal{{G}}^{i,n}$-adapted and, for any $T>0$,
\begin{equation}\label{eq:lin_growth_bound}
	\esssup_{(t,\omega)\in [0,T]\times \Omega}| \tilde{b}^{i,n}(t,\omega ,z) | \leq \kappa(1+z)\quad \text{for all} \quad z\geq0,
	\end{equation}
	for some $\kappa>0$. This will allow us to exploit Girsanov's theorem via the following lemma.
  
\begin{lemma}[Change of measure]\label{lem:girsanov}
	Fix $T>0$, $n\geq 1$, $i\in \{1,\ldots,n\}$, and consider $Z^{i,n}$ from \eqref{eq:Z_i_n}. There exist a probability measure $\Q^{i,n}$ on $\hat{\mathcal{F}}_T^{i,n}$, equivalent to the restriction of $\Prob$ to $\hat{\mathcal{F}}_T^{i,n}$, so that $Z^{i,n}$
    becomes a reflected Brownian motion on $[0,T]$ under $\Q^{i,n}$ for the filtration $\hat{\mathcal{F}}^{i,n}$. Furthermore, $Z^{i,n}$ becomes independent of $\mathcal{G}^{i,n}_T$ under $\Q^{i,n}$, and the Radon--Nikodym derivative $\di \Prob/\di \Q^{i,n}$ (still dependent on $\mathcal{G}^{i,n}_T$) satisfies
	\begin{equation}\label{eq:radon-nikodym_G-cond_1}
	\E^{\Q^{i,n}}\Bigl[ \frac{\di  \Prob }{\di  \Q^{i,n}} \; \big| \;  \mathcal{G}^{i,n}_T\Bigr] = 1
	\end{equation}
	as well as
	\begin{equation}\label{eq:p-bound_radon-nikodym}
	\E^{\Q^{i,n}}\Bigl[ \Bigl( \frac{\di \Prob}{ \di \Q^{i,n}}\Bigr)^p \; \big| \; \mathcal{G}_T^{i,n}\Bigr] \leq 	C,
	\end{equation}
	for some $p>1$ close enough to $1$, for a fixed constant $C>0$ which only depends on $p$ and $T$.
\end{lemma}
\begin{proof}
	Let $\tilde{\mathcal{F}}^{i,n}_t:=\sigma(B_s^{i,n},X_0^{i,n}:s\leq t) \lor \mathcal{G}_T^{i,n}$, noting that $\hat{\mathcal{F}}^{i,n}_t \subseteq \tilde{\mathcal{F}}^{i,n}_t$. Then define $\Q^{i,n}$ by
	\[
	\Q^{i,n}(E):= \E[Y_T\mathbf{1}_E], \qquad \text{for all} \quad E\in \hat{\mathcal{F}}^{i,n}_T=\tilde{\mathcal{F}}^{i,n}_T,
	\]
	where $Y_t$ is the stochastic exponential of $-\int_0^t \tilde{b}^{i,n}(s,\omega, Z_s^{i,n})dB_s^{i,n}$, which is adapted to $\hat{\mathcal{F}}^{i,n}_t $. By the same reasoning as in \cite[Proposition 6.1 \& Lemma 6.4]{HamblySojmark19}, we can deduce from the linear growth bound \eqref{eq:lin_growth_bound} on $\tilde{b}^{i,n}$ that $(Y_t)_{t\in[0,T]}$ is a martingale for the larger filtration $\tilde{\mathcal{F}}^{i,n}$. Therefore, an application of Girsanov's theorem shows that $Z^{i,n}$ is a Brownian motion under $\Q^{i,n}$ in the filtration $\tilde{\mathcal{F}}^{i,n}$ on $[0,T]$ (as well as in the smaller filtration $\hat{\mathcal{F}}^{i,n}$ to which it is also adapted).
	
	Next, we confirm that $Z^{i,n}$ is independent of $\mathcal{G}_T^{i,n}$ under $\Q^{i,n}$. Fix an arbitrary sequence of times $t_1,\ldots,t_k$ and bounded continuous functions $f_1,\ldots,f_k$. Then we get
	\begin{align*}
	\E^{\Q^{i,n}}\biggl[ \prod_{l=1}^k &f_l(Z^{i,n}_{t_l})\;\Big| \; \mathcal{G}_T^{i,n} \biggr] = \E^{\Q^{i,n}}\biggl[ \E^{\Q^{i,n}}\biggl[ \prod_{l=1}^k f_l\bigl( Z_0^{i,n}+(Z^{i,n}_{t_l}-Z_0^{i,n})\bigr) \;\Big| \; \sigma(X_0^{i,n}) \lor \mathcal{G}_T^{i,n} \biggr] \; \Big| \; \mathcal{G}_T^{i,n}\biggr]\\
	&=\E^{\Q^{i,n}}\biggl[ \E^{\Q^{i,n}}\biggl[ \prod_{l=1}^k f_l\bigl( z+(Z^{i,n}_{t_l}-Z_0^{i,n})\bigr) \biggr]_{z=Z_0^{i,n}} \biggr]=\E^{\Q^{i,n}}\biggl[ \prod_{l=1}^k f_l(Z^{i,n}_{t_l}) \biggr],
	\end{align*}
	since $Z^{i,n}_0$ is  $\sigma(X_0^{i,n})$-measurable (note $\Gamma^{i,n}(0,\cdot)$ is deterministic) while the increments of $Z^{i,n}$ are independent of $\sigma(X_0^{i,n}) \lor \mathcal{G}_T^{i,n}$ (the latter being contained in $\tilde{\mathcal{F}}^{i,n}_0$) by the above Girsanov argument, and since $\sigma(X_0^{i,n}) \perp \mathcal{G}_T^{i,n} $ (giving the last equality). This confirms the independence.
	
	It remains to show the two final claims \eqref{eq:radon-nikodym_G-cond_1} and \eqref{eq:p-bound_radon-nikodym}. Firstly, we have $\di \Prob/\di \Q^{i,n}=Y_T^{-1}$ with
	\[
	Y_T^{-1}=1+\int_0^T Y_s^{-1}\tilde{b}^{i,n}(s,\omega,Z^{i,n}_s) \di Z^{i,n}_s,
	\]
	where the integrand is $\tilde{\mathcal{F}}^{i,n}$-adapted. In turn, we get $\E^{\Q^{i,n}}[Y_T^{-1}\mid \mathcal{G}_T^{i,n}]=1$ from
	\[
\E^{\Q^{i,n}}\Bigl[ \int_0^T Y_s^{-1} \di  Z^{i,n}_s \;\big| \; \mathcal{G}_T^{i,n} \Bigr] =  \E^{\Q^{i,n}}\biggl[\E^{\Q^{i,n}}\Bigl[ \int_0^T Y_s^{-1} \di  Z^{i,n}_s \;\big| \; \tilde{\mathcal{F}}^{i,n}_0 \Bigr] \;\Big| \; \mathcal{G}_T^{i,n}\biggr] = 0,
	\]
	since the stochastic integral is a martingale on $[0,T]$ for the filtration $\tilde{\mathcal{F}}^{i,n}$ under $\Q^{i,n}$, by the above. This verifies \eqref{eq:radon-nikodym_G-cond_1}. Finally, recall again that $\tilde{b}^{i,n}$ is bounded in terms of $1+Z^{i,n}$, where the latter is independent of $\mathcal{G}_T^{i,n}$ under $\Q^{i,n}$. Thus, we can argue as in the proof of \cite[Lemma 6.5]{HamblySojmark19} to deduce the uniform conditional $L^p$ bound \eqref{eq:p-bound_radon-nikodym}, for a small enough $p>1$.
\end{proof}

\subsection{The perturbation formula and bounds}

We are interested in the random density function $p^{i,n}(t,x)$ for the conditional sub-probability measure \eqref{eq:conditional_sub_prob}. We will show that it exists, derive bounds for it, and confirm that it may be expressed as a particular perturbation around the density of a standard reflected Brownian motion on the positive half-line. The perturbed part of this expression involves two terms: one term for the way in which $Z^{i,n}$ of \eqref{eq:Z_i_n} differs from a reflected Brownian motion and another term for describing the loss of mass due to killing the particle upon reacting with the boundary.

Recall that the transition density for a reflected Brownian motion on $\R^+$ is given by
    \begin{equation}\label{eq:reflected_kernel}
    N(x,t;y,s) = \frac{1}{\sqrt{2\pi(t-s)}}e^{-\frac{(y-x)^2}{2(t-s)}} + \frac{1}{\sqrt{2\pi(t-s)}}e^{-\frac{(y+x)^2}{2(t-s)}}.
    \end{equation}

\begin{thm}[Perturbation formula]\label{thm:perturb_density}
	Restricted to the event $\{X_t^{i,n}\neq \dagger\}$, for $t>0$, the conditional law of $X_t^{i,n}$ given $\sigma(\cup_{s>0}\mathcal{{G}}_{s}^{i,n})$ has a random density $p^{i,n}(t,x)$, which is supported on $[\hat{A}_t^{i,n},\infty)$ and  jointly continuous in $x
    \geq \hat{A}_t^{i,n}$ and $t>0$. This density can be expressed as
	\begin{equation}\label{eq:p_and_p-hat}
	p^{i,n}(t,x)=\hat{p}^{i,n}\Bigl(t\,,\int_{\hat{A}^{i,n}_t}^x \frac{1}{\sigma(t,y)}\dy  \Bigr)\frac{1}{\sigma(t,x)},\quad \text{for all}\quad x \geq \hat{A}_t^{i,n},\,t> 0,
	\end{equation}
	where $\hat{p}^{i,n}$ is supported on $[0,\infty)$ and satisfies the implicit Volterra relation
	\begin{align}\label{eq:density_rep-1}
	\hat{p}^{i,n}(t,x) & =\int_0^\infty N(x,t,y,0)\hspace{1pt} \di \hat{\mu}^{i}_0(y) +\int_{0}^{t}\int_{0}^{\infty}\partial_{y}N(x,t; y ,r)\tilde{b}^{i,n}(r, y)\hat{p}^{i,n}(r,y)\dy \dr \nonumber\\
	& \qquad\; -\int_{0}^{t}N(x,t;0,r)\gamma(r,\hat{{C}}_{r}^{i,n})\sigma(r,\hat{{A}}_{r}^{i,n})^2p^{i,n}(r,\hat{A}^{i,n}_r)\dr
	\end{align}
     with $\hat{\mu}_0^{i}=\mathrm{Law}(X_0^i) \circ (\int_{a_0}^\cdot \frac{1}{\sigma(0,y)}\dy)^{-1}$.    Furthermore, for any $T>0$, we have
\begin{equation}\label{eq:lamperti_aronson}
	\hat{p}^{i,n}(t,x) \leq \frac{C_1}{\sqrt{t}} \exp\Bigl( -c_1x^{2}\Bigr),\quad \text{for all}\quad x\in[0,\infty), \,t\in(0,T],
	\end{equation}
for some constants $c_1,C_1>0$, and, in turn, there are constants $c_2,C_2>0$ such that
	\begin{equation}\label{eq:particle_aronson}
		p^{i,n}(t,x) \leq \frac{C_2 }{\sqrt{t}} \exp\Bigl( -c_2(x-\hat{A}^{i,n}_t)^2 \Bigr),\quad \text{for all}\quad x\in[\hat{A}^{i,n}_t,\infty), \,t\in(0,T],
	\end{equation}
	These constants can be chosen independently of $i=1,\ldots,n$ and $n\geq 1$, and we note that \eqref{eq:lamperti_aronson}--\eqref{eq:particle_aronson} confirm that the perturbation formula \eqref{eq:density_rep-1} is well-defined.
\end{thm}
\begin{proof}
	 To derive \eqref{eq:p_and_p-hat}-\eqref{eq:density_rep-1}, we shall work with $Z^{i,n}$ as defined in \eqref{eq:Z_i_n}. Consider an arbitrary interval $(a,b)$ in $\R^+$, and fix any $T>0$. We will begin by deriving a perturbation of
	\[
	 \Prob\bigl( Z^{i,n}_t\in (a,b) , t<\tau_Z^{i,n} \mid \mathcal{G}^{i,n}_T  \bigr)
	\]
around the probability that a reflected Brownian motion is in $(a,b)$ at time $t$. To do this, we consider, for each
	$s\in[0,t]$, the expected probability, conditionally on $\mathcal{{G}}_{T}^{i,n}$,
	of ending up in $(a,b)$ at time $t$ when first running $Z^{i,n}$
	for $s$ units of time subject to $s<\tau_{Z}^{i,n}$ and then
	starting an independent reflected Brownian motion from the
	resulting position and running it for the
	remaining time $t-s$. This can be expressed as
	\begin{align}
	&F_{(a,b)}(t;s)  :=\E \Bigl[\mathbf{1}_{s<\tau_{Z}^{i,n}}\int_{a}^{b}N(x,t;Z_{s}^{i,n},s) \dx \mid\mathcal{{G}}_{T}^{i,n}\Bigr],\quad \text{for}\quad  s \in [0,t),
	\label{eq:convec_heat}
	\end{align}
 for $N(x,t;y,s)$ from \eqref{eq:reflected_kernel}. Since $\lim_{s\uparrow t}\int_{a}^{b}N(x,t;y,s)\dx = 1$
	for all $y\in(a,b)$, while it vanishes for all $y\notin[a,b]$, it follows from the definition of $F_{(a,b)}$ in \eqref{eq:convec_heat}
	that 
	\begin{align*}
	F_{(a,b)}\bigl(t;t-\bigr)= \Prob \bigl( Z_{t}^{i,n}\in (a,b),  t<\tau_{Z}^{i,n}\mid\mathcal{{G}}_{t}^{i,n}\bigr),
	\end{align*}
	where we have also used that the law of $Z^{i,n}_t$ has a density, so there are no atoms, as follows readily from Lemma \ref{lem:girsanov}. At the other end point $s=0$, we instead have
	\begin{equation*}
	F_{(a,b)}(t;0) =\int_{0}^{\infty}\!\int_a^b N(x,t;y,0)\dx \di  \hat{\mu}_{0}^{i}(y)
	\end{equation*}
	 Consequently, we can write
	\begin{equation}\label{eq:prob_approx_error}
	\Prob (Z_{t}^{i,n}\in(a,b),\;t<\tau_{Z}^{i,n}\mid\mathcal{{G}}_{t}^{i,n}) = \int_{0}^{\infty}\!\!\int_a^b N(x,t;y,0) \dx \di  \hat{\mu}_{0}^{i}(y) +\int_{0}^{t}F_{(a,b)}(t;\ds),
	\end{equation}
	so it remains to show that the measure $F_{(a,b)}(t;\ds)$ is absolutely continuous on $[0,t]$ and that this gives the additional terms in \eqref{eq:density_rep-1}. To this end, notice first that \eqref{eq:tau_Upsilon} gives
	\begin{align}\label{eq:F_for_Ito}
	F_{(a,b)}(t;s)  & = 	\E \Bigl[ \Prob \bigl(s<\tau_{Z}^{i,n} \mid \mathcal{G}^{i,n}_T\lor \sigma(Z^{i,n}) \bigr) \int_{a}^{b}N(x,t;Z_{s}^{i,n},s) \dx\mid\mathcal{{G}}_{T}^{i,n}\Bigr]\nonumber\\
	&= 	\E \Bigl[ e^{-\int_{0}^{s}\sigma(r,\hat{{A}}_{r}^{i,n})\gamma(r,\hat{{C}}_{r}^{i,n})\di \ell_{r}^{0}(Z^{i,n})  } \int_{a}^{b}N(x,t;Z_{s}^{i,n},s)\dx\mid\mathcal{{G}}_{T}^{i,n}\Bigr].
	\end{align}
	Next, note that $N(x,t;y,s)$ is the Green's function for the heat equation with Neumann boundary condition, so $(y,s)\mapsto N(x,t;y,s)$ solves the adjoint (backward) problem
	\begin{equation*}
	\partial_{s}v(y,s)=-\frac{1}{2}\partial_{yy}v(y,s),\quad\partial_{y}v(0,s)=0, \quad \text{for}\quad (y,s)\in\R^{+}\times[0,t),
	\end{equation*}
	for the terminal condition $v(\cdot,t)=\delta_x$. Since the Brownian motion $B^{i,n}$, driving the dynamics of $Z^{i,n}$, is independent of $\mathcal{{G}}^{i,n}$, an application of It\^o's formula (inside the conditional expectation on the right-hand side of \eqref{eq:F_for_Ito}) therefore yields the equivalence of measures
	\begin{align}
	&F_{(a,b)}(t;\dr)  = \int_{a}^{b}  \E \bigl[\mathbf{1}_{r<\tau_{Z}^{i,n}}\partial_{y}N(x,t;Z_{r}^{i,n},r)\tilde{b}^{i,n}(r,Z_r^{i,n})\mid\mathcal{{G}}_{T}^{i,n}\bigr]  \dx  \dr \nonumber\\
	&\qquad \qquad - \int_{a}^{b} N(x,t;0,r)  \sigma(r,\hat{{A}}_{r}^{i,n})\gamma(r,\hat{{C}}_{r}^{i,n}) \dx  \di  \E \bigl[ \ell_{r\land \tau_{Z}^{i,n}}^{0} (Z^{i,n}) \mid\mathcal{{G}}_{T}^{i,n} \bigr] \label{eq:F_ab_char},
	\end{align}
	on $[0,s]$ for any $s<t$. While the second term on the right-hand side is well-defined over $[0,t]$, the above characterisation of $F_{(a,b)}(t;\dr)$ need not hold on $[0,t]$ due to the singularity of $r\mapsto\partial_{y}N(x,t;Z_r^{i,n},r)$ as  $r\nearrow t$ if $Z_r^{i,n}$ gets close to $x$. That is, the integrand may not be a legitimate $L^1$ function with respect to $\dr$ on $[0,t]$. Nevertheless, we will be able to show that the singularity is countered by the regularizing effect of the Brownian part of $Z^{i,n}$ once the conditional expectation is taken into account, using Lemma \ref{lem:girsanov}. Indeed, we claim that, for all $x\in\R^+$,
	\[
	r\mapsto \E \bigl[\mathbf{1}_{r<\tau_{Z}^{i,n}}\partial_{y}N(x,t;Z_{r}^{i,n},r)\tilde{b}^{i,n}(r,Z_r^{i,n})\mid\mathcal{{G}}_{T}^{i,n}\bigr]
	\]
	is (almost surely) in $L^1([0,t])$. Bounding each term of
	\begin{align*}
	\partial_{y}N(x,t;y,s) & =\frac{(x-y)}{\sqrt{2\pi}(t-s)^{3/2}}e^{-\frac{(x-y)^{2}}{2(t-s)}}-\frac{(x+y)}{\sqrt{2\pi}(t-s)^{3/2}}e^{-\frac{(x+y)^{2}}{2(t-s)}}.
	\end{align*}
and using the linear growth bound \eqref{eq:lin_growth_bound} on $\tilde{b}^{i,n}$, it suffices to show that
\begin{equation}\label{eq:1st_density_bound}
\int_{0}^{t}\E \biggl[  \frac{1}{(t-s)} e^{-\frac{a(x-Z_{s}^{i,n})^2}{(t-s)}}(1+Z_{s}^{i,n}) \;\Big| \;\mathcal{{G}}_{T}^{i,n}\biggr]\ds < \infty,
	\end{equation}
    where $a>0$ is just some constant. We now change measure to $\Q^{i,n}$ from Lemma \ref{lem:girsanov}, using \eqref{eq:radon-nikodym_G-cond_1}. Conditionally on $\mathcal{{G}}_{T}^{i,n}$, the density $p_\Q^{i,n}(s,x)$ of $Z_s^{i,n}$ under $\Q^{i,n}$ is
	\[
    p_\Q^{i,n}(s,x)= \int_0^\infty N(x,s,y,0)\di  \hat{\mu}^{i}_0(y)
    \]
By Assumption \ref{assump:ic_assumptions} and the boundedness of $\sigma$ from below, there is a $c_0>0$ with $\int e^{c_0y^2}\di  \hat{P}^{i}_0(y)<\infty$. Exploiting $(x-y)^2\geq \frac{\theta}{1+\theta}x^2 - \theta y^2$ for any $\theta\geq0$, we can bound $N(x,s;y,0)$ with $\theta = 2c_0s$ to get
\begin{equation}\label{eq:initial_Aronson_bound}
p_\Q^{i,n}(s,x) \leq \frac{2}{\sqrt{2\pi s}} \int_0^\infty e^{c_0y^2}\di  \hat{\mu}^{i}_0(y) e^{-\frac{c_0}{1+2c_0s}x^2} \leq \frac{c_1}{\sqrt{s}} e^{-\frac{c_2}{1+s}x^2}
\end{equation}
for universal constants $c_1,c_2>0$. Next, after changing the measure to $\Q^{i,n}$ in \eqref{eq:1st_density_bound}, we apply H\"older's inequality  with exponents $1/p+1/q=1$ to separate out the (conditional) $L^p$ norm of the Radon--Nikodym derivative. Taking $p>1$ close enough to $1$, it then follows from \eqref{eq:p-bound_radon-nikodym} of Lemma \ref{lem:girsanov} that \eqref{eq:1st_density_bound} is bounded by a constant multiple of
		\begin{equation*}
	\int_{0}^{t}
\biggl( \int_0^\infty  \frac{1}{(t-s)^q} e^{-\frac{aq(x-z)^2}{(t-s)}}(1+z)^q  p^{i,n}_\Q(s,z) \dz\biggr)^{\!\frac{1}{q}} \ds.
	\end{equation*}
Using \eqref{eq:initial_Aronson_bound} and completing the square, we can bound the $\dz$-integral by a constant multiple of
\begin{equation*}
\frac{1}{(t-s)^q}\frac{1}{s^{1/2}}\int_0^\infty e^{-\frac{aq}{t-s}(x-z)^2}  e^{-bz^2} \dz \leq  \frac{\kappa(t-s)^{1/2}}{(t-s)^{q} s^{1/2}} e^{-\lambda  x^2}
\end{equation*}
for suitable constants $b,\kappa,\lambda>0$. Raising to the power $1/q$ and integrating in $s$, it follows from $\int_0^t(t-s)^{-r}s^{-u}\ds \simeq t^{1-r-u} $ for $0\leq r,u<1$ that
\begin{equation}\label{eq:bound_perturbed_terms}
\int_0^t\E \bigl[\mathbf{1}_{r<\tau_{Z}^{i,n}}\partial_{y}N(x,t;Z_{r}^{i,n},r)\tilde{b}^{i,n}(r,Z_r^{i,n})\mid\mathcal{{G}}_{T}^{i,n}\bigr] \hspace{1.5pt}\dr\leq \kappa^\prime e^{-\lambda^\prime x^2}
\end{equation}
for some constants $\kappa^\prime, \lambda^\prime$ that only depend on $T$, $q$, and \eqref{eq:initial_Aronson_bound}. In particular, \eqref{eq:F_ab_char} holds on $[0,t]$, so we can plug it into \eqref{eq:prob_approx_error}. It also confirms we can apply Fubini's theorem, so we get
\begin{align*}
&\Prob (Z_{t}^{i,n}\in(a,b),\;t<\tau_{Z}^{i,n}\mid\mathcal{{G}}_{t}^{i,n}) = \int_a^b \int_{0}^{\infty} N(x,t;y,0)\,\di  \hat{\mu}_{0}^{i}(y) \, \dx\\
&\qquad+ \int_{a}^{b} \int_0^t \int_0^\infty \partial_{y}N(x,t;z,r)\tilde{b}^{i,n}(r,z)  \hat{p}^{i,n}(r,z) \dz \dr \dx\\
&\qquad - \int_{a}^{b} \int_0^t N(x,t;0,r)  \gamma(r,\hat{{C}}_{r}^{i,n})\sigma(r,\hat{{A}}_{r}^{i,n}) \di  \E \bigl[ \ell_{r\land \tau_{Z}^{i,n}}^{0} (Z^{i,n}) \mid\mathcal{{G}}_{T}^{i,n} \bigr] \dx,
\end{align*}
where we are taking $\hat{p}^{i,n}(r,\cdot)$ to denote (the precise representative of) the density of $Z_r^{i,n}$ subject to $r<\tau_Z^{i,n}$ conditional on $\mathcal{G}^{i,n}_T$. Since this holds for all $(a,b)\subset\R^+$ and all $t\in[0,T]$, we deduce that
\begin{align}\label{eq:p-hat-density}
\hat{p}^{i,n}(t,x) &= \int_{0}^{\infty} N(x,t;y,0)\,\di  \hat{\mu}^i_0(y) + \int_0^t \int_0^\infty \partial_{y}N(x,t;z,r)\tilde{b}^{i,n}(r,z)  \hat{p}^{i,n}(r,z) \dz \,\dr \nonumber\\
&\quad - \int_0^t N(x,t;0,r)  \gamma(r,\hat{{C}}_{r}^{i,n})\sigma(r,\hat{{A}}_{r}^{i,n})\, \di  \E \bigl[ \ell_{r\land \tau_{Z}^{i,n}}^{0} (Z^{i,n}) \mid\mathcal{{G}}_{T}^{i,n} \bigr]
\end{align}
for all $t\in[0,T]$ and almost every $x\in(0,\infty)$. In particular, since the third term is non-positive, we obtain directly from \eqref{eq:initial_Aronson_bound} and \eqref{eq:bound_perturbed_terms} that
\begin{equation}\label{eq:p-hat_Aronson}
\hat{p}^{i,n}(t,x) \leq C t^{-1/2} e^{-cx^2}
\end{equation} for all $t\in(0,T]$ and a.e.~$x\in(0,\infty)$, for some $c,C>0$ that are uniform in $i=1,\ldots,n$ and $n\geq 1$.

With the bound \eqref{eq:p-hat_Aronson}, we can observe that, for any times $ \underline{r} < \overline{r}$ in $[0,t]$,
\begin{equation*}
\E \bigl[ \ell_{\overline{r}\land \tau_{Z}^{i,n}}^{0} (Z^{i,n}) -\ell_{\underline{r} \land \tau_{Z}^{i,n}}^{0} (Z^{i,n}) \mid\mathcal{{G}}_{T}^{i,n} \bigr]  \leq \liminf_{\varepsilon \rightarrow 0} \int_{\underline{r} }^{\overline{r}} \frac{1}{\varepsilon} \int_0^\varepsilon  \hat{p}^{i,n}(r,x) \dx \dr \leq C \int_{\underline{r} }^{\overline{r}} r^{-1/2} \dr.
\end{equation*}
Since $N(x,t;0,r)\leq (t-r)^{-1/2}$ and $\int_0^t (t-r)^{-1/2} r^{-1/2} \dr $ is finite, it follows that we can apply dominated convergence to check that the third term on the right-hand side of \eqref{eq:p-hat-density} is continuous in $x\in\R^+$ and $t\in(0,T]$. Furthermore, it is immediate from the Aronson estimate and dominated convergence that both the first and second term on the right-hand side of \eqref{eq:p-hat-density} are continuous in $x\in\R^+$ and $t\in(0,T]$. Therefore, we can conclude from \eqref{eq:p-hat-density} that $\hat{p}^{i,n}(t,x)$ is in fact continuous in $x\in\R^+$ and $t\in(0,T]$. Moreover, with continuity of $\hat{p}^{i,n}(r,x)$ up to the boundary in $x$, and a maximal singularity of $r^{-1/2}$ at $x=0$, we can confirm that
\begin{equation}\label{eq:expected_local_time}
\E \bigl[ \ell_{s \land \tau_{Z}^{i,n}}^{0} (Z^{i,n}) \mid\mathcal{{G}}_{T}^{i,n} \bigr]= \lim_{\varepsilon \rightarrow 0}\int_0^s \frac{1}{\varepsilon} \int_0^\varepsilon \hat{p}^{i,n}(r,x) \dx  \dr= \int_0^s\hat{p}^{i,n}(r,0) \dr.
\end{equation}
The first equality follows by using the generalised occupation time formula \cite[Ex.~1.13,~Ch.~6]{revuz_yor} and dominated convergence in $\varepsilon$ with respect to the expectation, justified as in the proof of Lemma \ref{lemma:DCT_epsilon} below. The second equality follows by first applying dominated convergence for the outer $\dr$-integral, due to $0\leq \hat{p}^{i,n}(r,\cdot) \leq Cr^{-1/2}$ for $r\in[0,t]$, and then using the uniform continuity in $x$ on right-neighbourhoods of 0. Utilising \eqref{eq:expected_local_time} in \eqref{eq:p-hat-density}, we obtain \eqref{eq:density_rep-1} with $\sigma(r,\hat{A}^{i,n}_r)\hat{p}^{i,n}(r,0)$ in place of $\sigma(r,\hat{A}^{i,n}_r)^2p^{i,n}(r,\hat{A}_r^{i,n})$ in the final term. Noting that \eqref{eq:p_and_p-hat} follows from \eqref{eq:conditional_sub_prob}--\eqref{eq:Lamperti_transform_Gamma}, we deduce \eqref{eq:density_rep-1}. Thus, it only remains to observe that (i) the continuity of $(t,x)\mapsto p^{i,n}(t,x)$ follows from the continuity of $(t,x)\mapsto \int_{\hat{A}^{i,n}_t}^x \frac{1}{\sigma(t,y)}\dy$ along with that of $(t,x)\mapsto \hat{p}^{i,n}(t,x)$ and (ii) we obtain the Aronson estimate \eqref{eq:particle_aronson} for $p^{i,n}$ from that of $\hat{p}^{i,n}$, which was established in \eqref{eq:p-hat_Aronson}, using the lower bound $\int_{\hat{A}^{i,n}_t}^x \frac{1}{\sigma(t,y)}\dy \geq c (x- \hat{A}^{i,n}_t)$ for all $t\geq0$ and $x\geq \hat{A}^{i,n}_t$, for some $c>0$.
\end{proof}

\begin{corollary}[$L^2$ control]\label{cor:L2-sqrt}
    For any $T>0$, there exists $C>0$ such that
    \begin{equation}\label{eq:L2-control}
    \bigl\Vert p^{i,n}(t,\cdot) \bigr\Vert_2^2 \leq \frac{C}{\sqrt{t}},\quad t\in(0,T],
    \end{equation}
    where $C$ can be chosen uniformly in $i=1,\ldots, n$ and $n\geq 1$.
\end{corollary}
\begin{proof}
This follows directly from \eqref{eq:density_rep-1}. Indeed,
\[
\int_0^\infty\Bigl(\int_0^\infty N(x,t,y,0)\di  \hat{\mu}_0^{i}(y)\Bigr)^2\dx = \int_0^\infty\!\int_0^\infty \!\int_0^\infty N(x,t,y,0)N(x,t,z,0)\dx \di \hat{\mu}_0^{i}(y) \di  \hat{\mu}_0^{i}(z),
\]
where the inner integral is bounded by a constant multiple of $t^{-1/2}$ for $t\in(0,T]$, and, due to \eqref{eq:lamperti_aronson}, the $L^2$-norms of the second and third terms of \eqref{eq:density_rep-1} are bounded by a constant for $t\in[0,T]$.
\end{proof}

We stress that the above methods can also be applied to the joint law of multiple particles. For example, let $\mathcal{G}^{i,i^\prime;n}$ be defined as in \eqref{eq:G_filtration} but with $j\notin \{ i,i^\prime\}$, where $i\neq i^\prime$. Then, the same arguments as in the proof of Theorem \ref{thm:perturb_density} yield the following observation.

\begin{prop}[Pairwise estimates]\label{prop:pairwise_bound}
    For any $T>0$, there exists $C,c>0$ such that
    \[
    \Prob\bigl(X^{i,n}_t\in \dx , \,X^{j,n}_t\in \dy \mid \mathcal{G}^{i,j;n}_T \bigr) \leq Ct^{-1} e^{-c\Vert(x,y) \Vert^2 } \dx \dy,\quad t\in(0,T],
    \]
    provided $i\neq j$, where $C,c$ are uniform in $i,j=1,\ldots,n$ and $n\geq 1$.
\end{prop}
\begin{proof}
Fix two indices $i\neq j$. Analogously to the paragraph following \eqref{eq:reduced_filtration}, consider the auxiliary system $\hat{\mathbf{X}}^{n,-(i,j)}$ where the $i$’th and $j$'th particles are fully reflected. We can then apply \eqref{eq:Lamperti_transform_Gamma} to each of them to get a characterisation analogous to \eqref{eq:conditional_sub_prob}, which we can now relate to a reflected Brownian motion in the positive quadrant by following the same arguments. Proceeding as in the proof of Theorem \ref{thm:perturb_density}, we thus obtain the desired bound on the conditional density.
\end{proof}

\section{Weak convergence}\label{sect:weak_convergence}

We can express the empirical measures \eqref{eq:emprical_measures} as
\begin{equation*}
P_{t}^{n}(E)=\frac{1}{n}\sum_{i=1}^{n}\mathbf{{1}}_{E}(X_{t}^{i,n})\quad \text{for}\quad  E\subseteq\R\cup\{\dagger \}.
\end{equation*}
Setting $I^n_t:=P_t^n(\{\dagger \})$, conservation of mass takes the form $I^n_t+P^n_t(\R)=1$. Note that $I^n$ evolves discontinuously upon the killing of the particles. It will often be convenient to work in the frame of the moving boundary. This amounts to considering the push-forward empirical measures
\begin{equation}\label{eq:shift_empirical}
{\mu}_{t}^{n}:=P_{t}^{n}\circ\theta_{A_{t}^{n}}^{-1},
\end{equation}
where $\theta_c$ is given by $\theta_{c}(\dagger):=\dagger$ and $\theta_{c}(x):=x+c$ for $x\in\R$. To fix notation, we let any finite measure $\mu$ on $\R$ act on continuous functions $f\in C(\R)$ by $\langle \mu, f \rangle := \int_\R f(x)\di \mu(x)$. 

\subsection{The empirical measure flows and tightness}

Let $\mu^n$ be given by \eqref{eq:shift_empirical}. As the local times are supported only on the set of times where the given particle is at the moving boundary, we have $X^{i,n}_{\tau^{i}} = A_{\tau^i}^n$ and hence
\[
\langle {\mu}_{t}^{n} , f \rangle = \frac{1}{n}\sum_{i=1}^n \mathbf{1}_{t<\tau^i}f(X^{i,n}_{t} - A^n_{t}) =\frac{1}{n}\sum_{i=1}^n f(X^{i,n}_{t\land \tau^{i}} - A_{t\land \tau^i}^n) - f(0)I_t^n.
\]
For a smooth function $\phi$, we can apply It\^o's formula to the first expression on the right-hand side of the above decomposition. In this way, we obtain the c\`adl\`ag dynamics
\begin{align}\label{eq:empirical_dynamics}
\di \langle\mu^{n}_t,\phi\rangle	&=\langle\mu_{t}^{n},\tfrac{{1}}{2}\sigma_{A^n}(t,\cdot)^{2}\partial_{xx}\phi\rangle\dt+\langle\mu_{t}^{n}, b_{A^n}(t,\cdot)\partial_{x}\phi\rangle\dt - \langle\mu_{t}^{n},\partial_{x}\phi\rangle\di A_{t}^{n} \nonumber 
 \\ &\quad + \frac{1}{n}\sum_{i=1}^n\int_0^{t\land \tau^i} \!\!\! \sigma_{A^n}(s,X^{i,n}_s) \partial_x \phi(X^{i,n}_s)\di B^i_s  + \frac{1}{2} \partial_{x}\phi(0)\frac{1}{n}\sum_{i=1}^{n}\di \ell_{t \land \tau^{i}}^{i,n}-\phi(0)\di I_{t}^{n}
\end{align}
for all $t\geq0$ and all test functions $\phi \in C_b^\infty(\R)$, where we have introduced the notation
\begin{equation*}
    b_{h} (t, x) = b(t, x+h(t)), \quad \sigma_{h}(t,x) = \sigma(t, x+h(t)),
\end{equation*}
for the drift and diffusion coefficients shifted by a given curve $t\mapsto h(t)$ in the spatial variable. We shall use this notation throughout. In deriving \eqref{eq:empirical_dynamics}, we have further exploited that the local times are supported on the set of times where the respective particles are at the boundary.  Note that the discontinuities in the evolution of $\mu^n_t$ are all concentrated in the last term of \eqref{eq:empirical_dynamics} driven by $I^n$.

We view each $\mu^n$ as a measure-valued c\`adl\`ag process which is an element
\[
\phi \mapsto (\langle \mu^n_t, \phi \rangle)_{t\in[0,T]},\quad \phi \in \mathcal{S},
\]
of the Skorokhod space $D_{\mathcal{S}^\prime}[0,T]$, equipped with Skorokhod's M1 topology as defined in \cite{ledger_M1}, where we recall that $\mathcal{S}^\prime$ denotes the space of tempered distributions on $\R$. Likewise, any other Skorokhod space we shall consider will be equipped with the M1 topology.

\begin{prop}[Tightness]\label{prop:tightness} The family $(\mu^n)_{n\geq 1}$ is tight on $D_{\mathcal{S}^\prime}[0,T]$.
\end{prop}
\begin{proof}
For any $s<t$, we have
\[
\langle\mu^{n}_t ,\phi\rangle - \langle\mu^{n}_s ,\phi\rangle = (Y^{n,\phi}_t-Y^{n,\phi}_s) +  \frac{1}{2} \partial_{x}\phi(0)\frac{1}{n}\sum_{i=1}^{n}(\ell_{t\land \tau^i}^{i,n}-\ell_{s\land \tau^i}^{i,n})-\phi(0)(I_{t}^{n}-I_{s}^{n}),
\]
for $\phi\in \mathcal{S}$, where $Y^{n, \phi}$ collects the first four terms of the right-hand side of \eqref{eq:empirical_dynamics}. By Jensen's inequality,  the Burkholder--Davis--Gundy inequality, and Theorem~\ref{thm:perturb_density}, we get that $\E[(Y^{n,\phi}_t-Y^{n,\phi}_s)^4]=O(|t-s|^2)$ uniformly in $n\geq 1$ as $|t-s|\rightarrow 0$. Moreover, representing the local time of the shifted particle $X^{i,n}_t-A^{n}_t$ through Skorokhod's map, one can exploit the Lipschitz property of the latter \cite[Lemma 1.1.1]{reflected_SDEs_book} to show that we have the control $\E[(\ell_{t\land \tau^i}^{i,n}-\ell_{s\land \tau^i}^{i,n})^4]=O(|t-s|^2)$, uniformly in $n\geq 1$, by estimating this quantity using the particle dynamics and the Gr\"onwall and Burkholder--Davis--Gundy inequalities. Finally, $I^n$ is bounded in $n\geq 1$ and has increasing sample paths, so, by the same reasoning as \cite[Proposition 4.1]{ledger_M1}, tightness follows from the criterion for tightness in the M1 topology in \cite[Proposition 4.2]{ledger_M1}.
\end{proof}

By the same reasoning as in the above proof, the c\`adl\`ag processes $Z^{n,
	\phi}$ given by
\[
Z^{n,\phi}_t:= \frac{1}{2}\partial_{x}\phi(0) \frac{1}{n}\sum_{i=1}^{n}\ell_{t\land \tau^i}^{i,n}-\phi(0)I_{t}^{n},
\]
for $\phi \in C^1(\R)$, are tight on $D_{\R}[0,T]$ equipped with the M1 topology. In particular, by \cite[Thm.~3.2]{ledger_M1}, $Z^n$ given by $\phi \mapsto (Z^{n,\phi}_t)_{t\in[0,T]}$ is tight on $D_{\mathcal{S}^\prime}[0,T]$ together with $\mu^n$. We stress that also $I^n$ is tight on $D_{\R}[0,T]$ and note that $Z^{n,
	-\mathbf{1}}=I^n$, where $ \mathbf{1}$ denotes the function that is constantly  $1$. We shall use this notation throughout the paper. Since M1-convergence implies convergence of the marginals outside a co-countable set of times, it follows from dominated convergence that, when $I^n$ tends to $I^*$, also $A^n$ and $C^n$ tend to their respective limits $A^*$ and $C^*$ defined in terms of $I^*$, as in \eqref{eq:weak_A_C_I}. It suffices to keep track of $(\mu^n,Z^n)$, due to $Z^{n,
	-\mathbf{1}}=I^n$, but we write $(\mu^n,I^n,Z^n)$ since we will be working explicitly with $I^n$ throughout. The triple $(\mu^n,I^n,Z^n)$ collects the empirical measure flow, its loss of mass, and its boundary condition.

By Burkholder--Davis--Gundy, the Brownian term in \eqref{eq:empirical_dynamics} vanishes uniformly in probability as $n\rightarrow \infty$. Since we have exponential decay in the sense of $\sup_{n\geq 1}\E[\langle \mu^n_t,e^{c (\cdot)}\rangle]<\infty$ for some $c>0$ by Theorem~\ref{thm:perturb_density}, we can upgrade the continuity of the marginal projections on $D_{\mathcal{S}^\prime}[0,T] $ (see \cite[Prop.~2.7]{ledger_M1}) to have weak convergence for each of the terms involving $\mu^n$ in \eqref{eq:empirical_dynamics}. Thus, we can deduce that any limit point $(\mu^*,I^*,Z^*)$ of $(\mu^n,I^n,Z^n)_{n\geq1}$ satisfies
\begin{equation}\label{eq:weak_limit_point}
\di \langle\mu^{*}_t,\phi\rangle	=\langle\mu_{t}^{*},\tfrac{{1}}{2}\sigma_{A^*}(t,\cdot)^{2}\partial_{xx}\phi\rangle\dt+\langle\mu_{t}^{*}, b_{A^*}(t,\cdot)\partial_{x}\phi\rangle\dt - \langle\mu_{t}^{*}, \partial_{x}\phi\rangle\di A_{t}^{*} + \di Z^{*,\phi}_t,
\end{equation}
for every $\phi \in C^\infty_b(\R)$, where the convergence of the integrals against $A^n$ is, e.g., ensured by \cite[Thm.~2.6]{sojmark_wunderlich}. Moreover, we have that $\mu^*$ is a sub-probability measure with $I^*_t+\mu^*_t([A^*_t,\infty))=1$ and we have that $A^*$ is defined as in \eqref{eq:weak_A_C_I} in terms of $I^*$. From here, the main work lies in deriving an appropriate characterisation of $Z^{*,\phi}$, in order to ultimately recover the weak formulation \eqref{eq:the_weak_formulation}--\eqref{eq:weak_A_C_I} of Definition \ref{def:weak_formulation}.

\subsection{Characterizing the limiting boundary condition}\label{subsect:characterise_limit_BC}

We start by defining the filtrations
\begin{equation}\label{eq:Filtration_F_bar}
\mathcal{\bar{{F}}}_{t}^{n}:=\sigma(X_{0}^{j},B_{s}^{j},\{s<\tau^j\}:s\in[0,t],\,j\in\{1,\ldots,n\}).
\end{equation}
These contain enough information that, for each $n\geq 1$, the processes $\langle \mu^n, \phi \rangle$ and $Z^{n,\phi}$ are adapted to $(\mathcal{\bar{{F}}}_{t}^{n})_{t\geq 0}$, for any $\phi\in C_b^\infty(\R)$. In particular, also $I^n$, $A^n$, and $C^n$ are adapted to $(\mathcal{\bar{{F}}}_{t}^{n})_{t\geq 0}$. What these filtrations do \emph{not} reveal, however, are the conditional killing probabilities. This makes it possible to uncover the following martingale property of the $Z^n$.

\begin{prop}[Martingale structure]\label{prop:In_key_martingale_result} 
	For any $\phi \in C^1(\R)$, the limit points of $(Z^{n,\phi})_{n\geq 1}$ are continuous processes. Moreover, let $\psi_{\varepsilon}$ be the probability density function of a reflected
	Brownian motion on the positive half-line at time $\varepsilon>0$, and let $H$ be a bounded continuous process $H$ adapted to the filtration $(\mathcal{\bar{{F}}}_{t}^{n})_{t\geq0}$ from \eqref{eq:Filtration_F_bar}. Then the difference
    \begin{equation}\label{eq:vanishing_difference}
 \int_0^t H_s \di Z^{n,\phi}_s-\lim_{\varepsilon\downarrow0}\int_{0}^{t}H_s\bigl(\tfrac{1}{2}\partial_x\phi(0)-\phi(0)\gamma(s,C_{s}^{n})\bigr)\langle\mu_{s}^{n},\sigma_{A^n}(s,\cdot)^{2}\psi_{\varepsilon}\rangle \ds 
    \end{equation}
is a martingale for $(\mathcal{\bar{{F}}}_{t}^{n})_{t\geq0}$ and it vanishes uniformly on compacts in $L^2(\Omega)$ as $n\rightarrow \infty$.
\end{prop}
\begin{proof}
From \cite[Theorem~2.5]{fausti_soj_22} we have that
\begin{equation}\label{eq:infected_martingale}
M^n_t := I^n_t - \frac{1}{n}\sum_{i=1}^{n}\int_{0}^{t}\mathbf{1}_{s<\tau^{i}} \gamma(s,C_{s}^{n})\di \ell_{s}^{i,n}
\end{equation}
is a martingale that vanishes uniformly on compacts in $L^2(\Omega)$ as $n\rightarrow \infty$. As in the proof of Proposition \ref{prop:tightness}, we have an estimate $\E[(\ell_{t\land \tau^i}^{i,n}-\ell_{s\land \tau^i}^{i,n})^4]\leq C|t-s|^2$ where the constant $C$ is uniform in $i,n\geq1$, so Kolmogorov's continuity criterion allows us to confirm that the limit of the second term on the right-hand side of \eqref{eq:infected_martingale} is supported on the space of continuous paths. In view of $M^n$ vanishing, this yields the first claim. Next, we can write
\begin{equation}\label{eq:H_infected_martingale}
\int_0^t H_s \di Z^{n,\phi}_s=\frac{1}{n}\sum_{i=1}^n\int_0^{t\land \tau^i} \!\!   \bigl(  \tfrac{1}{2} \partial_x\phi(0)-\phi(0)\gamma(s,C^n_s)\bigr)H_s \di \ell^{i,n}_s - \phi(0)\int_0^t H_s \di M^n_s.
\end{equation}
Since $M^n$ is a martingale, so is the second term on the right-hand side of \eqref{eq:H_infected_martingale}. Moreover,  this term then vanishes uniformly on compacts in $L^2(\Omega)$ as $n \rightarrow \infty$ by the same reasoning as in the proof of \cite[Theorem~2.5]{fausti_soj_22}. To confirm that \eqref{eq:vanishing_difference} is a martingale, we show that the first term on the right-hand side of \eqref{eq:H_infected_martingale} equals the second term of the difference \eqref{eq:vanishing_difference}. To this end, let $\ell^{i;x}$ denote the local time \eqref{eq:local_time_defn} of $X^{i,n}$ defined along $t\mapsto A^n_{t}+x$, for any $x\geq0$. In particular, $\ell^{i;0}=\ell^{i,n}$. Applying the generalised occupation time formula \cite[Ex.~1.13,~Ch.~6]{revuz_yor}, we then get
	\[
\int_{0}^{\infty}\int_{0}^{t\land\tau^{i}}\!\!\!H_s \gamma(s,C_{s}^{n})\psi_{\varepsilon}(x)\di \ell_{s}^{i;x} \hspace{1pt} \dx=\int_{0}^{t}H_s\gamma(s,C^n_{s})\mathbf{{1}}_{s<\tau^{i}}\sigma(s,X_{s}^{i,n})^{2}\psi_{\varepsilon}(X_{s}^{i,n}-A^n_s) \ds
	\]
	and hence
\begin{equation}\label{eq:local_time_to_empirical}
	\frac{1}{n}\sum_{i=1}^{n}\int_{0}^{\infty}\int_{0}^{t\land\tau^{i}}\!\!\!H_s\gamma(s,C_{s}^{n})\psi_{\varepsilon}(x)\di \ell_{s}^{i;x} \dx=\int_{0}^{t}H_s\gamma(s,C^n_{s})\langle\mu_{s}^{n},\sigma_{A^n}(s,\cdot)^{2}\psi_{\varepsilon}\rangle \ds
	\end{equation}
	for any $\varepsilon>0$. 
   By working with the approximating Riemann sums, controlling the associated error as for $S_2(n,\varepsilon)$ in the proof of \cite[Prop.~4.2]{fausti_soj_22}, and relying on the right-continuity of each $x\mapsto \ell^{i;x}$, it follows by arguing as in the aforementioned proof that
	\[
\int_{0}^{\infty}\int_{0}^{t\land\tau^{i}}H_s\gamma(s,C^n_{s})\psi_{\varepsilon}(x) \di \ell_{s}^{i;x} \dx \rightarrow \int_{0}^{t\land\tau^{i}}H_s\gamma(s,C^n_{s}) \di  \ell^{i,n}_{s}
	\]
    almost surely as $\varepsilon \rightarrow 0$. From this and \eqref{eq:local_time_to_empirical}, we in turn conclude that
    \begin{equation}\label{eq:limit_local_time}
\lim_{\varepsilon\downarrow0}\int_{0}^{t}H_s \gamma(s,C_{s}^{n})\langle \mu_{s}^{n},\sigma_{A^n}(s,\cdot)^{2}\psi_{\varepsilon}\rangle \ds = \frac{1}{n}\sum_{i=1}^{n}\int_{0}^{t\land \tau^{i}} H_s \gamma(s,C_s^{n}) \di  \ell_{s}^{i,n}.
	\end{equation}
	Of course, the same arguments apply also without $\gamma(s,C^n_s)$ in the integrand, so we get that \eqref{eq:vanishing_difference} equals the second term on the right-hand side of \eqref{eq:H_infected_martingale}. Since the latter was shown to be a martingale with respect to $(\mathcal{\bar{{F}}}_{t}^{n})_{t\geq0}$, this completes the proof.
\end{proof}

The next lemma lets the limit in $\varepsilon$ pass through the expected values of \eqref{eq:vanishing_difference}.
\begin{lemma}\label{lemma:DCT_epsilon}
	Let everything be as in Proposition \ref{prop:In_key_martingale_result}. Then we have
	\[
	\E\Bigl[\int_{0}^{t}H_s \di  Z_{s}^{n,\phi}\Bigr]=  \lim_{\varepsilon\downarrow0}\E\Bigl[\int_{0}^{t}H_s\bigl(\frac{1}{2} \partial_x\phi(0) - \phi(0)\gamma(s,C_{s}^{n})\bigr) \langle\mu_{s}^{n},\sigma_{A^{n}}(s,\cdot)^{2}\psi_{\varepsilon}\rangle \ds\Bigr].
	\]
\end{lemma}

\begin{proof}
	Fix $\varepsilon>0$. Recalling the equality in \eqref{eq:local_time_to_empirical}, we can exploit the monotonicity of the local times along with the boundedness of $\gamma$ and $H$ to see that
	\begin{align*}
		\Bigl|\int_{0}^{t}H_s\bigl( \tfrac{1}{2}\partial_x\phi(0)-&\phi(0)\gamma(s,C_{s}^{n})\bigr)\langle\mu_{s}^{n},\sigma_{A^n}(s,\cdot)^{2}\psi_{\varepsilon}\rangle \ds\Bigr|  \leq\frac{c}{n}\sum_{i=1}^{n}\int_{0}^{\infty}\int_{0}^{t\land\tau^i}\psi_{\varepsilon}(x) \di  \ell_{s}^{i;x} \dx\\
		& = \frac{c}{n}\sum_{i=1}^{n}\int_{0}^{t\land\tau^i}\!\!\!\!\int_{0}^{\infty}\psi_{\varepsilon}(x) \dx \di \ell_{s}^{i;x} \leq\frac{c}{n}\sum_{i=1}^{n}\sup_{x\geq0}\ell_{t}^{i;x},
	\end{align*}
for a constant $c>0$, where the last line used that $\psi_\varepsilon$ integrates to 1. This bound is independent of $\varepsilon>0$, and we have that $\sup_{x\geq0}\ell_{t}^{i;x}$ is finite in expectation by Barlow--Yor's BDG type inequality for the local times of continuous semimartingales \cite[Page~199]{barlow_yor} (see also \cite[Ch.~XI,~Thm.~2.4]{revuz_yor}). Using that the difference \eqref{eq:vanishing_difference} is a martingale, we can thus take expectations in \eqref{eq:vanishing_difference} and apply dominated convergence as $\varepsilon \downarrow 0$ to obtain the desired result.
\end{proof}

We now obtain the following characterisation of the boundary condition for \eqref{eq:weak_limit_point}. We shall write $Z^n$ for the process given by the assignment $\phi \mapsto Z^{n,\phi}$, and likewise in the limit.

\begin{prop}\label{prop:f_integral_converge} Consider any given limit point $(\mu^*,I^*,Z^*)$ of $(\mu^n,I^n,Z^n)$.  Let $f: [0,T]\times D_{\R}[0,T]\times D_{\mathcal{S}^\prime}[0,T] \rightarrow \R$ be a bounded function that is continuous with respect to uniform convergence. Setting $f(s):=f(s,I_{\cdot \land s}^{*},\mu^*_{\cdot \land s})$ we have
	\[
\E \Bigl[\int_{0}^{t}f(s) \di  Z_{s}^{*,\phi}\Bigr]=\lim_{\varepsilon\downarrow0}\E \Bigl[\int_{0}^{t}f(s)\bigl(\tfrac{1}{2} \partial_x\phi(0)-\phi(0)\gamma(s,C_{s}^{*})\bigr)\langle\mu_{s}^{*},\sigma_{A^{*}}(s,\cdot)^{2}\psi_{\varepsilon}\rangle \ds\Bigr]
	\]
    for any $\phi \in C^\infty(\R)$.
\end{prop}

\begin{proof}
Set $f_n(s):=f(s,I_{\cdot \land s}^{n},\mu^n_{\cdot \land s})$. By Proposition \ref{prop:In_key_martingale_result} we can apply \cite[Theorem 3.6]{sojmark_wunderlich} and uniform integrability to confirm that
	\begin{equation}\label{eq:integral_convergence}
			\E \Bigl[\int_{0}^{t}f(s) \di  Z_{s}^{*,\phi}\Bigr] =\lim_{n\rightarrow\infty}\E \Bigl[\int_{0}^{t}f_n(s) \di  Z_{s}^{n,\phi}\Bigr].
	\end{equation}
Set also $f_{n,\phi}(s):=f(s,I_{\cdot \land s}^{n},\mu^n_{\cdot \land s})(\tfrac{1}{2}\partial_x\phi(0)-\phi(0)\gamma(s,C_{s}^{n}))$. By \eqref{eq:integral_convergence} and Lemma~\ref{lemma:DCT_epsilon}, we then have
\begin{equation}\label{eq:eps_then_n}
\E \Bigl[\int_{0}^{t}f(s) \di  Z_{s}^{*,\phi}\Bigr] = \lim_{n\rightarrow \infty} \lim_{\varepsilon \downarrow 0 }\E \left[\int_{0}^{t}f_{n,\phi}(s)\langle \mu_{s}^{n},\sigma_{A^{n}}(s,\cdot)^{2}\psi_{\varepsilon}\rangle \ds\right] .
\end{equation}
	Fix an arbitrary $\delta>0$ and write
	\begin{align}
&\E \left[\int_{0}^{t}f_{n,\phi}(s)\langle \mu_{s}^{n},\sigma_{A^{n}}(s,\cdot)^{2}\psi_{\varepsilon}\rangle \ds\right]  =\E \left[\int_{0}^{t}f_{n,\phi}(s)\langle \mu_{s}^{n},\sigma_{A^{n}}(s,\cdot)^{2}\psi_{\delta}\rangle \ds\right]\nonumber \\
		& \qquad \qquad \qquad \qquad +\E \left[\int_{0}^{t}f_{n,\phi}(s)\langle \mu_{s}^{n},\sigma_{A^{n}}(s,\cdot)^{2}(\psi_{\varepsilon}-\psi_{\delta})\rangle \ds\right].\label{eq:uni_cont_limit}
	\end{align}
 Set
\[
\hat{f}_{i,n,\phi}(t):=f(t,\hat{I}^{i,n}_t,\hat{\mu}_t^{i,n})(\tfrac{1}{2} \partial_x\phi(0)-\phi(0)\gamma(t,\hat{C}_{t}^{n})),\quad \hat{\mu}_t^{i,n}:=\frac{1}{n}\sum_{j\neq i} \mathbf{1}_{t<\tau^j}\delta_{\hat{X}^{j,n,(-i)}_t}.
\]
For each $i=1,\ldots,n$, we can then observe that
	\begin{align*}
&\E \left[\mathbf{1}_{s<\tau^{i}}f_{n,\phi}(s)\sigma_{A^{n}}(s,X_{s}^{i,n}-A^n_s)^{2}(\psi_{\varepsilon}(X_{s}^{i,n}-A^n_s)-\psi_{\delta}(X_{s}^{i,n}-A^n_s))\right]
        \\ & \quad	=\E \left[\hat{f}_{i,n,\phi}(s)\E \bigl[\mathbf{1}_{s<\tau^{i}}\sigma(s,X_{s}^{i,n})^{2}\bigl(\psi_{\varepsilon}(X_{s}^{i,n}-\hat{A}^{i,n}_s)-\psi_{\delta}(X_{s}^{i,n}-\hat{A}^{i,n}_s)\bigr)\mid\mathcal{{G}}_{T}^{i,n}\bigr]\right] 
        \\
        &\quad =\E \left[\hat{f}_{i,n,\phi}(s) \int_{\hat{A}_s^{i,n}}^\infty 
        \sigma(s,x)^{2}\bigl(\psi_{\varepsilon}(x-\hat{A}^{i,n}_s)-\psi_{\delta}(x-\hat{A}^{i,n}_s)\bigr)p^{i,n}(s,x) \dx
        \right]  ,
	\end{align*}
	by the definition of $\mathcal{{G}}_{T}^{i,n}$ in \eqref{eq:G_filtration}, where the last equality follows from Theorem~\ref{thm:perturb_density}. In addition to the continuity of $p^{i,n}$ in Theorem~\ref{thm:perturb_density}, we can verify from \eqref{eq:lamperti_aronson} and \eqref{eq:density_rep-1} that, for any $\theta >0$, there exists $c(\theta)>0$ so that $|\sigma(s,x)^{2}p^{i,n}(s,x)-\sigma(s,y)^{2}p^{i,n}(s,y)|\leq \theta/2$, for all $x,y\in[\hat{A}_s^{i,n}, \hat{A}_s^{i,n}+c(\theta)]$, uniformly in $n\geq 1$. Thus, we can estimate
    \begin{equation*}   \Bigl| \int_{\hat{A}_s^{i,n}}^\infty 
\sigma(s,x)^{2}p^{i,n}(s,x)\bigl(\psi_{\varepsilon}(x-\hat{A}^{i,n}_s)-\psi_{\delta}(x-\hat{A}^{i,n}_s)\bigr) \dx \Bigr| \leq \theta + \kappa \bigl(e^{-c(\theta)^2/2\varepsilon} +e^{-c(\theta)^2/2\delta}\bigr),  
    \end{equation*}
for a constant $\kappa >0$, uniformly in $n\geq 1$, by splitting the integral on $[\hat{A}_s^{i,n}, \hat{A}_s^{i,n}+c(\theta)]$ and its complement. Consequently, 
    \[
\limsup_{\varepsilon,\delta \downarrow 0}\,\limsup_{n\rightarrow \infty}\,\biggl| \,\E \left[\int_{0}^{t}\hat{f}_{i,n,\phi}(s)\langle \mu_{s}^{n},\sigma_{A^{n}}(s,\cdot)^{2}(\psi_{\varepsilon}-\psi_{\delta})\rangle \ds\right] \biggr| \leq \theta,
\]
for any $\theta>0$, and hence the right-hand side is zero. From this and \eqref{eq:eps_then_n}--\eqref{eq:uni_cont_limit}, we can now conclude that
\begin{align}
\E \Bigl[\int_{0}^{t}f(s) \di  Z_{s}^{*,\phi}\Bigr] &= \lim_{\delta\downarrow 0} \lim_{n\rightarrow \infty} \E \left[\int_{0}^{t}f_{n,\phi}(s)\langle \mu_{s}^{n},\sigma_{A^{n}}(s,\cdot)^{2}\psi_{\delta}\rangle \ds\right] \nonumber\\
&=\lim_{\delta\downarrow 0} \E \left[\int_{0}^{t}f(s)\bigl(\tfrac{1}{2} \partial_x\phi(0)-\phi(0)\gamma(s,C_{s}^{*})\bigr)\langle \mu_{s}^{*},\sigma_{A^{*}}(s,\cdot)^{2}\psi_{\delta}\rangle \ds\right]\label{eq:delta_to_zero}
\end{align}
which completes the proof.
\end{proof}

For any given limit point $(\mu^*,I^{*})$ on $D_{\mathcal{S}^\prime}[0,T] \times D_{\R}[0,T]$, let $(\mathcal{F}^*_t)_{t\geq0}$ be the corresponding filtration defined by
\begin{equation}\label{eq:limit_filtration}
\mathcal{F}^*_t := \sigma\bigl(\langle \mu^*_s, \phi \rangle ,I^*_s : s\in[0,t], \phi \in \mathcal{S}\bigr).
\end{equation}
\begin{corollary} For any $\phi\in C^1(\R)$ and any bounded continuous process $H$ adapted to $(\mathcal{F}^*_t)_{t\geq 0}$, we have
	\[
	\E \Bigl[\int_{0}^{t}H_s  \di  Z_{s}^{*,\phi}\Bigr]=\lim_{\delta\downarrow0}\E \Bigl[\int_{0}^{t}H_s\bigl( \tfrac{1}{2}\partial_x\phi(0)-\phi(0)\gamma(s,C_{s}^{*})\bigr)\langle\mu_{s}^{*},\sigma_{A^{*}}(s,\cdot)^{2}\psi_{\delta}\rangle \ds\Bigr].
	\]
	\end{corollary}
\begin{proof}
	Take a sequence of bounded continuous functions $f_m$ to which Proposition \ref{prop:f_integral_converge} applies. Let these be such that $f_m(t,I^*_{\cdot \land t},\mu^*_{\cdot \land t})$ converges to $H$ as $m\rightarrow \infty$. In turn, $f_m(t,I^n_{\cdot \land t},\mu^n_{\cdot \land t})$ converges to $H$ as first $n \rightarrow \infty$ and then $m\rightarrow \infty$. Repeating the proof of Proposition \ref{prop:f_integral_converge}, we are able to send first $n\rightarrow \infty$ and then $m\rightarrow \infty$ in the first term of the corresponding analogue of \eqref{eq:uni_cont_limit}, before finally sending $\delta \rightarrow 0$, as in  \eqref{eq:delta_to_zero}. In this way, the claim follows. 
\end{proof}

In view of all of the above, we have shown the following.

\begin{prop}[Randomized measure-valued weak solution]\label{prop:rand_weak_solution} Let $(\mu_t,I_t,Z_t)_{t\in[0,T]}$ be a limit point of $(\mu^n_t,I^n_t,Z^n_t)_{t\in[0,T]}$. Then, $\mu=(\mu_t)_{t\in[0,T]}$ is a stochastic process taking values in $\mathcal{M}_{\leq 1}(\R^+)$ with $\mu_0=P_0\circ \theta_{a_0}^{-1}$ and it satisfies
\begin{equation}\label{eq:weak_form_1}
\di \langle\mu_t,\phi\rangle	=\langle\mu_{t},\tfrac{{1}}{2}\sigma_{A}(t,\cdot)^{2}\partial_{xx}\phi\rangle\dt+\langle\mu_{t}, (b_{A}(t,\cdot)-A^\prime_{t})\partial_{x}\phi\rangle\dt + \di Z^{\phi}_t,
\end{equation}
for all $t\in[0,T]$ and all $\phi \in C^{2}_b(\R)$, almost surely, with 
\begin{equation}\label{eq:bc_as_limit}
\E\biggl[\int_0^t H_s \di Z_t^{ \phi}\biggr]
= \lim_{\delta \downarrow 0}  \E\biggl[\int_0^t H_s \bigl( \tfrac{1}{2} \partial_x \phi(0) - \phi(0) \gamma(s, C_s) \bigr)\langle \mu_s, \sigma_A(s, \cdot)^2 \psi_\delta  \rangle \ds \biggr],
\end{equation}
for all $t\in[0,T]$, all $\phi\in C_b^2(\R)$, and
all bounded continuous processes $H$ adapted to the filtration \eqref{eq:limit_filtration} generated by $(\mu,I)$, where
\begin{equation}\label{eq:randomized_A_C_I}
A_t = a_0 + \alpha\int_{t-\bar{d}}^t \varrho(t-s) I_{s} \ds,\quad C_t = \int_{t-\bar{d}}^t \varrho(t-s) (I_{s} - I_{s-\bar{d
}})\ds,\quad I_t= 1- \mu_t(\R^+).
\end{equation}
\end{prop}

Note that taking $\phi = -\mathbf{1}$ in \eqref{eq:weak_form_1} yields $I=Z^{-\mathbf{1}}$, in agreement with the finite particle setting. Note also that there need not be differentiability of $I$ and $A$ in the above. However, analogously to \eqref{eq:A_is_AC}, we can derive from \eqref{eq:randomized_A_C_I} that  $A$ is at least absolutely continuous with 
\begin{equation}\label{eq:A_weak_derivative}
A_t =  a_0 + \int_0^t A^\prime_s \di s,\quad A^\prime_s = 
\alpha\int_{s-\bar{d}}^s \varrho (s-r )\di  I_r.
\end{equation}
This is the sense in which $A^\prime$ in \eqref{eq:weak_form_1} should be understood.

\section{Decoupled mean-field representation}\label{sect:decoupled_mf}

Throughout this section, we fix a limit point $(\mu,I,Z)$ where $\mu=(\mu_t)_{t\in[0,T]}$ is then a randomized weak solution in the sense of Proposition \ref{prop:rand_weak_solution}. 

\subsection{Decoupled mean-field particle}

Given $\mu$, we introduce a decoupled particle $X$, which is the natural mean-field analogue of the finite particle system, except that we let the corresponding front $A$ and current contagiousness $C$ be specified exogenously in terms of $I^\mu_t:=1-\mu_t(\R^+)$. 

Let $\mathcal{F}^\mu$ denote the filtration generated by $I^\mu$. The fully reflected `decoupled' mean-field particle is defined by
\begin{equation}\label{eq:decoupled_reflect}
\left\{ \begin{array}{@{}l@{}l}
	\di X_t = b(t,X_t)\dt + \sigma(t,X_t)\di B_t + \tfrac{1}{2} \di \ell_t, & \quad t \in[0,T], \vspace{3pt}\\
	A_t = a_0 + \alpha \int_{t-\bar{d}}^t \varrho(t-s) I_s^\mu \ds , & \quad t \in[0,T],
	\end{array} \right.
	\end{equation}
    where $\ell$ denotes the local time of $X$ along $A$ such that $X_t $ lives in $[A_t,\infty)$. Here $B$ is a Brownian motion, $X_0\sim P_0$, and both are taken to be independent of each other as well as being independent of $\mathcal{F}^\mu$. Existence of a solution $X$ adapted to the filtration $\mathcal{F}^\mu_t \lor \sigma(X_0, B_s:s\leq t)$ can be easily confirmed, since the boundary $A$ in \eqref{eq:decoupled_reflect} is given exogenously and is absolutely continuous with bounded derivative, as in \eqref{eq:A_is_AC}. Next, we let $\chi$ be a standard Exponential random variable independent of $X_0$, $B$, and $\mathcal{F}^\mu$, and we then define the decoupled infection time
\begin{equation*}
    \tau:= \inf \Bigl\{ t\geq0 : \int_0^t \gamma(s,C_s) \di \ell_s   \geq \chi \Bigr\}
    \end{equation*}
    with $C_t=\int_{t-\bar{d}}^t \varrho(t-s) (I^{\mu}_{s} - I^{\mu}_{s-\bar{d}})\ds$. Note that $\tau$ only depends on $X$ through the local time $\ell$, as $C$ is given exogenously and $\chi$ is independent of $\mathcal{F}^\mu_t \lor \sigma(X_0, B_s:s\leq t)$. 
    
    Consider the $\mathcal{F}^\mu$-adapted flow of (random) sub-probability measures $(P_t)_{t\in[0,T]}$ defined by
\[
\langle P_t ,\phi \rangle=\E[\phi(X_t) \mathbf{1}_{t<\tau}\mid \mathcal{F}^\mu_T].
\]
Since $B$ and $X_0$ were taken to be independent of $\mathcal{F}^\mu_T$, we can observe that, for any $\mathcal{F}^\mu$-adapted test functions $(\phi(t,\cdot))_{t\in[0,T]}$, we have
\begin{align}
\langle P_t , \phi \rangle &= \E[\E[\phi(t,X_t)\mathbf{1}_{t<\tau}\mid \mathcal{F}^\mu_T \lor \sigma(X_0, B_s :s\leq t) ] \mid \mathcal{F}^\mu_T] \nonumber \\
& = \E[\phi(t,X_t) \Prob(t<\tau\mid \mathcal{F}^\mu_T \lor \sigma(X_0, B_s :s\leq t) ) \mid \mathcal{F}^\mu_T] \nonumber \\
& = \E[\phi(t,X_t) e^{-\int_0^t \gamma(s,C_s)\di \ell_s} \mid \mathcal{F}^\mu_T]. \label{eq:rewrite_nu}
\end{align}
Taking also $\phi(t,x)$ to be $C^{1,2}$, an application of It\^o's formula in \eqref{eq:rewrite_nu} gives
\begin{equation}\label{eq:lin_weak_formulation_P}
\di  \langle P_t , \phi(t,\cdot) \rangle = \bigl\langle P_{t},\tfrac{{1}}{2}\sigma_{A}(t,\cdot)^{2}\partial_{xx}\phi(t,\cdot) + b_{A}(t,\cdot)\partial_{x}\phi(t,\cdot)+\partial_t\phi(t,\cdot) \bigr\rangle\dt + \di Z^{P,\phi}_t
\end{equation}
where
\begin{align}
Z^{P, \phi}_t &= \E\Bigl[ \int_0^t \bigl(\tfrac{1}{2}\partial_x\phi(s,A_s)- \gamma(s,C_s)\phi(s,A_s)\bigr)e^{-\int_0^s \gamma(r,C_r) \di \ell_r} \di \ell_s  \mid \mathcal{F}^\mu_T\Bigr] \nonumber \\
& = \E\Bigl[ \int_0^t \bigl(\tfrac{1}{2} \partial_x\phi(s,A_s)- \gamma(s,C_s)\phi(s,A_s)\bigr)\Prob\bigl(s < \tau \mid \mathcal{F}^\mu_T\lor \sigma(X_0,B_r:r\leq t)\bigr) \di \ell_s  \mid \mathcal{F}^\mu_T\Bigr] \nonumber \\
& =\E\Bigl[ \int_0^t \bigl(\tfrac{1}{2}\partial_x\phi(s,A_s)- \gamma(s,C_s)\phi(s,A_s)\bigr)\mathbf{1}_{s<\tau}\di \ell_s  \mid \mathcal{F}^\mu_T\Bigr]. \label{eq:BC_P}
\end{align}
In \eqref{eq:lin_weak_formulation_P} we have applied the conditional Fubini theorem, and in \eqref{eq:BC_P} we have again applied the same reasoning as in \eqref{eq:rewrite_nu} together with the fact that $\ell$ is supported on $\{t\geq0 :X_t=A_t\}$.

To compare with $\mu$ from Proposition \ref{prop:rand_weak_solution}, we shall examine $P$ in the frame of the moving boundary $A$. As in \eqref{eq:shift_empirical}, we consider $\theta_{c}(x)=x+c$ and introduce $(\nu_t)_{t\in[0,T]}$ defined by
\begin{equation}\label{eq:defn_nu}
\nu_t = P_t \circ \theta^{-1}_{A_{t}}.
\end{equation}

\begin{prop}[Decoupled weak solution]\label{prop:linearised_problem_nu}
    Let $A$ and $C$ be defined in terms of a given random weak solution $\mu$ according to Proposition \ref{prop:rand_weak_solution}. Then, the corresponding $\mathcal{M}_{\leq 1}(\R^+)$-valued stochastic process $(\nu_t)_{t\in[0,T]}$ defined in \eqref{eq:defn_nu} satisfies
    \begin{align}
\di \langle\nu_t,\phi\rangle	&=\langle\nu_{t},\tfrac{{1}}{2}\sigma_{A}(t,\cdot)^{2}\partial_{xx}\phi\rangle\dt+\langle\nu_{t}, (b_{A}(t,\cdot)-A^\prime_{t})\partial_{x}\phi\rangle\dt + \di Z^{\nu, \phi}_t,\label{eq:weak_form_nu} \\
 \di Z_t^{\nu, \phi}
&=  \bigl( \tfrac{1}{2}\partial_x \phi(0) - \phi(0) \gamma(t, C_t) \bigr) \sigma(t, A_t)^2 v_t(0)  \dt,\label{eq:bc_nu_classical}
\end{align}
with $Z^\nu_0=0$, for all $t\in[0,T]$ and $\phi \in C^{\infty}_b(\R)$, where $v_t$ denotes the (random) density of $\nu_t$. Furthermore, $v_t$ is adapted to the filtration $\mathcal{F}^\mu_t=\sigma(I^\mu_s:s\leq t)$, jointly continuous on $(0,T]\times \R^+$, and satisfies the same Aronson estimate and $L^2$ control as in \eqref{eq:lamperti_aronson} and \eqref{eq:L2-control}, respectively. Note also that $I^\nu_t:=1-\nu_t(\R^+)$ satisfies $I^\nu=Z^{\nu,-\mathbf{1}}$.
\end{prop}
\begin{proof} By the definition of $\nu$ in \eqref{eq:defn_nu}, we can take $\phi(t,x)=\phi(x-A_t)$, for $\phi \in C_b^\infty(\R)$, as the $\mathcal{F}^\mu$-adapted $C^{1,2}$ test function in \eqref{eq:lin_weak_formulation_P}--\eqref{eq:BC_P} and change variables to obtain \eqref{eq:weak_form_nu} with
\begin{equation}\label{eq:BC_nu}
Z^{\nu, \phi}_t :=\E\Bigl[ \int_0^t \bigl(\tfrac{1}{2}\partial_x\phi(0)- \gamma(s,C_s)\phi(0)\bigr)\mathbf{1}_{s<\tau}\di \ell_s  \mid \mathcal{F}^\mu_T\Bigr].
\end{equation}
It remains to argue that \eqref{eq:BC_nu} in fact takes the form \eqref{eq:bc_nu_classical} with (each realisation of) $\nu$ having a continuous density function $v$. To this end, the key observation is the following: by the construction of $X$, we can argue in complete analogy with Theorem \ref{thm:perturb_density} and Corollary \ref{cor:L2-sqrt} to conclude that, conditional on $\mathcal{F}^\mu_T$, $\nu_t$ has a density $v_t$ such that $(t,x)\mapsto v(t,x)$ is continuous up to the boundary and satisfies \eqref{eq:lamperti_aronson} and \eqref{eq:L2-control}. As in \eqref{eq:limit_local_time} and the proof of Lemma \ref{lemma:DCT_epsilon}, we in particular get
\begin{align}
&\E\Bigl[ \int_0^t \bigl( \tfrac{1}{2} \partial_x\phi(0)- \gamma(s,C_s)\phi(0)\bigr)\mathbf{1}_{s<\tau}\di \ell_s  \mid \mathcal{F}^\mu_T\Bigr] \nonumber \\
& \;\; = \lim_{\delta \downarrow 0} \E\Bigl[ \int_0^t \bigl( \tfrac{1}{2} \partial_x\phi(0)- \gamma(s,C_s)\phi(0)\bigr)\mathbf{1}_{s<\tau}\psi_\delta(X_s - A_s)\sigma_A(s,X_s)^2 \ds  \mid \mathcal{F}^\mu_T\Bigr] \nonumber \\
&\;\; = \int_0^t \bigl(\tfrac{1}{2} \partial_x\phi(0)- \gamma(s,C_s)\phi(0)\bigr)\lim_{\delta \downarrow 0} \langle \nu_s , \sigma_A(s,\cdot)^2\psi_\delta \rangle \ds \nonumber
\end{align}
which is equal to \eqref{eq:bc_nu_classical} by the continuity of $v$. The claim $I=Z^{\nu,-\mathbf{1}}$ follows by taking $\phi \equiv -\mathbf{1}$ as the test function in \eqref{eq:weak_form_nu}, since $\langle \nu_0 , \mathbf{1} \rangle=1$.
\end{proof}

 For a given limit point $\mu$, which satisfies Proposition \ref{prop:rand_weak_solution}, we now have that the associated $\nu$ in Proposition \ref{prop:linearised_problem_nu} satisfies that same weak formulation but with $A$ and $C$ frozen in terms of $\mu$. Thus, we refer to it as a decoupled, or linearised, weak solution. Crucially, $\nu$ has a well-behaved density and, in turn, enjoys the stronger characterisation \eqref{eq:bc_nu_classical} of its boundary condition.

\subsection{Uniqueness of the linearised problem}

 In this section, we show that any limit point $\mu$ from Proposition \ref{prop:rand_weak_solution} must agree with the corresponding $\nu$ from Proposition \ref{prop:linearised_problem_nu}. In particular, $\mu$ inherits the good properties of $\nu$. For clarity of notation, we shall write $I^\mu$ and $Z^{\mu,\phi}$ for the quantities corresponding to the given $\mu$ (from Proposition \ref{prop:rand_weak_solution}). We will carry out energy estimates for mollifications of the two measure flows $\mu$ and $\nu$. It will be convenient to use the following Dirichlet heat kernel as the mollifier.

\begin{defn}[Dirichlet heat kernel]\label{defn:heat_kernel}
    We define the Dirichlet heat kernel $G_\varepsilon$ by
    \begin{equation*}
G_{\varepsilon}(x,y) : = p_{\varepsilon}(x-y) - p_{\varepsilon}(x+y),
\quad \text{where} \quad p_{\varepsilon}(x) : = \frac{1}{\sqrt{2 \pi \varepsilon}} e^{-\frac{x^2}{2 \varepsilon}}.
\end{equation*}
This is the Green's function for $\partial_\varepsilon u(\varepsilon,x) = \frac{1}{2}\partial_{xx}u(\varepsilon,x)$ on $\R^+$ with $u(\varepsilon,0)=0$.
\end{defn}

\begin{notation}
     For any given signed measure-valued stochastic process $(\pi_t)_t$, we denote
    \begin{equation}\label{eq:sub_meas_mu_eps}
        f^{\pi,\varepsilon}_t (x) : = \langle \pi_t , G_{\varepsilon}(x,\cdot) \rangle = \int_0^{\infty} G_{\varepsilon}(x,y) \di \pi_t(y), \qquad
        F^{\pi, \varepsilon}_t(x) : = \int_x^{\infty} f^{\pi, \varepsilon}_{t}(z) \dz,
    \end{equation}
    and $Z_t^{\pi, G_{\varepsilon}}(x) := Z_t^{\pi, G_{\varepsilon}(x, \cdot)}$.
\end{notation}

\begin{thm}[Uniqueness of the linearised problem] \label{thm:uniqueness_linear} Let $\mu$ be given by Proposition \ref{prop:rand_weak_solution}, and let $\nu$ be the corresponding decoupled weak solution provided by Proposition \ref{prop:linearised_problem_nu} with the same initial conditions $\nu_0 = \mu_0$. Then $(\mu_t)_{t\in[0,T]}$ and $(\nu_t)_{t\in[0,T]}$ coincide with probability 1.
\end{thm}
\begin{proof}
Fix $\mu $ from Proposition~\ref{prop:rand_weak_solution} and take $y \mapsto G_{\varepsilon}(x,y)$ as a test function in the weak formulation \eqref{eq:weak_form_1}. With the notation introduced above, $\{f^{\mu,\varepsilon}_t \}_{\varepsilon}$ is a family of smooth sub-probability densities on $[0, \infty)$ obtained by a particular mollification of $\mu$. The family $\{F^{\mu,\varepsilon}_t (x) \}_{\varepsilon}$ gives the corresponding probability mass on $(x,\infty)$ and satisfies
\begin{equation*}
\partial_x F^{\mu, \varepsilon}_t(x) = - f^{\mu, \varepsilon}_t(x).
\end{equation*}
For all bounded continuous processes $H$ adapted to the filtration \eqref{eq:limit_filtration} generated by $(\mu, I^{\mu})$, the boundary condition \eqref{eq:bc_as_limit} reads
\begin{equation}\label{eq:Z_with_heat_kernel_linear}
\expect{\int_0^t H_s \di Z_t^{\mu, G_{\varepsilon}} (x)}
= - \lim_{\delta \downarrow 0}  \expect{ \int_0^t H_s \partial_x p_{\varepsilon}(x) \langle \mu_s, \sigma_A(s, \cdot)^2 \psi_\delta  \rangle \ds},
\end{equation}
where we have used that $\partial_y G_{\varepsilon}(x,0) = \frac{2x}{\varepsilon} p_{\varepsilon}(x) = - 2 \partial_x p_{\varepsilon}(x)$ and $G_{\varepsilon}(x,0) = 0$.
The weak formulation \eqref{eq:weak_form_1} gives
\begin{align*}
f^{\mu,\varepsilon}_t (x) &- f^{\mu,\varepsilon}_0 (x)
= \int_0^t \big\langle \mu_s , \tfrac{1}{2} \sigma_{A}(s, \cdot)^2 \partial_{yy} G_{\varepsilon}(x,\cdot)
+ [b_A(s, \cdot) - A'_s] \partial_y G_{\varepsilon}(x,\cdot) \big\rangle \ds
+ Z_t^{\mu, G_{\varepsilon}}(x) \\
&= \int_0^t \big\langle \mu_s , \tfrac{1}{2} \sigma_{A}(s, \cdot)^2 \partial_{xx} G_{\varepsilon}(x,\cdot)
- [b_A(s, \cdot) - A'_s] \big( \partial_x G_{\varepsilon}(x,\cdot) + 2 \partial_x p_{\varepsilon}(x + \cdot)   \big) \big\rangle \ds  + Z_t^{\mu, G_{\varepsilon}}(x),
\end{align*}
where we have switched derivatives from $y$ to $x$ using relations $\partial_{yy} G_{\varepsilon} (x,y)=\partial_{xx} G_{\varepsilon} (x,y)$ and $\partial_{y} G_{\varepsilon}  (x,y) = - \partial_{x} G_{\varepsilon}  (x,y) - 2 \partial_x p_{\varepsilon}(x + y) $. Integrating over the interval $[x, \infty) \subset [0, \infty)$, we get
\begin{align*}
F^{\mu, \varepsilon}_t(x) - F^{\mu, \varepsilon}_0(x)
&= - \tfrac{1}{2} \int_0^t \partial_{x} \langle \mu_s , \sigma_{A}(s, \cdot)^2 G_{\varepsilon}(x,\cdot) \rangle \ds
+ \int_0^t  \langle \mu_s , [b_A(s, \cdot) - A'_s] G_{\varepsilon}(x,\cdot) \rangle \ds  \\
&\quad + 2 \int_0^t \langle \mu_s , [b_A(s, \cdot) - A'_s] p_{\varepsilon}(x + \cdot) \rangle \ds
+ \int_x^{\infty} Z_t^{\mu, G_{\varepsilon}}(z) \dz,
\end{align*}
where we have used that $\partial_x G_\varepsilon(x, y), G_\varepsilon(x, y)$ and $p_\varepsilon(x+y)$ tend to $0$ as $x \to \infty$.
Now let $\nu$ be the corresponding linearised weak solution provided by Proposition \ref{prop:linearised_problem_nu} with $\nu_0 = \mu_0$. Taking $y \mapsto G_{\varepsilon}(x,y)$ as a test function in \eqref{eq:weak_form_nu} and combining with the equation for $F^{\mu, \varepsilon}_t$, we get
\begin{align*}
\di \big(F^{\mu, \varepsilon}_t(x) - F^{\nu, \varepsilon}_t(x) \big)
&= - \tfrac{1}{2} \partial_{x} \langle \mu_t  - \nu_t , \sigma_{A}(t, \cdot)^2 G_{\varepsilon}(x,\cdot) \rangle \dt  \\
&\quad + \big\langle \mu_t - \nu_t , [b_A(t, \cdot) - A'_t] \big( G_{\varepsilon}(x,\cdot) + 2 p_{\varepsilon}(x + \cdot) \big) \big\rangle \dt \\
&\quad+  \int_x^{\infty} \di \big( Z_t^{\mu, G_{\varepsilon}}(z) - Z_t^{\nu, G_{\varepsilon}}(z) \big) \dz,
\end{align*}
where we have applied Fubini's on the last term on the right-hand side.

Denoting $\Delta_t:=\mu_t - \nu_t$, we can then derive an evolution equation for $ \gamma(t, C_t)F^{\Delta, \varepsilon}_t(x)^2$ which we integrate in time and space to obtain
\begin{align}
\gamma(t, C_t) \| F^{\Delta, \varepsilon}_t \|_2^2
&= - \int_0^t \gamma(s, C_s) \int_0^{\infty} F^{\Delta, \varepsilon}_s(x) \, \partial_{x} \langle \Delta_s, \sigma_{A}(s, \cdot)^2 G_{\varepsilon}(x,\cdot) \rangle \dx \ds \nonumber \\
&\quad + 2 \int_0^t \gamma(s, C_s) \int_0^{\infty} F^{\Delta, \varepsilon}_s(x) \, \big\langle \Delta_s, [b_A(s, \cdot) - A'_s] \big( G_{\varepsilon}(x,\cdot) + 2 p_{\varepsilon}(x + \cdot) \big) \big\rangle \dx \ds \nonumber \\
&\quad + 2  \int_0^{\infty} \int_x^{\infty} \int_0^t \gamma(s, C_s) F^{\Delta, \varepsilon}_s(x) \di \big( Z_s^{\mu, G_{\varepsilon}}(z) - Z_s^{\nu, G_{\varepsilon}}(z) \big) \dz \dx \nonumber \\
&\quad + \int_0^t \| F^{\Delta, \varepsilon}_s \|_2^2  \big( \partial_s \gamma(s, C_s)\ds + \partial_{c} \gamma(s,C_s) \di C_s \big)
=:- \mathcal{I}^{\sigma}_{\varepsilon} + 2 \mathcal{I}^{b}_{\varepsilon} + 2 \mathcal{I}^{Z}_{\varepsilon} + \mathcal{I}^{\di \gamma}_{\varepsilon}. \label{eq:l2_delta}
\end{align}
We deal with the terms one by one. First of all, integrating by parts we have
\begin{equation*}
\mathcal{I}^{\sigma}_{\varepsilon}
= - \int_0^t \gamma(s, C_s)  \int_0^{\infty} \partial_x F^{\Delta, \varepsilon}_s(x)
 \langle \Delta_s, \sigma_{A}(s, \cdot)^2 G_{\varepsilon}(x,\cdot) \rangle \dx \ds,
\end{equation*}
where the boundary term vanished since $F^{\Delta, \varepsilon}_s(x) \rightarrow 0$ as $x$ approaches $+\infty$, and $G_{\varepsilon}(0,y) = 0$. Then, adding and subtracting terms,
\begin{equation}\label{eq:integral_sigma_lin}
\mathcal{I}^{\sigma}_{\varepsilon}
 = \int_0^t \gamma(s, C_s) \int_0^{\infty} \Big( \sigma_{A}(s, x)^2 f^{\Delta, \varepsilon}_s(x)^2 + f^{\Delta, \varepsilon}_s(x) c^{\sigma, \varepsilon}_s(x) \Big) \dx \ds,
\end{equation}
where we have introduced the commutator
\begin{equation}\label{eq:commutator}
c^{\phi, \varepsilon}_t(x)
:= f_t^{\phi\cdot \Delta,\varepsilon}(x) - \phi(t,x)f^{ \Delta,\varepsilon}_t(x)=
\langle \Delta_t, \phi(t, \cdot) G_{\varepsilon}(x,\cdot) \rangle  - \phi(t, x) f^{\Delta, \varepsilon}_t(x)
\end{equation}
and used the shorthand $c^{\sigma,\varepsilon}$ for $c^{\sigma_A^2,\varepsilon}$. Similarly, we use the shorthand $c^{b,\varepsilon}$ for $c^{b_A,\varepsilon}$ below. Moreover, we shall need the following remainder
\begin{equation}\label{eq:remainder}
r^{\varepsilon}_t(x)
:=  \langle \Delta_t, [b_A(t, \cdot) - A'_t] \,  p_{\varepsilon}(x + \cdot) \rangle.
\end{equation}
Now consider $\mathcal{I}^{b}_{\varepsilon}$. We split it into a $G_{\varepsilon}$ term and a $p_{\varepsilon}$ term, integrate by parts and add and subtract terms as above, to yield
\begin{align*}
\mathcal{I}^{b}_{\varepsilon} 
&= \int_0^t \gamma(s, C_s) \int_0^{\infty} - \tfrac{1}{2} [b_A(s, x) - A'_s] \, \partial_x \big( F^{\Delta, \varepsilon}_s(x)^2 \big)
+ F^{\Delta, \varepsilon}_s(x) \big( c^{b, \varepsilon}_s(x)  + 2 r^{\varepsilon}_s(x) \big) \dx \ds
 \\
 &\le  \int_{0}^t \gamma(s, C_s) \Big( \tfrac{1}{2}  K_{b} \| F^{\Delta, \varepsilon}_s   \|_2^2 + \tfrac{1}{2} [b_A(s, 0) - A'_s] \, F^{\Delta,\varepsilon}_s(0)^2 + \hspace{-3pt} \int_0^{\infty}\hspace{-3pt} F^{\Delta, \varepsilon}_s(x) \big( c^{b, \varepsilon}_s(x)  + 2 r^{\varepsilon}_s(x) \big) \dx \Big) \ds
  \\
 &\le  \int_{0}^t \gamma(s, C_s) \Big( \tfrac{1}{2}  K_{b} \| F^{\Delta, \varepsilon}_s   \|_2^2 -
 \tfrac{1}{2} K_b \hspace{-3pt}\int_0^{\infty} \hspace{-3pt} \partial_x \big( F^{\Delta, \varepsilon}_s(x)^2   \big) \dx + \hspace{-3pt}\int_0^{\infty} \hspace{-3pt} F^{\Delta, \varepsilon}_s(x) \big( c^{b, \varepsilon}_s(x)  + 2 r^{\varepsilon}_s(x) \big) \dx \Big) \ds, 
\end{align*}
where we have used the Lipschitz and linear growth conditions on $b$ from Assumption~\ref{assump:coefficient_assumptions} to bound the (weak) derivative $|\partial_x b_A(s,x)|$ and $|b_A(s,0)|$ by a constant $K_b$; the $o(1/x)$ decay of $F_s^{\Delta, \varepsilon}(x)$ as $x \to \infty$ by Lemma~\ref{lemma:F_decay}; we have dropped the term $-A'_s \gamma(s,C_s)F_s^{\Delta,\varepsilon}(0)^2$ since it is negative, as $\gamma$ and $\varrho$ are non-negative by Assumption~\ref{assump:coefficient_assumptions}, and thus $A'_s$ is also non-negative by \eqref{eq:A_weak_derivative}.
Applying the generalized Young's inequality with free parameter $\theta > 0$ yields
\begin{align}
    2 \mathcal{I}^b_{\varepsilon}
     &\le   K_{b}  \int_{0}^t \gamma(s, C_s) \| F^{\Delta, \varepsilon}_s   \|_2^2 \ds
    + K_b \int_0^t \gamma(s, C_s) \bigg( \frac{1}{\theta}\,\|F^{\Delta, \varepsilon}_s \|_2^2
    +  \theta \, \| f^{\Delta, \varepsilon}_s    \|_2^2  \bigg) \ds \nonumber \\
    &\quad  + 2 \int_0^t \gamma(s, C_s) \int_0^{\infty} F^{\Delta, \varepsilon}_s(x)\Big( c^{b, \varepsilon}_s(x)  + 2 r^{\varepsilon}_s(x) \Big) \dx \ds. \label{eq:integral_b_lin}
\end{align}
We now deal with $\mathcal{I}^{\di \gamma}_{\varepsilon}$. By Assumption~\ref{assump:coefficient_assumptions}, $\gamma(t, x)$ is continuous and strictly positive, therefore bounded away from 0, and has continuous derivatives, so bounded on $[0,T]$. Moreover, by the continuous differentiability of $C_t$ and the assumptions on $\gamma$, we have
\begin{equation}\label{eq:integral_digamma}
\mathcal{I}^{\di \gamma}_{\varepsilon}
= \int_0^t \gamma(s, C_s) \| F^{\Delta, \varepsilon}_s \|_2^2 \frac{\big( \partial_s \gamma(s, C_s) + \partial_{c} \gamma(s,C_s) \tfrac{\di C_s}{\ds} \big)}{\gamma(s, C_s)} \ds 
\le \tfrac{K_{\gamma,C}}{\kappa_{\gamma}} \int_0^t \gamma(s, C_s) \| F^{\Delta, \varepsilon}_s \|_2^2\ds.
\end{equation}
Using \eqref{eq:integral_sigma_lin}, \eqref{eq:integral_b_lin} and \eqref{eq:integral_digamma} into \eqref{eq:l2_delta}, and grouping similar terms, we get
\begin{align}
\gamma(t, C_t)\| F^{\Delta, \varepsilon}_t \|_2^2
& \le  - \int_0^t \hspace{-3pt} \gamma(s, C_s) \int_0^{\infty} \hspace{-5pt} \sigma_{A}(s, x)^2 \, f^{\Delta, \varepsilon}_s(x)^2 \dx \ds  - \int_0^t \hspace{-3pt} \gamma(s, C_s) \int_0^{\infty} \hspace{-3pt} f^{\Delta, \varepsilon}_s(x)  c^{\sigma, \varepsilon}_s(x) \dx \ds \nonumber \\
&\quad + \Big(K_{b}\big(1 + \tfrac{1}{\theta}\big) + \tfrac{K_{\gamma,C}}{\kappa_{\gamma}} \Big)  \int_{0}^t \gamma(s, C_s) \| F^{\Delta, \varepsilon}_s   \|_2^2 \ds
+  K_b \theta \int_0^t \gamma(s, C_s) \| f^{\Delta, \varepsilon}_s    \|_2^2  \ds \nonumber \\
&\quad + 2 \int_0^t \gamma(s, C_s) \int_0^{\infty} F^{\Delta, \varepsilon}_s(x) \Big( c^{b, \varepsilon}_s(x)  + 2 r^{\varepsilon}_s(x) \Big) \dx \ds + 2 \mathcal{I}^{Z}_{\varepsilon}. \label{eq:unique_bound_L1-remainder}
\end{align}
Applying Young's inequality to the second and fifth terms on the right-hand side, using that $\sigma^2 \ge {\kappa}_{\sigma}$ by Assumption~\ref{assump:coefficient_assumptions}, and taking expectation on both sides of the inequality, we obtain
\begin{align*}
\expect{\gamma(t, C_t) \| F^{\Delta, \varepsilon}_t \|_2^2}
&\le \big( 3 + K_{b}\big(1 + \tfrac{1}{\theta} \big) + \tfrac{K_{\gamma,C}}{\kappa_{\gamma}} \big)  \int_{0}^t \expect{\gamma(s, C_s) \| F^{\Delta, \varepsilon}_s   \|_2^2} \ds  \\
&\quad + \big( ( K_b +\tfrac{1}{2} ) \theta - \kappa_{\sigma} \big) \expect{ \int_0^t \gamma(s, C_s) \| f^{\Delta, \varepsilon}_s    \|_2^2  \ds }  + \tfrac{1}{2\theta} \expect{\int_0^t \gamma(s, C_s) \| c^{\sigma, \varepsilon}_s \|_2^2  \ds} \\ 
&\quad+ \expect{ \int_0^t \gamma(s, C_s)  \| c^{b, \varepsilon}_s \|_2^2 \ds} 
+ 2 \expect{\int_0^t \gamma(s, C_s)  \| r^{\varepsilon}_s \|_2^2  \ds} + 2 \expect{\mathcal{I}^{Z}_{\varepsilon}}.
\end{align*}
Note that as long as we choose $\theta \le \frac{\kappa_{\sigma}}{K_b + 1/2}$, the second term on the right hand side is negative and we can drop it. Since $\gamma(t, C_t)$ is bounded on $[0,T]$, Lemma~\ref{lemma:error_terms_decay_linear} ensures that $\expect{ \int_0^t \gamma(s, C_s) \| c^{\sigma, \varepsilon}_s \|_2^2 \ds}$ and $\expect{\int_0^t \gamma(s, C_s) \| c^{b, \varepsilon}_s \|_2^2 \ds}$ go to 0 as $\varepsilon \downarrow 0$. Moreover, $\lim_{\varepsilon \downarrow 0} \, \expect{\int_0^t \| r^{\varepsilon}_s \|_2^2  \ds} = 0$
by Lemma~\ref{lemma:error_terms_decay_linear_2}. Thus the above simplifies further to 
\begin{equation}\label{eq:expectation_ineq_lin}
\expect{\gamma(t, C_t) \| F^{\Delta, \varepsilon}_t \|_2^2}
\le \Big( 3 + K_{b}\big(1 + \tfrac{1}{\theta} \big) + \tfrac{K_{\gamma,C}}{\kappa_{\gamma}} \Big)  \int_{0}^t \expect{\gamma(s, C_s) \| F^{\Delta, \varepsilon}_s   \|_2^2} \ds
+ o(1) + 2 \expect{\mathcal{I}^{Z}_{\varepsilon}}.
\end{equation}
Consider the last term on the right-hand side of the above inequality. By Fubini we get
\begin{align*}
\expect{\mathcal I_{\varepsilon}^Z}
&= - \int_0^{\infty} \int_x^{\infty} \lim_{\delta \downarrow 0 }
\expect{\int_0^t \gamma(s, C_s) F^{\Delta, \varepsilon}_s(x) \partial_z p_{\varepsilon}(z) \langle \Delta^{}_s, \sigma_A(s, \cdot)^2 \psi_\delta  \rangle \ds} \dz \dx \\
&=  \int_0^{\infty} p_{\varepsilon}(x) \lim_{\delta \downarrow 0 }
\expect{\int_0^t \gamma(s, C_s) F^{\Delta, \varepsilon}_s(x) \langle \Delta^{}_s, \sigma_A(s, \cdot)^2 \psi_\delta  \rangle \ds} \dx \\
&=  \int_0^{\infty} p_{\varepsilon}(x)
\expect{\int_0^t F^{\Delta, \varepsilon}_s(x) \di \big( I^{\mu}_s - I_s^{\nu} \big)} \dx,
\end{align*}
where the first equality above follows from equation \eqref{eq:Z_with_heat_kernel_linear} (and its equivalent for the linearised solution $\nu$), with $H_S = \gamma(s, C_s) F^{\Delta, \varepsilon}_s(x)$, which is adapted to the filtration generated by $(\mu, I^{\mu})$. Similarly the third equality follows from taking $H_s = F^{\Delta, \varepsilon}_s(x)$ and $\phi = - \mathbf{1}$ in \eqref{eq:bc_as_limit}, recalling that $ Z^{\mu, -\mathbf{1}}_s =  I^{\mu}_s$ and $ Z^{\nu, -\mathbf{1}}_s = I^{\nu}_s$ as per Proposition~\ref{prop:rand_weak_solution} (noting that this also holds for $\nu$).

Applying Fubini one more time, Lemma \ref{lemma:square_of_I_diff} gives us that $\expect{\mathcal I_{\varepsilon}^Z}$ is negative, so we can discard it from the right-hand side of \eqref{eq:expectation_ineq_lin}. Using the integrating factor $\exp \big\{ -(3 + K_b(1 + \tfrac{1}{\theta}) + \tfrac{K_{\gamma,C}}{\kappa_{\gamma}}) t\big\}$, we arrive at the estimate
\begin{equation*}
    \expect{\gamma(t, C_t) \| F^{\Delta, \varepsilon}_t \|_2^2} \le o(1)\exp \big\{ \big( 3 + K_{b}\big(1 + \tfrac{1}{\theta} \big)  + \tfrac{K_{\gamma,C}}{\kappa_{\gamma}} \big)t \big\},
\end{equation*}
for all $t \in [0,T]$. We can observe that $\| F^{\Delta, \varepsilon}_t \|_2 \leq 2\int_0^\infty y\hspace{1pt} \di \hspace{1pt}|\Delta_t|(y)$, so sending $\varepsilon \downarrow 0$ and applying dominated convergence gives
$
 \expect{\gamma(t, C_t) \| F^{\Delta}_t \|_2^2} =0,
$
for all $t\in[0,T]$, where $F^{\Delta}_t(x) := \Delta_t(x,\infty)$, and hence the uniqueness claim follows.
\end{proof}

\begin{lemma}\label{lemma:F_decay}  With $\pi=\mu,\nu$, we have $\int_0^{\infty} x \hspace{1pt}\di  \pi_t(x) < \infty$, for all $t\in[0,T]$, and hence $f^{\pi, \varepsilon}_t(x)$ and $F^{\pi, \varepsilon}_t(x)$ defined in \eqref{eq:sub_meas_mu_eps} are of order $o(1/x)$ as $x\rightarrow \infty$.
\end{lemma}
\begin{proof}
The first claim is contained in Proposition \ref{prop:linearised_problem_nu} for $\nu$, while it follows by the weak convergence and Theorem \ref{thm:perturb_density} for $\mu$. The remaining claims follow readily from this, by splitting the corresponding integrals, using the decay of the exponential kernels, and applying DCT and Fubini.
\end{proof}

\begin{lemma}[Controlling the commutators]\label{lemma:error_terms_decay_linear}
Let $\phi(t,x) \in \{ b_A(t,x), \sigma_{A}(t, x)^2 \}$. Then, the commutator $c^{\phi,\varepsilon}_t(x)$ defined in \eqref{eq:commutator} satisfies
	\begin{equation*}
	  \| c^{\phi,\varepsilon}_t \|^2_2  \leq C\sqrt{\varepsilon},
	\end{equation*}
    for all $t\in[0,T]$, where $C$ only depends on $\phi$ and $T>0$.
\end{lemma}
\begin{proof} By Assumption \ref{assump:coefficient_assumptions}, $x\mapsto\phi(t,x)$ is Lipschitz uniformly in $t\in[0,T]$. Hence, we get
\[
	|c^{\phi,\varepsilon}_t(x)|
	\leq C_\phi \int_0^\infty |x - y| \, p_\varepsilon(x-y) \, \di |\Delta_t|(y).
\]
We can bound $r p_\varepsilon(r) \leq \kappa \sqrt{\varepsilon} p_{2\varepsilon}(r)$ since $r \, e^{-r^2/4\varepsilon}$ is maximised at $r=\sqrt{2\varepsilon}$, so
\[
	|c^{\phi,\varepsilon}_t(x)|
	\leq
	\kappa C_\phi \sqrt{\varepsilon} \int_0^\infty p_{2\varepsilon}(x-y) \, \di |\Delta_t|(y).
\]
Squaring this, applying Jensen's inequality, and 
integrating in $x$, we get
\[
	\Vert c^{\phi,\varepsilon}_t\Vert_2^2
	\leq 2\kappa^2 C_\phi^2 \varepsilon \int_0^\infty \Vert p_{2\varepsilon} \Vert_2^2 \,\di |\Delta_t|(y) \leq C\sqrt{\varepsilon}
\]
with $C=4\kappa^2 C_\phi^2/\sqrt{8\pi}$, since $\Vert p_{2\varepsilon} \Vert_2^2= 1/\sqrt{8\pi \varepsilon}$ and $|\Delta_t|(0,\infty)\leq 2$.
\end{proof}

\begin{lemma}[Controlling the remainder]\label{lemma:error_terms_decay_linear_2}
    The remainder $r_t^\varepsilon(x)$ defined in \eqref{eq:remainder} satisfies that, for every $\theta \in (0, \tfrac{1}{2})$, there exists $\delta \in (1-\theta,1)$ and $C=C(\theta) > 0$ such that
\[
\E\big[\|r^\varepsilon_t\|_{2}^2\big]
	\leq C \varepsilon^{\frac{1}{2} - \theta} t^{-\delta},
\]
for all $t \in (0,T]$, for $\varepsilon > 0$ sufficiently small.
\end{lemma}
\begin{proof} Recall $r^{\varepsilon}_t(x)
=  \langle \Delta_t, [b_A(t,\cdot) - A_t^\prime] \,  p_{\varepsilon}(x + \cdot) \rangle$, where $\Delta_t =\mu_t - \nu_t$. For $\mu$, we have
\[
\E\bigl[ (\mu^n(a,b))^2\bigr]\leq \frac{1}{n^2}\sum_{i\neq j} \Prob\bigl(X^{i,n}_t\in (a,b) , \,X^{j,n}_t\in(a,b) \bigr) + \frac{1}{n},
\]
so we can apply the uniform bound in Proposition \ref{prop:pairwise_bound} and deduce from the weak convergence (to the given limit point $\mu$) that $\E[(\mu(a,b))^2]\leq C|b-a|^2t^{-1}$, for all $t\in(0,T]$, for some $C>0$. For $\nu$, the density estimate in Proposition \ref{prop:linearised_problem_nu} directly gives $\nu_t(a,b)\leq c |b-a|t^{-1/2}$ and hence $\nu_t(a,b)^2\leq c^2 |b-a|^2t^{-1}$, for all $t\in(0,T]$, for some $c>0$. As both values are in $[0,1]$, we get
\begin{equation}\label{eq:square_boundary_behavious}
\E\bigl[(\mu_t(a,b))^2\bigr]\leq C_\delta |b-a|^{2\delta}t^{-\delta} \quad\text{and}\quad  (\nu_t(a,b))^2\leq C_\delta |b-a|^{2\delta}t^{-\delta}, 
\end{equation}
for all $t\in(0,T]$, for some $C_\delta>0$, for each $\delta\in[0,1]$, which we exploit later in the proof.

By Assumption \ref{assump:coefficient_assumptions}, we have $|b_A(t,x)-A^{\prime}_t|\leq c(1+|x|)$ for $t\in[0,T]$. Take $\varepsilon \in(0,1)$. Using that $y \, p_\varepsilon(y) \leq \kappa \sqrt{\varepsilon} \, p_{2\varepsilon}(y)$ as in the proof of Lemma \ref{lemma:error_terms_decay_linear}, we get
\begin{equation}\label{eq:remainder_pointwise_bound}
	|r^\varepsilon_t(x)|
	\leq c e^{-x^2/2\varepsilon}
	\int_0^\infty (1+y) p_\varepsilon(y) \, \di |\Delta_t|(y) \leq c\kappa  e^{-x^2/2\varepsilon}\int_0^\infty p_{2\varepsilon}(y) \, \di |\Delta_t|(y).
\end{equation}
Squaring and integrating in $x$, we get $\|r^\varepsilon_t\|_{2}^2
	\leq C \sqrt{\varepsilon} \langle|\Delta_t| ,p_{2\varepsilon} \rangle^2$ with $C=c^2\kappa^2 \sqrt{\pi}/2$. Writing
\[
\langle|\Delta_t| ,p_{2\varepsilon} \rangle \leq \frac{1}{\sqrt{4\pi \varepsilon}} |\Delta_t|(0,z) + 2p_{2\varepsilon}(z)
\]
and taking $z=\varepsilon^\lambda$ for $\lambda \in(0,1/2)$, it follows that
\[
	\|r^\varepsilon_t\|_{2}^2 \leq \frac{C}{\sqrt{\varepsilon}} \bigl(\mu_t(0,\varepsilon^\lambda)^2 + \nu_t(0,\varepsilon^\lambda)^2 + e^{-\varepsilon^{2\lambda-1}/2}\bigr),
\]
for a fixed $C>0$. By \eqref{eq:square_boundary_behavious}, we then get
\[
	\E[\|r^\varepsilon_t\|_{2}^2]
	\leq C \varepsilon^{-1/2}\varepsilon^{2\lambda\delta} t^{-\delta}
		+ C \varepsilon^{-1/2} e^{-\varepsilon^{2\lambda - 1}/2}.
\]
for each $\delta\in[0,1]$, for some $C=C(\delta)>0$, so the claim follows by choosing $\delta \in (1-\theta, 1)$
and setting $\lambda = (1-\theta)/2\delta < 1/2$, for any given $\theta\in(0,1/2)$.
\end{proof}

\begin{lemma}[Boundary term]\label{lemma:square_of_I_diff} We have
	\begin{equation*}
	\lim_{\varepsilon \downarrow 0} \expect{ \int_0^t \int_0^{\infty} p_{\varepsilon}(x) F^{\Delta, \varepsilon}_s(x)  \dx \di  \big( I^{\mu}_s - I_s^{\nu} \big)}
	= - \frac{1}{4} \expect{ \big( I^{\mu}_t - I_t^{\nu}  \big)^2   }.
	\end{equation*}
\end{lemma}
\begin{proof}
Consider
	\begin{equation*}
	\int_0^{\infty} p_{\varepsilon} (x) F^{\Delta, \varepsilon}_s(x) \dx
	=\int_0^{\infty} \int_0^{\infty} p_{\varepsilon} (x) \int_x^{\infty}  G_{\varepsilon}(z,y) \dz \dx \di (\mu_s - \nu_s )(y),
	\end{equation*}
    where we have applied Fubini's theorem to exchange the order of integration. For any $y > 0$, by definition of $p_\varepsilon$ and $G_\varepsilon$, we get
    \begin{align}
    &\Bigl| \int_0^{\infty} p_{\varepsilon} (x) \int_x^{\infty}  G_{\varepsilon}(z,y) \dz \dx - \tfrac{1}{2} \Bigr| = \Bigl| \int_0^{\infty} p_{\varepsilon} (x) \Bigl( \int_x^{\infty}  G_{\varepsilon}(z,y) \dz - 1 \Bigr) \dx \Bigr| \nonumber \\
    & \quad \leq  \int_{y/2}^{\infty} p_{\varepsilon} (x) \dx   +  \int_{0}^{y/2} p_{\varepsilon} (x) \int_0^{x}  G_{\varepsilon}(z,y) \dz \dx 
    \leq  \sqrt{2} e^{-y^2/16\varepsilon} +  \sqrt{2} e^{-y^2/16\varepsilon}  \int_{0}^{y/2} p_{\varepsilon} (x) \dx ,
    \label{eq:two_terms_pG}  
    \end{align}
where have exploited that $x\geq y/2$ to bound $p_\varepsilon(x)$ in the first term, and that $z\leq x \leq y/2 $ to bound $G_\varepsilon$ in the second term. Since \eqref{eq:two_terms_pG} tends to zero as $\varepsilon$ goes to zero, for every $y>0$, it follows from dominated convergence that
	\begin{align}
	\lim_{\varepsilon \downarrow 0} \int_0^{\infty} p_{\varepsilon} (x) F^{\Delta, \varepsilon}_s(x) \nonumber \dx
	&= \lim_{\varepsilon \downarrow 0} \int_0^{\infty} \int_0^{\infty} p_{\varepsilon} (x) \int_x^{\infty}  G_{\varepsilon}(z,y) \dz \dx \di (\mu_s - \nu_s )(y) \nonumber \\
	&= \int_0^{\infty} \Big( \lim_{\varepsilon \downarrow 0} \int_0^{\infty} p_{\varepsilon} (x) \int_x^{\infty}  G_{\varepsilon}(z,y) \dz \dx \Big) \di (\mu_s - \nu_s )(y) \nonumber \\
    &= \tfrac{1}{2}\bigl(\mu_s(\R^+) - \nu_s(\R^+)\bigr), \label{eq:limit_pe_FDelta}
	\end{align}
 a.s.~for a.e.~$s\in[0,T]$. That \eqref{eq:two_terms_pG} vanishes for $\varepsilon>0$ is sufficient since $\nu_s(\{0\})=0=\mu_s(\{0\})$ for a.e.~$s\in[0,T]$. For $\nu$ this holds by Proposition~\ref{prop:linearised_problem_nu} and for $\mu$ we have
\[
0\leq \E\biggl[ \int_0^T  \liminf_{\delta \downarrow 0} \langle\mu_s, \sigma_A(s, \cdot)^2 \psi_\delta  \rangle \ds \biggr] < \infty,
\]
by \eqref{eq:bc_as_limit} and Fatou's lemma. Finally, applying DCT again to pass the limit inside the expectation and the time-integral, and recalling that $I_t^{\mu} = 1 - \mu_t(\R^+)$ and likewise for $I_t^{\nu}$, we get
	\begin{equation*}
	\lim_{\varepsilon \downarrow 0} \expect{ \int_0^t \int_0^{\infty} p_{\varepsilon}(x) F^{\Delta, \varepsilon}_s(x) \dx \di \big( I^{\mu}_s - I_s^{\nu} \big)}
	= \tfrac{1}{2}\expect{ \int_0^t \big( I^{\nu}_s - I_s^{\mu} \big)  \di \big( I^{\mu}_s - I_s^{\nu} \big)}.
	\end{equation*}
	Since $I^\mu-I^\nu$ is continuous and of finite variation on $[0,T]$, being the difference of two  continuous increasing processes, the Stieltjes integral on the right-hand side of the above expression is equal to $-(I^\mu_t-I^\nu_t)^2/2\leq 0$. This completes the proof.
\end{proof}

\section{Uniqueness of the free boundary problem}\label{sect:free_boundary_uniqueness}

As in the previous sections, we continue to work in the frame of the moving boundary. The price we pay for having converted to a fixed domain is that we must deal with nonlinear (and nonlocal) quantities in the adjusted drift and diffusion coefficients. In the previous section, this did not affect the analysis, since we were only comparing with the corresponding decoupled weak solution. 

Following on from Propositions \ref{prop:rand_weak_solution} and \ref{prop:linearised_problem_nu}, we begin by stating a stronger and deterministic notion of weak solution. In Section \ref{sect:unique_weak_soln}, we show that there can be at most one such solution. Then, Section \ref{subsect:proof_thms2.2-2.3} invokes this to conclude the proof of Theorems \ref{thm:free_boundary_problem} and \ref{thm:summary_contributions}.

\begin{defn}[Weak solution in the frame of the free boundary]\label{def:clas_weak_formulation} Let $(\mu_t)_{t\in[0,T]}$ be a measure flow such that $\mu_t\in \mathcal{M}_{\leq 1}(\R^+)$ and, for $t>0$, $\di \mu_t(x)=u_t(x)\dx$ with $u_t \in C(\R^+)$, where $u$ satisfies \begin{equation}\label{eq:density_control_uniqueness}
		\bigl\Vert u_t \bigr\Vert_2^2 \leq \frac{c_0}{\sqrt{t}}\quad \text{and}\quad |u_t(x)|\leq \frac{c_1}{\sqrt{t}}e^{-c_2x^2}, \quad x\in\R^+,\; t\in(0,T],
	\end{equation}
for some constants $c_0,c_1>0$. We say $\mu$ is a weak solution in the frame of the free boundary if
	\begin{align}
		\di \langle\mu_t,\phi\rangle	&=\langle\mu_{t},\tfrac{{1}}{2}\sigma_{A}(t,\cdot)^{2}\partial_{xx}\phi\rangle\dt+\langle\mu_{t}, (b_{A}(t,\cdot)-A^\prime_{t})\partial_{x}\phi\rangle\dt + \di Z^{\mu,\phi}_t, \label{eq:weak_form} \\
		 \di Z_t^{\mu, \phi}
		&=    \bigl(\tfrac{1}{2} \partial_x \phi(0) -  \gamma(t, C_t)\phi(0) \bigr) \sigma_A(t, 0)^2 u_t(0)  \dt, \label{eq:bc_with_density}
	\end{align}
	with $\mu_0 = P_0 \circ \theta^{-1}_{a_0}$ and $Z_0^{\mu}=0$, for all $t\in[0,T]$ and $\phi \in C^{\infty}_b(\R)$, where
	\begin{equation}\label{eq:ACI_defn_weak_sol}
		A_t=a_0 + \alpha \! \displaystyle \int_{t-\bar{d}}^{t}\varrho(t-s)I_{s} \ds, \;\;
		C_t = \!\int_{t-\bar{d}}^{t}\varrho(t-s)(I_{s}-I_{s-\bar{d}}) \ds, \;\;  I_t= 1- \mu_t(\R^+).
	\end{equation}
\end{defn}

For our estimates, we shall exploit that  $A$ in \eqref{eq:ACI_defn_weak_sol} is absolutely continuous on $[0,T]$ with
\begin{equation}\label{eq:A_prime_for_uniqueness}
    A^\prime_t = - \alpha \varrho(\bar{d})I_{t-\bar{d}} + \alpha \int_{0}^{\bar{d}}I_{t-s} \di \varrho(s).
    \end{equation}
To see this, recall that by Assumption~\ref{assump:coefficient_assumptions} $\varrho$ is of bounded variation and right-continuous, and $\varrho(0) = 0$, so Fubini's theorem and an integration by parts give
\begin{align*}
\int_{t-\bar{d}}^t \varrho(t-s) I_{s} \ds
&= \int_0^{\bar{d}} \!\! \varrho(r) I_{t-r} \dr
= - \int_0^{\bar{d}} \!\! \varrho(r) \di \Bigl(\int_0^{t-r} \hspace{-7pt} I_s \di s \Bigr) \\
&= \int_0^{\bar{d}} \! \int_0^{t-r} \!\!\!\! I_{s}\di s  \di \varrho(r) -  \varrho(\bar{d}) \! \int_0^{t-\bar{d}} \!\! \!\!I_{s}
\di s
= \int_0^t \! \int_0^{\bar{d} \wedge (t-s)} \!\!\!\!\!\!\!\!\! I_{s}\di \varrho(r) \di s -  \varrho(\bar{d}) \! \int_0^{t} \! I_{s-\bar{d}}
\di s ,
\end{align*}
where we recall our convention $I_t=0$ for $t\leq 0$.

\subsection{Uniqueness of weak solutions}\label{sect:unique_weak_soln}

For the arguments that follow, we stress that the second part of \eqref{eq:density_control_uniqueness} in particular guarantees that Lemmas \ref{lemma:F_decay}, \ref{lemma:error_terms_decay_linear}, \ref{lemma:error_terms_decay_linear_2}, and \ref{lemma:square_of_I_diff} from the previous section all apply.

\begin{thm}[Uniqueness]\label{thm:gen_uniqueness}
    Fix $T>0$, and let $\mu=(\mu_t)_{t\in[0,T]}$ and $\nu=(\nu_t)_{t\in[0,T]}$ be any two weak solutions satisfying Definition \ref{def:clas_weak_formulation}. If $\mu_0=\nu_0$, then $\mu_t=\nu_t$ for $t\in[0,T]$.
\end{thm}
\begin{proof}
With $\phi(y) = G_{\varepsilon}(x,y)$, given in Definition~\ref{defn:heat_kernel}, note that the boundary term \eqref{eq:bc_with_density} reads as
\begin{align*}
	Z_t^{\mu, G_{\varepsilon}} (x)
	&= - \int_0^t \partial_x p_{\varepsilon}(x) \sigma_A(s, 0)^2 u_s(0) \ds.
\end{align*}
Regarding \eqref{eq:weak_form}, it follows as in Theorem~\ref{thm:uniqueness_linear} that
\begin{align*}
	F^{\mu, \varepsilon}_t(x) - F^{\mu, \varepsilon}_0(x)
	&= - \tfrac{1}{2} \int_0^t \partial_{x} \langle \mu_s , \sigma_{A}(s, \cdot)^2 G_{\varepsilon}(x,\cdot) \rangle \ds
	+ \int_0^t  \langle \mu_s , [b_A(s, \cdot) - A'_s] \, G_{\varepsilon}(x,\cdot) \rangle \ds  \\
	&\quad + 2 \int_0^t \langle \mu_s , [b_A(s, \cdot) - A'_s] \, p_{\varepsilon}(x + \cdot) \rangle \ds
	+ \int_0^t p_{\varepsilon}(x) \sigma_A(s, 0)^2 u_s(0) \ds.
\end{align*}
Now let $\nu$ be another weak solution with $\nu_0 = \mu_0$. We denote by $\tilde{A}$, $\tilde{C}$ and $\tilde{I}$ the quantities in \eqref{eq:ACI_defn_weak_sol} corresponding to the weak solution $\nu$, and by $v_t(x)$ the density of $\nu_t$. Then we get
\begin{align}\label{eq:weak_form_nu_1}
	f^{\nu,\varepsilon}_t (x) - f^{\nu,\varepsilon}_0 (x)
	&= \int_0^t \langle \nu_s , \tfrac{1}{2} \sigma_{\tilde{A}}(s, \cdot)^2 \partial_{xx} G_{\varepsilon}(x, \cdot)
	+ [b_{\tilde{A}}(s, \cdot) - \tilde{A}'_s] \, \partial_x G_{\varepsilon}(x, \cdot) \rangle \ds + Z_t^{\nu, G_{\varepsilon}}(x),
\end{align}
where we have again switched derivatives from $y$ to $x$, and we get
\begin{equation*}
	Z_t^{\nu, G_{\varepsilon}}(x)
	=  - \int_0^t \partial_x p_{\varepsilon}(x) \sigma_{\tilde A}(s, 0)^2 v_s(0) \ds.
\end{equation*}
 Integrating \eqref{eq:weak_form_nu_1} over $[x, \infty)$, we arrive at
\begin{align*}
	\di \big(F^{\mu, \varepsilon}_t(x) - F^{\nu, \varepsilon}_t(x) \big)
	&= - \tfrac{1}{2} \partial_{x} \langle \mu_t  - \nu_t , \sigma_{A}(t, \cdot)^2 G_{\varepsilon}(x,\cdot) \rangle \dt \\
	&\quad - \tfrac{1}{2} \partial_{x} \langle \nu_t , [\sigma_{A}(t, \cdot)^2 -\sigma_{\tilde A}(t, \cdot)^2] G_{\varepsilon}(x,\cdot) \rangle \dt
	\\
	&\quad + \langle \mu_t - \nu_t , [b_A(t, \cdot) - A'_t] \big( G_{\varepsilon}(x,\cdot) + 2 p_{\varepsilon}(x + \cdot) \big) \rangle \dt \\
	&\quad + \langle \nu_t , [b_A(t, \cdot) - b_{\tilde A}(t, \cdot) - A'_t + \tilde A'_t] \big( G_{\varepsilon}(x,\cdot) + 2 p_{\varepsilon}(x + \cdot) \big) \rangle \dt \\
	&\quad + p_{\varepsilon}(x) \big( \sigma_A(t, 0)^2 u_t(0)
	-  \sigma_{\tilde A}(t, 0)^2 v_t(0) \big) \dt,
\end{align*}
where we have added and subtracted appropriate terms to reach this formulation.

Denote $\Delta_t : = \mu_t - \nu_t$, $\delta_{\sigma}(t,x) := \sigma_A(t,x)^2 - \sigma_{\tilde A}(t,x)^2 $, $\delta_b(t,x) := b_A(t,x) - b_{\tilde A}(t,x)$ and $\delta_{A'}(t) := A'_t - \tilde{A}'_t$. We first derive an equation for $ \gamma(t, C_t) F^{\Delta, \varepsilon}_t(x)^2$. Integrating in $x$, we then get
\begin{align*}
	\gamma(t,C_t) \|F^{\Delta, \varepsilon}_t\|_2^2
	&= - \int_0^t \gamma(s, C_s) \int_0^{\infty} \hspace{-7pt} F^{\Delta, \varepsilon}_s(x) \, \partial_{x} \langle \Delta_s , \sigma_{A}(s, \cdot)^2 G_{\varepsilon}(x,\cdot) \rangle \dx \ds \\
	&\quad - \int_0^t \gamma(s, C_s) \int_0^{\infty} \hspace{-7pt} F^{\Delta, \varepsilon}_s(x) \, \partial_{x} \big\langle \nu_s , \delta_{\sigma}(s, \cdot) G_{\varepsilon}(x,\cdot) \big\rangle \dx \ds
	\\
	&\quad + 2 \int_0^t \gamma(s, C_s) \int_0^{\infty} \hspace{-7pt} F^{\Delta, \varepsilon}_s(x) \big\langle \Delta_s , [b_A(s, \cdot) - A'_s] \big( G_{\varepsilon}(x,\cdot) + 2 p_{\varepsilon}(x + \cdot) \big) \big\rangle \dx \ds \\
	&\quad + 2 \int_0^t \gamma(s, C_s) \int_0^{\infty} \hspace{-7pt} F^{\Delta, \varepsilon}_s(x) \big\langle \nu_s , (\delta_b(s, \cdot) - \delta_{A'}(s)) \big( G_{\varepsilon}(x,\cdot) + 2 p_{\varepsilon}(x + \cdot) \big) \big\rangle \dx \ds \\
	&\quad + 2 \int_0^t \hspace{-3pt} \int_0^{\infty} \hspace{-7pt} F^{\Delta, \varepsilon}_s(x) \, p_{\varepsilon}(x) \Big( \gamma(s, C_s)  \sigma_A(s, 0)^2 u_s(0) - \gamma(s, \tilde C_s)  \sigma_{\tilde A}(s, 0)^2 v_s(0)  \Big) \dx \ds \\
	&\quad + 2 \int_0^t \big( \gamma(s, \tilde C_s) - \gamma(s, C_s) \big)  \sigma_{\tilde A}(s, 0)^2 v_s(0) \int_0^{\infty} \hspace{-7pt} F^{\Delta, \varepsilon}_s(x) \, p_{\varepsilon}(x)   \dx \ds \\
	&\quad + \int_0^t \| F^{\Delta, \varepsilon}_s(x)\|_2^2 \big( \partial_s \gamma(s, C_s)\ds + \partial_{c} \gamma(s,C_s) \di C_s \big) \\
	&:= - \mathcal{I}^{\sigma}_{\varepsilon} - \tilde{\mathcal{I}}^{\sigma}_{\varepsilon} + 2 \mathcal{I}^{b}_{\varepsilon} + 2 \tilde{\mathcal{I}}^{b}_{\varepsilon} +2 \mathcal{I}^{p}_{\varepsilon} +2 \tilde{\mathcal{I}}^{p}_{\varepsilon} + \mathcal{I}^{\di \gamma}_{\varepsilon}.
\end{align*}
The terms $\mathcal{I}^{\sigma}_{\varepsilon}$, $\mathcal{I}^{b}_{\varepsilon}$ and $\mathcal{I}^{\di \gamma}_{\varepsilon}$ are identical to those in the proof of Theorem~\ref{thm:uniqueness_linear}, yielding
\begin{align}
-\mathcal{I}^{\sigma}_{\varepsilon}
	&\le
	- \kappa_{\sigma} \int_0^t \gamma(s, C_s)  \| f^{\Delta, \varepsilon}_s   \|_2^2 \ds
	+ \int_0^t \gamma(s, C_s) \Big( \tfrac{\theta}{2} \| f^{\Delta, \varepsilon}_s \|_2^2
	+ \tfrac{1}{2\theta} \| c^{\sigma, \varepsilon}_s \|_2^2 \Big) \ds, \nonumber \\
2\mathcal{I}^{b}_{\varepsilon}
    &\le \Big(3 + K_{b}\big(1 + \tfrac{1}{\theta}\big) \Big)  \int_{0}^t \gamma(s, C_s) \| F^{\Delta, \varepsilon}_s   \|_2^2 \ds
    +  K_b \theta \int_0^t \gamma(s, C_s) \| f^{\Delta, \varepsilon}_s    \|_2^2  \ds \nonumber \\
    &\quad + \int_0^t \gamma(s, C_s) \Big( \| c^{b, \varepsilon}_s \|_2^2 +  2\| r^{\varepsilon}_s \|_2^2 \Big) \ds, \nonumber \\
\mathcal{I}^{\di \gamma}_{\varepsilon}
	&\le \tfrac{K_{\gamma,C}}{\kappa_{\gamma}} \int_0^t \gamma(s, C_s) \| F^{\Delta, \varepsilon}_s \|_2^2\ds, \label{eq:bounds_from_linear_uniq}
\end{align}
where $\theta > 0$ is a free parameter and the terms $c^{\sigma, \varepsilon}_t(x)$, $c^{b, \varepsilon}_t(x)$ and $r^{\varepsilon}_t(x)$ are as in the proof of Theorem~\ref{thm:uniqueness_linear}. Let us now consider the remaining terms. First of all, integrating by parts gives
\begin{align}\label{eq:integral_sigma}	\tilde{\mathcal{I}}^{\sigma}_{\varepsilon}
	&= \int_0^t \gamma(s, C_s)  \int_0^{\infty} f^{\Delta, \varepsilon}_s (x)
	\big( f^{\nu, \varepsilon}_s (x) \delta_{\sigma}(s,x) +  \tilde{c}^{\,\sigma, \varepsilon}_s(x) \big) \dx \ds,
\end{align}
where
\begin{equation*}
\tilde{c}^{\,\sigma, \varepsilon}_t (x) = 
\big\langle \nu_t , \delta_{\sigma}(t,\cdot) G_{\varepsilon}(x,\cdot) \big\rangle - f^{\nu, \varepsilon}_t (x) \delta_{\sigma}(t,x).
\end{equation*}
Recalling Assumption~\ref{assump:coefficient_assumptions}, we use that $\sigma$ is bounded and Lipschitz to yield
\begin{align*}
    |\delta_\sigma(s,x)| \le |\sigma(s, x+ A_s)^2 - \sigma(s, x+\tilde{A}_s)^2|
    \le  K_{\sigma} |A_s - \tilde{A}_s|
    \le \alpha K_\sigma \Big| \int_{s-\bar{d}}^s \varrho(s-r)(I_r - \tilde I_r) \dr \Big|.
\end{align*}
By Assumption~\ref{assump:coefficient_assumptions} $\varrho(t)$ is supported $[0, \bar{d}]$ and bounded, while $\gamma(t,x)$ is bounded and bounded away from 0. Thus, the generalized Young's inequality and Jensen's inequality give
\begin{align*}
    \int_0^t \gamma(s, C_s)  &\int_0^{\infty} f^{\Delta, \varepsilon}_s (x)
	 f^{\nu, \varepsilon}_s (x) \delta_{\sigma}(s,x) \dx \ds \\
     &\le \tfrac{\theta}{2} \int_0^t \gamma(s, C_s) \| f^{\Delta,\varepsilon} \|_2^2 \ds + \tfrac{\alpha^2 K_\sigma^2}{2\theta} \int_0^t \gamma(s, C_s) \| f^{\nu,\varepsilon}_s \|_2^2 \Big| \int_{0}^s \varrho(s-r)(I_r - \tilde I_r) \dr \Big|^2 \ds \\
     &\le \tfrac{\theta}{2} \int_0^t \gamma(s, C_s) \| f^{\Delta,\varepsilon} \|_2^2 \ds + \tfrac{\alpha^2 K_\sigma^2 K_\gamma c_0  }{2\theta} \int_0^t \int_{0}^s s^{-1/2}|I_r - \tilde I_r |^2 \varrho(s-r) \dr \ds \\
     &\le \tfrac{\theta}{2} \int_0^t \gamma(s, C_s) \| f^{\Delta,\varepsilon} \|_2^2 \ds + \tfrac{\alpha^2 K_\sigma^2 K_\gamma K_{\varrho} c_0 }{\theta} \sqrt{t} \int_0^t |I_r - \tilde I_r |^2 \dr,
\end{align*}
where we have used that $\| f^{\nu,\varepsilon}_s \|_2^2$ is of order $O(s^{-1/2})$ by \eqref{eq:density_control_uniqueness}, and we have exchanged the order of integration using Fubini's theorem. Then, one more application of Young's inequality on the second term of the right-hand side of \eqref{eq:integral_sigma} yields
\begin{align}\label{eq:bound_integral_sigma}
|\tilde{\mathcal{I}}^{\sigma}_{\varepsilon}|
	 &\le \tfrac{\alpha^2 K_\sigma^2 K_\gamma K_{\varrho} c_0 }{\theta \kappa_{\gamma}} \sqrt{t} \int_0^t \gamma(s, C_s)|I_s - \tilde I_s |^2 \di s + \int_0^t \gamma(s, C_s) \Big( \tfrac{\theta}{2} \| f^{\Delta, \varepsilon}_s \|_2^2 + \tfrac{1}{2\theta} \| \tilde{c}^{\,\sigma, \varepsilon}_s \|_2^2    \Big) \ds.
\end{align}
Now consider $\tilde{\mathcal{I}}^{b}_{\varepsilon}$. We split it into a $G_{\varepsilon}$ term and a $p_{\varepsilon}$ term, and add and subtract terms to yield
\begin{align*}
	\tilde{\mathcal{I}}^{b}_{\varepsilon} 
    &= \int_0^t \gamma(s, C_s) \int_0^{\infty} \Big( F^{\Delta, \varepsilon}_s(x) f^{\nu, \varepsilon}_s (x) \big( \delta_b(s,x) - \delta_{A'}(s) \big) + F^{\Delta, \varepsilon}_s(x) \big( \tilde{c}^{\,b,\varepsilon}_s(x) + 2\tilde{r}^{\varepsilon}_s(x)   \big) \Big) \dx \ds,
\end{align*}
where we have defined
\begin{align*}
	\tilde{c}^{\,b,\varepsilon}_t(x)
	&:= \big\langle \nu_t, \,  \big( \delta_b(t,\cdot) - \delta_{A'}(t) \big) G_{\varepsilon}(x, \cdot  ) \big\rangle  - f^{\nu, \varepsilon}_t (x)  \big( \delta_b(t,x) - \delta_{A'}(t) \big), \\
	\tilde{r}^{\varepsilon}_t(x)
	&:=  \big\langle \nu_t, \,  \big( \delta_b(t,\cdot) - \delta_{A'}(t) \big) p_{\varepsilon}(x + \cdot  ) \big\rangle .
\end{align*}
Recall that $b$ is Lipschitz by Assumption~\ref{assump:coefficient_assumptions}, and use expression \eqref{eq:A_prime_for_uniqueness} for $A'_t$ to obtain
\begin{align*}
    |\delta_b(s,x)| 
    &\le K_b |A_s - \tilde A_s|
    \le \alpha K_b \Big|\int_{s-\bar{d}}^s \varrho(s-r) (I_r - \tilde{I}_r) \dr \Big| \\
    |\delta_{A'}(s)| &\le \big| A'_s - \tilde A'_s\big| \le \alpha \varrho(\bar{d}) |I_{s - \bar{d}} - \tilde{I}_{s - \bar{d}}| +  \alpha  \int_{0}^{\bar{d}}\big| I_{s-r} - \tilde{I}_{s-r}\big| \di \| \varrho \|_{\mathsc{tv}(0,r)}.
\end{align*}
Applying Young's and Jensen's inequalities, recalling that $I_t = \tilde{I}_t =0$ for $ t\le 0 $, and using the $O(s^{-1/2})$ estimate for $\| f^{\nu,\varepsilon}_s \|_2^2$ and bounded-ness of $\gamma(t, C_t)$ as before, and Fubini's theorem to exchange the order of integration combined with appropriate changes of variables, we compute
\begin{align}
    \int_0^t &\gamma(s, C_s) \int_0^{\infty}  F^{\Delta, \varepsilon}_s(x) f^{\nu, \varepsilon}_s(x) \big( \delta_b(s,x) - \delta_{A'}(s) \big) \dx \ds \nonumber\\
    &\le \int_0^t \gamma(s, C_s) \| F^{\Delta, \varepsilon}_s \|_2^2 \ds + \alpha^2  c_0  \int_0^t s^{-1/2}\gamma(s, C_s) \tfrac{K_{b}^2}{2} \Big(\int_{0}^s \varrho(s-r) (I_r - \tilde{I}_r) \dr \Big)^2 \ds \nonumber \\
    &\quad + \alpha^2  c_0  \int_0^t s^{-1/2}\gamma(s, C_s) \bigg[ \varrho(\bar{d})^2|I_{s-\bar{d}} - \tilde{I}_{s-\bar{d}}|^2  + \Big( \int_{0}^{\bar{d}} \big| I_{s-r} - \tilde{I}_{s-r}\big| \di \| \varrho \|_{\mathsc{tv}(0,r)} \Big)^2 \bigg] \ds \nonumber \\
    &\le \int_0^t \gamma(s, C_s) \| F^{\Delta, \varepsilon}_s \|_2^2 \ds
    + \alpha^2 K_{b}^2 c_0  \tfrac{K_\gamma K_{\varrho}}{\kappa_\gamma} \sqrt{t}  \int_0^t \gamma(s, C_s) |I_s - \tilde{I}_s|^2 \ds \nonumber
     \\
     &\quad + \alpha^2 K_{\gamma} c_0 \bigg[ \varrho(\bar{d})^2 \bar{d}^{-1/2} \int_0^t |I_s - \tilde{I}_s|^2 \ds 
     + \| \varrho \|_{\mathsc{tv}(0,\bar{d})} \int_0^t \hspace{-3pt} \hspace{-3pt} s^{-1/2} \hspace{-3pt} \int_{0}^{\bar{d}} \big| I_{s-r} - \tilde{I}_{s-r}\big|^2  \di \| \varrho \|_{\mathsc{tv}(0,r)} \ds \bigg]. \label{eq:estimate_for_I_tilde_b}
\end{align}
Consider the last double integral in \eqref{eq:estimate_for_I_tilde_b}. A change of variables and Fubini's theorem yield
\begin{gather*}
\int_0^t \hspace{-3pt} \hspace{-3pt} s^{-1/2} \hspace{-3pt} \int_{0}^{\bar{d}} \big| I_{s-r} - \tilde{I}_{s-r}\big|^2  \di \| \varrho \|_{\mathsc{tv}(0,r)} \ds =
    \int_0^t \int_0^{\bar{d} \wedge (t-u)} (u+r)^{-1/2} \di \| \varrho \|_{\mathsc{tv}(0,r)}  \big| I_{u} - \tilde{I}_{u}\big|^2   \di u \\
    \le \int_0^t \int_0^{\bar{d}} (u+r)^{-1/2} \di \| \varrho \|_{\mathsc{tv}(0,r)}  \big| I_{u} - \tilde{I}_{u}\big|^2   \di u.
\end{gather*}
Thus we conclude that
\begin{align}
\tilde{\mathcal{I}}^{b}_{\varepsilon} 
    &\le 
    \tfrac{5}{2}\int_0^t \gamma(s, C_s) \| F^{\Delta, \varepsilon}_s \|_2^2 \ds  + \alpha^2 c_0 \tfrac{K_{\gamma}}{\kappa_{\gamma}} \Big( K_{b}^2  K_{\varrho} \sqrt{t} + \varrho(\bar{d})^2 \bar{d}^{-1/2} \Big)  \int_0^t \gamma(s, C_s) |I_s - \tilde{I}_s|^2 \ds \nonumber
     \\
     &\quad + \alpha^2 K_{\gamma} c_0\| \varrho \|_{\mathsc{tv}(0,\bar{d})} \int_0^t \Phi(s)  \big| I_{s} - \tilde{I}_{s}\big|^2   \di s 
     + \int_0^t \gamma(s, C_s) \Big( \tfrac{1}{2}\| \tilde{c}^{\,b, \varepsilon}_s \|_2^2 +  \| \tilde{r}^{\varepsilon}_s \|_2^2 \Big) \ds, \label{eq:bound_integral_b}
\end{align}
where we have denoted $\Phi(s) := \int_0^{\bar{d}} (s+r)^{-1/2} \di \| \varrho \|_{\mathsc{tv}(0,r)}$ (which is in $L^2$ by Lemma \ref{lem:Phi_squared}).

We are left with having to estimate $\mathcal{I}^{p}_{\varepsilon}$ and $\tilde{\mathcal{I}}^{p}_{\varepsilon}$. We leave the latter as it is for the moment and deal with it after taking the limit as $\varepsilon \downarrow 0$; for the first, recalling that testing the weak formulation \eqref{eq:weak_form} and the boundary conditions \eqref{eq:bc_with_density} with $\phi \equiv - \mathbf{1}$ yields $Z^{\mu, - \mathbf{1}}_t = I_t$ and $Z^{\nu, - \mathbf{1}}_t = \tilde I_t$, we have
\begin{align}
\mathcal{I}^{p}_{\varepsilon} 
&=\int_0^t \int_0^{\infty} F^{\Delta, \varepsilon}_s(x) \, p_{\varepsilon}(x) \Big( \gamma(s, C_s) \, \sigma_A(s, 0)^2  u_s(0) - \gamma(s, \tilde C_s) \,  \sigma_{\tilde A}(s, 0)^2 v_s(0)  \Big) \dx \ds \nonumber \\
&=  \int_0^{\infty} p_{\varepsilon}(x) \int_0^t F^{\Delta, \varepsilon}_s(x)   \di \bigl( Z^{\mu, - \mathbf{1}}_s- Z^{\nu, - \mathbf{1}}_s  \bigr) \dx
=  \int_0^t \int_0^{\infty} p_{\varepsilon}(x) F^{\Delta, \varepsilon}_s(x) \dx  \bigl( I^\prime_s- \tilde{I}^\prime_s  \bigr) \ds. \label{eq:bound_integral_p}
\end{align}
We now get ready to take the limit as $\varepsilon \downarrow 0$ on both sides of the inequality
\begin{equation}\label{eq:uniq_final_ineq}
    \gamma(t,C_t) \|F^{\Delta, \varepsilon}_t\|_2^2 \le
- \mathcal{I}^{\sigma}_{\varepsilon} + |\tilde{\mathcal{I}}^{\sigma}_{\varepsilon}| + 2 \mathcal{I}^{b}_{\varepsilon} + 2 \tilde{\mathcal{I}}^{b}_{\varepsilon} +2 \mathcal{I}^{p}_{\varepsilon} +2 \tilde{\mathcal{I}}^{p}_{\varepsilon} + \mathcal{I}^{\di \gamma}_{\varepsilon}.
\end{equation}
Define $f^{\Delta}_t(x) := u_t(x) - v_t(x)$ and $F^{\Delta}_t(x) := \int_x^{\infty} f_t^{\Delta}(y) \dy$. Note that by continuity of the density $u_t(x)$, we must have that $\lim_{\varepsilon \downarrow 0} f^{\mu, \varepsilon}_t(x) = u_t(x)$ (and equivalently for $v_t(x)$). Thus $f^{\Delta, \varepsilon}_t (x) \to f^{\Delta}_t(x)$ as $\varepsilon \downarrow 0$, and similarly, since $u_t(x)$ and $v_t(x)$ are integrable, dominated convergence yields that $F^{\Delta, \varepsilon}_t (x) \to F^{\Delta}_t(x)$ as $\varepsilon \downarrow 0$. By Lemmas~\ref{lemma:error_terms_decay_linear} and \ref{lemma:error_terms_decay_linear_2} applied to $h_t \in \{ c^{\sigma, \varepsilon}_s,   \tilde{c}^{\,\sigma, \varepsilon}_s, c^{b, \varepsilon}_s, \tilde{c}^{\,b, \varepsilon}_s, r^{\varepsilon}_s, \tilde{r}^{\varepsilon}_s \}$, all commutator and remainder terms of the form $\int_0^t \gamma(s,C_s) \| h_s\|_2^2 \ds$ go to $0$ as $\varepsilon \downarrow 0$. Passing to the limits in the bounds \eqref{eq:bounds_from_linear_uniq}, \eqref{eq:bound_integral_sigma}, \eqref{eq:bound_integral_b}, we thus get
\begin{align*}
    \lim_{\varepsilon \downarrow 0} \big( &- \mathcal{I}^{\sigma}_{\varepsilon} + |\tilde{\mathcal{I}}^{\sigma}_{\varepsilon}| + 2 \mathcal{I}^{b}_{\varepsilon} + 2 \tilde{\mathcal{I}}^{b}_{\varepsilon} + \mathcal{I}^{\di \gamma}_{\varepsilon} \big)
    \le \Big( K_{b}\big(1 + \tfrac{1}{\theta}\big) + 8 + \tfrac{K_{\gamma,C}}{\kappa_\gamma} \Big) \int_{0}^t \gamma(s, C_s) \| F^{\Delta}_s   \|_2^2 \ds \\
    &+ \big((1 + K_b) \theta - \kappa_{\sigma} \big) \int_0^t \gamma(s, C_s)  \| f^{\Delta}_s   \|_2^2 \ds
    + \alpha^2 K_{\gamma} c_0\| \varrho \|_{\mathsc{tv}(0,\bar{d})} \int_0^t \Phi(s)  \big| I_{s} - \tilde{I}_{s}\big|^2   \di s 
     \\
    &+ \alpha^2 c_0 \tfrac{K_{\gamma}}{\kappa_{\gamma}} \Big( K_{\varrho} ( \tfrac{ K_\sigma^2}{\theta} + K_{b}^2  ) \sqrt{t} + \varrho(\bar{d})^2 \bar{d}^{-1/2} \Big)  \int_0^t \gamma(s, C_s) |I_s - \tilde{I}_s|^2 \ds.
\end{align*}
Using \eqref{eq:limit_pe_FDelta} and \eqref{eq:bound_integral_p} we have
\begin{equation*}
    \lim_{\varepsilon \downarrow 0} 2 \mathcal{I}^{p}_{\varepsilon}
    =   \int_0^t (\tilde{I}_s - I_s) \tfrac{\di}{\ds}\big( I_s- \tilde{I}_s  \big) \ds = - \frac{1}{2}\int_0^t \tfrac{\di}{\ds}\big( I_s- \tilde{I}_s  \big)^2 \ds \le 0.
\end{equation*}
Finally, dominated convergence yields
\begin{align*}
    \lim_{\varepsilon \downarrow 0} 2 \tilde{\mathcal{I}}^{p}_{\varepsilon}
    &= \int_0^t \big( \gamma(s, \tilde C_s) - \gamma(s, C_s) \big)  \sigma_{\tilde A}(s, 0)^2 v_s(0) (\tilde{I}_s - I_s) \ds \\
    &\leq 4 \tilde{K}_{\gamma, \varrho, \sigma} \int_0^t \int_{0}^s  |\tilde I_{r} - I_{r}| \dr \, v_s(0) |\tilde{I}_s - I_s| \ds,
\end{align*}
by the Lipschitzness of $\gamma(t, x)$, Assumption~\ref{assump:coefficient_assumptions}, and Jensen's inequality. Bounding $v_s(0)$ by \eqref{eq:density_control_uniqueness}, we can thus integrate by parts and use Jensen's inequality to get
\begin{align*}
    \lim_{\varepsilon \downarrow 0} 2 \tilde{\mathcal{I}}^{p}_{\varepsilon}   
    &\le 2 \tilde{K}_{\gamma, \varrho, \sigma} c_1 \int_0^t \tfrac{\di}{\ds} \Big(\int_{0}^s  |\tilde I_{r} - I_{r}| \dr \Big)^2 s^{-1/2} \ds
    \\
    &\le 2 \tilde{K}_{\gamma, \varrho, \sigma} c_1 \bigg[ \Big(\int_{0}^t |\tilde I_{r} - I_{r}|\dr \Big)^2 t^{-1/2} 
    + \tfrac{1}{2} \int_0^t  \Big(\int_{0}^s |\tilde I_{r} - I_{r}| \dr \Big)^2 s^{-3/2} \ds \bigg]
        \\
    &\le 2 \tilde{K}_{\gamma, \varrho, \sigma} c_1 \bigg[ t^{1/2}\int_{0}^t |\tilde I_{r} - I_{r}|^2 \dr
    + \tfrac{1}{2} \int_0^t  s^{-1/2} \ds \int_{0}^t |\tilde I_{r} - I_{r}|^2 \dr  \bigg]
    \\
    &\le \tfrac{4 \tilde{K}_{\gamma, \varrho, \sigma} c_1 }{\kappa_\gamma} t^{1/2} \int_{0}^t \gamma(r, C_r) |\tilde I_{r} - I_{r}|^2 \dr.
\end{align*}
Then, putting the estimates together and sending $\varepsilon \downarrow 0$ in \eqref{eq:uniq_final_ineq}, we get
\begin{align}\label{eq:final_estimate_uniq}
    \gamma(t,C_t) \|F^{\Delta}_t\|_2^2
    &\le \mathfrak{K}_0  \int_{0}^t  \gamma(s, C_s) \| F^{\Delta}_s   \|_2^2 \ds 
    +  \mathfrak{K}_1 \int_0^t \gamma(s, C_s)  \| f^{\Delta}_s   \|_2^2 \ds \nonumber \\
    &\quad
    + \mathfrak{K}_2 \int_0^t  \gamma(s, C_s) |I_s - \tilde{I}_s|^2 \ds
    + \mathfrak{K}_3 \int_0^t \Phi(s)  \big| I_{s} - \tilde{I}_{s}\big|^2   \ds,
\end{align}
where
\begin{gather*}
    \mathfrak{K}_0 = \mathfrak{K}_0(\theta) = \Big( K_{b}\big(1 + \tfrac{1}{\theta}\big) + 8 + \tfrac{K_{\gamma,C}}{\kappa_\gamma} \Big),
    \quad
    \mathfrak{K}_1 = \mathfrak{K}_1(\theta) = \big((1 + K_b) \theta - \kappa_{\sigma} \big), \\
    \mathfrak{K}_2 = \mathfrak{K}_2(\theta) = \tfrac{4 \tilde{K}_{\gamma, \varrho, \sigma} c_1 }{\kappa_\gamma} \sqrt{t} + \alpha^2 c_0 \tfrac{K_{\gamma}}{\kappa_{\gamma}} \Big( K_{\varrho} ( \tfrac{ K_\sigma^2}{\theta} + K_{b}^2  ) \sqrt{t} + \varrho(\bar{d})^2 \bar{d}^{-1/2} \Big), 
    \quad
    \mathfrak{K}_3 = \alpha^2 K_{\gamma} c_0\| \varrho \|_{\mathsc{tv}(0,\bar{d})}.
\end{gather*}
We now estimate the term $\big| I_{s} - \tilde{I}_{s}\big|^2$. For each $t>0$, let $\delta(t) > 0$ and take a smooth function $x\mapsto \chi(t,x)$ such that $\chi \equiv 0$ for $0 \le x \le \delta(t)/2 $ and $\chi \equiv 1$ on $ x \ge \delta(t)$, with $\partial_x\chi(t,x) \le \kappa / \delta(t)$ for a fixed constant $\kappa>0$.  Then, $\| \partial_x\chi \|_2^2\leq \kappa^2/2\delta(t)$, so we can observe that
\begin{align}
    |I_t - \tilde{I}_t|^2 
    &\le 2|\langle \Delta_t, \chi \rangle |^2 +  2|\langle \Delta_t, \mathbf{1}_{[0, \delta(t)]} \rangle |^2
    \le  2\Big| \int_{0}^{\infty} \chi(t,x) f^{\Delta}_t(x) \dx \Big|^2 + 2 \Big| \int_{0}^{\delta(t)} \!f^{\Delta}_t(x) \dx \Big|^2 \nonumber \\
    &\le 2 \Big| - \int_{0}^{\infty} \partial_x\chi(t,x) \int_x^{\infty} f^{\Delta}_t(y) \dy \dx \Big|^2 + 2 \delta(t) \| f_t^{\Delta} \|_2^2 \nonumber \\
    &\le 2\| \partial_x\chi \|_2^2 \| F^{\Delta}_t \|_2^2 +2 \delta(t) \| f_t^{\Delta} \|_2^2
    \le \tfrac{\kappa^2}{ \delta(t)} \| F^{\Delta}_t \|_2^2 +2 \delta(t)\| f_t^{\Delta} \|_2^2. \label{eq:diff_I_delta_bound}
\end{align}
Going back to \eqref{eq:final_estimate_uniq}, we first take $0 < \theta < \kappa_s/(1+K_b)$, so $\mathfrak{K}_2 < 0$. Then, with such a fixed $\theta$, set
\begin{equation*}
    \delta(t) := \frac{|\mathfrak{K}_1| \kappa_{\gamma}}{4 (\mathfrak{K}_2 K_{\gamma} + \mathfrak{K}_3 \Phi(t))},
\end{equation*}
which is positive and bounded. Using \eqref{eq:diff_I_delta_bound} with $\delta(t)$ as above in the last two terms on the right-hand side of \eqref{eq:final_estimate_uniq} yields
\begin{gather*}
    \int_0^t \big( \mathfrak{K}_2  \gamma(s, C_s) + \mathfrak{K}_3 \Phi(s)\big) |I_s - \tilde{I}_s|^2 \ds\\
    \le \frac{4 \kappa^2 }{|\mathfrak{K}_1|\kappa_{\gamma}}\int_0^t \big(\mathfrak{K}_2 K_{\gamma} + \mathfrak{K}_3 \Phi(s)\big)^2 \| F^{\Delta}_s \|^2_2 \ds
    + \frac{|\mathfrak{K}_1|\kappa_{\gamma}}{2} \int_0^t \| f^{\Delta}_s \|^2_2 \ds
    \\
    \le \frac{4 \kappa^2 }{|\mathfrak{K}_1| \kappa_{\gamma}^2}\int_0^t \big(\mathfrak{K}_2 K_{\gamma} + \mathfrak{K}_3 \Phi(s)\big)^2 \gamma(s, C_s) \| F^{\Delta}_s \|^2_2 \ds
    + \frac{|\mathfrak{K}_1|}{2} \int_0^t \gamma(s, C_s)  \|f^{\Delta}_s \|^2_2 \ds,
\end{gather*}
where we have also used that $\gamma(s,C_s)/\kappa_{\gamma} \ge 1$.
Thus \eqref{eq:final_estimate_uniq} reduces to
\begin{align*}
    \gamma(t,C_t) \|F^{\Delta}_t\|_2^2
    &\le \int_{0}^t  \Big( \mathfrak{K}_0 + \frac{4 \kappa^2 }{|\mathfrak{K}_1| \kappa_{\gamma}^2}\big(\mathfrak{K}_2 K_{\gamma} + \mathfrak{K}_3 \Phi(s)\big)^2 \Big) \gamma(s, C_s) \| F^{\Delta}_s   \|_2^2 \ds \\
    &\quad +  \int_0^t \Big( \mathfrak{K}_1 + \frac{|\mathfrak{K}_1|}{2} \Big)  \gamma(s, C_s) \| f^{\Delta}_s \|_2^2 \ds. 
\end{align*}
Recalling $\mathfrak{K}_1 <0$ by our choice of $\theta$, the last term on the right-hand side is negative and we can drop it. Since $\Phi$ is square-integrable by Lemma~\ref{lem:Phi_squared} below, a Gr\"onwall argument concludes the proof.
\end{proof}

\begin{lemma}\label{lem:Phi_squared}
Under Assumption~\ref{assump:coefficient_assumptions}, the function $\Phi(s) = \int_0^{\bar{d}} (s + r)^{-1/2} \di \Vrho{r}$ is in $L^2([0,T])$.
\end{lemma}

\begin{proof}
Stieltjes integration by parts and using $\Vrho{0+} = 0$ yields
\begin{equation}\label{eq:Phi_IBP}
\Phi(s) = \frac{\Vrho{\bar d}}{\sqrt{s + \bar d}} + \frac{1}{2} \int_0^{\bar d} \frac{\Vrho{r}}{(s+r)^{3/2}} \dr.
\end{equation}
Squaring with $(a+b)^2 \le 2a^2 + 2b^2$, and integrating in $s$ on $[0,T]$ gives
\begin{equation*}
\Phi(s)^2 \le 2\,\Vrho{\bar d}^2 \log\big(\tfrac{T+\bar d}{\bar d}\big) + \frac{1}{2}\int_0^T \Big(\int_0^{\bar d} \frac{\Vrho{r}}{(s+r)^{3/2}} \dr\Big)^{\!2} \ds.
\end{equation*}
For the second term on the right-hand side, apply Minkowski's integral inequality to get
\begin{equation*}
\bigg( \int_0^T \Big( \int_0^{\bar d} \frac{\Vrho{r}}{(s+r)^{3/2}} \dr \Big)^{ 2} \ds \bigg)^{ 1/2}
\le  \int_0^{\bar d} \bigg( \int_0^T \frac{\Vrho{r}^2}{(s+r)^3} \ds \bigg)^{ 1/2} \dr
\le  \int_0^{\bar d} \frac{\Vrho{r}}{r \sqrt{2}} \dr,
\end{equation*}
where we computed the inner integral yielding $\tfrac{1}{2}\big(r^{-2} - (T+r)^{-2}\big) \le \tfrac{1}{2}r^{-2}$. Since $\Vrho{r}/r $ is integrable by Assumption~\ref{assump:coefficient_assumptions}, we are done.
\end{proof}

\subsection{Proof of Theorems \ref{thm:free_boundary_problem} and \ref{thm:summary_contributions}}\label{subsect:proof_thms2.2-2.3}

By Theorem \ref{thm:uniqueness_linear}, any limit point $(\mu^*,I^*,Z^*)$ from Section \ref{sect:weak_convergence} agrees with the corresponding $\mathcal{F}^{\mu^*}_t$-adapted measure flow $\nu^*$ from Proposition \ref{prop:linearised_problem_nu}. We can observe that $\nu^*$ is supported on the set of weak solutions in the sense of Definition \ref{def:clas_weak_formulation}. However, Theorem \ref{thm:gen_uniqueness} gives that this set is a singleton, so we conclude that $(\mu^n,A^n)$ converges weakly to a deterministic limit $(\mu,A)$, which is the unique weak solution in the sense of Definition \ref{def:clas_weak_formulation}, for any $T>0$. Moreover, Proposition \ref{prop:linearised_problem_nu} also gives that $\mu$ has a density $u_s(x)$ with the required regularity properties so that $u(s,x):=u_s(x-A_s)$ yields a weak solution in the sense of Definition \ref{def:weak_formulation} with the desired Aronson estimate.

It remains to confirm that the uniqueness from Theorem \ref{thm:gen_uniqueness} extends to weak solutions in the sense of Definition \ref{def:weak_formulation}. Then, the above concludes the proof of Theorems \ref{thm:free_boundary_problem} and \ref{thm:summary_contributions}. For any given weak solution $(v,A)$ in the sense of Definition \ref{def:weak_formulation}, we set $\bar{\nu}_0:=P_0\circ \theta_{a_0}^{-1}$ and $\bar{\nu}_t:=v_t(x)\dx$ for $t>0$, where $v_t(x):=v(t,x+A_t)$. Then, $\bar{\nu}$ satisfies \eqref{eq:weak_form}--\eqref{eq:ACI_defn_weak_sol} of Definition \ref{def:clas_weak_formulation}, for any $T>0$. However, we cannot invoke Theorem \ref{thm:gen_uniqueness} directly, as we do not have the control \eqref{eq:density_control_uniqueness} a priori. Nevertheless, a few minor observations confirm that there is indeed uniqueness.

First of all, notice that we have   $\int_{0}^\infty x  \hspace{1pt}\di \bar{\nu}_t(x) <\infty$ and hence Lemma \ref{lemma:F_decay} applies. Indeed, by taking a suitable smooth (concave) test function $\phi_n(x)$ which approximates $x\mapsto x\land n$, we can perform a Gr\"onwall argument in \eqref{eq:the_weak_formulation} for $\int_{A_t}^\infty \phi_n(x)v(t,x) \dx$ to bound it in terms of a factor $e^{ct}$ and  $\int_{a_0}^\infty x \hspace{0.5pt}\di P_0(x)<\infty$, and so monotone convergence gives $\int_{A_t}^\infty x v(t,x) \dx<\infty$, as desired.

There are no issues with Lemmas \ref{lemma:error_terms_decay_linear} and \ref{lemma:square_of_I_diff}, but Lemma \ref{lemma:error_terms_decay_linear_2}, need not hold without additional control. Nevertheless, the proof of Theorem \ref{thm:uniqueness_linear} still goes through for $\Delta_t:=\bar{\nu}_t-\tilde{\nu}_t$, where we let $\tilde{\nu}$ denote the (deterministic) decoupled weak solution defined from $I^{\bar{\nu}}_t=1-\bar{\nu}_t(\R^+)$ as in Proposition \ref{prop:linearised_problem_nu}. Indeed, we can notice from \eqref{eq:unique_bound_L1-remainder} along with $|F^{\Delta, \varepsilon}_s(x)|\leq 2$ that it suffices to have $\int_0^t \Vert r_s^\varepsilon \Vert_1 \ds \rightarrow 0$ as $\varepsilon \downarrow 0$. By \eqref{eq:remainder_pointwise_bound} in the proof of Lemma \ref{lemma:error_terms_decay_linear_2}, there is a $c>0$ such that
\begin{equation}\label{eq:remainder_L1_decay}
	\Vert r^\varepsilon_t \Vert_1 \leq c \sqrt{\varepsilon} \langle |\Delta_t|, p_{2\varepsilon}\rangle,
\end{equation}
for all $t\in[0,T]$ and $\varepsilon>0$. Since $|\Delta_t|(0,\infty)\leq 2$, we thus have $\Vert r^\varepsilon_t \Vert_1 \leq c /\sqrt{\pi}$ for all $t\in[0,T]$ and $\varepsilon>0$. Moreover, by the continuity of densities $v$ and $\tilde{v}$, Lebesgue's differentiation theorem gives
\[
\langle |\Delta_s|, p_{2\varepsilon}(y) \rangle \rightarrow (v_s(0)+\tilde{v}_s(0))/2\quad \text{as}\quad \varepsilon \downarrow 0,
\]
for all $s\in(0,T]$, so \eqref{eq:remainder_L1_decay} and dominated convergence confirms that the remainder $r^\varepsilon$ vanishes in the required way. Consequently, the estimates in Theorem \ref{thm:uniqueness_linear} give $\bar{\nu}=\tilde{\nu}$ and hence $\bar{\nu}$ satisfies \eqref{eq:density_control_uniqueness}. It follows that Theorem \ref{thm:gen_uniqueness} applies to $\bar{\nu}$, so there is indeed uniqueness of weak solutions in the sense of Definition \ref{def:weak_formulation}. This concludes the proof of Theorems \ref{thm:free_boundary_problem} and \ref{thm:summary_contributions}.

%%% bibliography

\bibliographystyle{alpha}
\bibliography{biblio_epidemic}

\begin{thebibliography}{BDWO23}

\bibitem[Bar20]{barnes}
C.~L. Barnes.
\newblock Hydrodynamic limit and propagation of chaos for {B}rownian particles reflecting from a {N}ewtonian barrier.
\newblock {\em Ann. Appl. Probab.}, 30(4):1582--1613, 2020.

\bibitem[BB20]{banerjee_brown}
S.~Banerjee and B.~Brown.
\newblock Inert drift system in a viscous fluid: steady state asymptotics and exponential ergodicity.
\newblock {\em Trans. Amer. Math. Soc.}, 373(9):6369--6409, 2020.

\bibitem[BBD19]{banerjee_burdzy_duarte}
S.~Banerjee, K.~Burdzy, and M.~Duarte.
\newblock Gravitation versus {B}rownian motion.
\newblock {\em Ann. Inst. Henri Poincar\'{e} Probab. Stat.}, 55(3):1531--1565, 2019.

\bibitem[BBE23]{banerjee_etal}
S.~Banerjee, A.~Budhiraja, and B.~Estevez.
\newblock The inert drift atlas model.
\newblock {\em Comm. Math. Phys.}, 399(3):2083--2147, 2023.

\bibitem[BCS04]{Burdzy_et_al}
K.~Burdzy, Z.-Q. Chen, and J.~Sylvester.
\newblock The heat equation and reflected {B}rownian motion in time-dependent domains.
\newblock {\em Ann. Probab.}, 32(1B):775--804, 2004.

\bibitem[BDWO23]{berestycki_desjardins_weitz_oury}
H.~Berestycki, B.~Desjardins, J.~S. Weitz, and J.-M. Oury.
\newblock Epidemic modeling with heterogeneity and social diffusion.
\newblock {\em J. Math. Biol.}, 86(4):Paper No. 60, 59, 2023.

\bibitem[BM19]{banerjee_etal2}
S.~Banerjee and D.~Mukherjee.
\newblock Join-the-shortest queue diffusion limit in {H}alfin--{W}hitt regime: tail asymptotics and scaling of extrema.
\newblock {\em Ann. Appl. Probab.}, 29(2):1262--1309, 2019.

\bibitem[BN02]{burdzy_nualart}
K.~Burdzy and D.~Nualart.
\newblock {B}rownian motion reflected on {B}rownian motion.
\newblock {\em Probab. Theory Related Fields}, 122(4):471--493, 2002.

\bibitem[Bre22a]{Bressloff2022}
P.~C. Bressloff.
\newblock Diffusion-mediated absorption by partially reactive targets: {B}rownian functionals and generalized propagators.
\newblock {\em J. Phys. A: Math. Theor.}, 55:205001, 2022.

\bibitem[Bre22b]{Bressloff2022spectral}
P.~C. Bressloff.
\newblock Spectral theory of diffusion in partially absorbing media.
\newblock {\em Proc. R. Soc. A}, 478:20220319, 2022.

\bibitem[Bre24]{Bressloff2024}
P.~C. Bressloff.
\newblock A generalized {D}ean--{K}awasaki equation for an interacting {B}rownian gas in a partially absorbing medium.
\newblock {\em Proc. R. Soc. A}, 480:20230915, 2024.

\bibitem[BS22]{baker_shkolnikov_undercooling}
G.~Baker and M.~Shkolnikov.
\newblock Zero kinetic undercooling limit in the supercooled {S}tefan problem.
\newblock {\em Ann. Inst. Henri Poincar\'{e} Probab. Stat.}, 58(2):861--871, 2022.

\bibitem[BY82]{barlow_yor}
M.~T. Barlow and M.~Yor.
\newblock Semimartingale inequalities via the {G}arsia--{R}odemich--{R}umsey lemma, and applications to local times.
\newblock {\em J. Funct. Anal.}, 49(2):198--229, 1982.

\bibitem[DPdH96]{Dai_Pra}
P.~Dai~Pra and F.~den Hollander.
\newblock {M}c{K}ean--{V}lasov limit for interacting random processes in random media.
\newblock {\em J. Stat. Phys.}, 84(3):735--772, 1996.

\bibitem[FS22]{fausti_soj_22}
E.~Fausti and A.~S{\o}jmark.
\newblock An interacting particle system for the front of an epidemic advancing through a susceptible population.
\newblock Preprint, arXiv:2210.09286, 2022.

\bibitem[Gre20]{Grebenkov2020}
D.~S. Grebenkov.
\newblock Paradigm shift in diffusion-mediated surface phenomena.
\newblock {\em Phys. Rev. Lett.}, 125:078102, 2020.

\bibitem[Gre22]{Grebenkov2022}
D.~S. Grebenkov.
\newblock An encounter-based approach for restricted diffusion with a gradient drift.
\newblock {\em J. Phys. A: Math. Theor.}, 55:045203, 2022.

\bibitem[Had16]{Hadeler}
K.~P. Hadeler.
\newblock {S}tefan problem, traveling fronts, and epidemic spread.
\newblock {\em Discrete Contin. Dyn. Syst. Ser. B}, 21(2):417--436, 2016.

\bibitem[HM22]{hambly_meier}
B.~Hambly and J.~Meier.
\newblock {M}c{K}ean--{V}lasov equations with positive feedback through elastic stopping times.
\newblock {\em Electron. Commun. Probab.}, 27:Paper No. 41, 13, 2022.

\bibitem[HS19]{HamblySojmark19}
B.~Hambly and A.~S{\o}jmark.
\newblock An {SPDE} model for systemic risk with endogenous contagion.
\newblock {\em Finance Stoch.}, 23(3):535--594, 2019.

\bibitem[It{\^o}83]{Ito_S_prime}
K.~It{\^o}.
\newblock Distribution-valued processes arising from independent {B}rownian motions.
\newblock {\em Math. Z.}, 182(1):17--33, 1983.

\bibitem[Kni01]{Knight}
F.~B. Knight.
\newblock On the path of an inert object impinged on one side by a {B}rownian particle.
\newblock {\em Probab. Theory Related Fields}, 121(4):577--598, 2001.

\bibitem[Led16]{ledger_M1}
S.~Ledger.
\newblock {S}korokhod's {$M_1$} topology for distribution-valued processes.
\newblock {\em Electron. Commun. Probab.}, 21(34):1--11, 2016.

\bibitem[McK75]{mckean_local}
H.~P. McKean.
\newblock {B}rownian local times.
\newblock {\em Adv. Math.}, 16:91--111, 1975.

\bibitem[Mit83]{Mitoma_S_prime}
I.~Mitoma.
\newblock Tightness of probabilities on {$C([0,1];\mathcal{S}')$} and {$D([0,1];\mathcal{S}')$}.
\newblock {\em Ann. Probab.}, 11(4):989--999, 1983.

\bibitem[Pil14]{reflected_SDEs_book}
A.~Pilipenko.
\newblock {\em An introduction to stochastic differential equations with reflection}.
\newblock Lectures in Pure and Applied Mathematics. Potsdam University Press, 2014.

\bibitem[PP22]{pang_pardoux_2022}
G.~Pang and {\'E}.~Pardoux.
\newblock Functional limit theorems for non-{M}arkovian epidemic models.
\newblock {\em Ann. Appl. Probab.}, 32(3):1615--1665, 2022.

\bibitem[QRZ03]{Zhongmin2004}
Z.~Qian, F.~Russo, and W.~Zheng.
\newblock Comparison theorem and estimates for transition probability densities of diffusion processes.
\newblock {\em Probab. Theory Related Fields}, 127(3):388--406, 2003.

\bibitem[QX23]{Zhongmin2023}
Z.~Qian and X.~Xu.
\newblock Probability bounds for reflecting diffusion processes.
\newblock {\em Statist. Probab. Lett.}, 199:109855, 2023.

\bibitem[QZ04]{QianZheng2004}
Z.~Qian and W.~Zheng.
\newblock A representation formula for transition probability densities of diffusions and applications.
\newblock {\em Stochastic Process. Appl.}, 111(1):57--76, 2004.

\bibitem[RY99]{revuz_yor}
D.~Revuz and M.~Yor.
\newblock {\em Continuous martingales and {B}rownian motion}, volume 293 of {\em Grundlehren der mathematischen Wissenschaften}.
\newblock Springer-Verlag, Berlin, third edition, 1999.

\bibitem[SW26]{sojmark_wunderlich}
A.~S{\o}jmark and F.~Wunderlich.
\newblock Weak convergence of stochastic integrals on {S}korokhod space in {S}korokhod's {$J_1$} and {$M_1$} topologies.
\newblock {\em Probab. Theory Related Fields}, 2026.
\newblock https://doi.org/10.1007/s00440-026-01476-y.

\bibitem[WD21]{Wang_Du_2021}
R.~Wang and Y.~Du.
\newblock Long-time dynamics of a diffusive epidemic model with free boundaries.
\newblock {\em Discrete Contin. Dyn. Syst. Ser. B}, 26(4):2201--2238, 2021.

\bibitem[WD22]{Wang_Du_2022}
R.~Wang and Y.~Du.
\newblock Long-time dynamics of a nonlocal epidemic model with free boundaries: {S}preading-vanishing dichotomy.
\newblock {\em J. Differential Equations}, 327:322--381, 2022.

\bibitem[Whi07]{White}
D.~White.
\newblock Processes with inert drift.
\newblock {\em Electron. J. Probab.}, 12:1509--1546, 2007.

\bibitem[WND19]{Wang_etal_2019}
Z.~Wang, H.~Nie, and Y.~Du.
\newblock Spreading speed for a {W}est {N}ile virus model with free boundary.
\newblock {\em J. Math. Biol.}, 79(2):433--466, 2019.

\end{thebibliography}

\end{document}